\documentclass[12pt]{article}

\usepackage{natbib,url}
\usepackage{graphicx, amsmath, amsthm, amssymb,float,setspace,color,mathtools,bm}
\usepackage{appendix}
\usepackage{xcolor}

\usepackage{xr}
 \DeclareRobustCommand{\legendsquare}[1]{
  \textcolor{#1}{\rule{1ex}{1ex}}
}
 \definecolor{darkgreen}{rgb}{0, 0.5, 0}

\usepackage{caption}
\usepackage{subcaption}
\usepackage{pdfpages}
\usepackage[flushleft]{threeparttable}

\usepackage[a4paper, margin=	0.8 in]{geometry}

\newtheorem{theorem}{Theorem}
\newtheorem{lemma}{Lemma}
\newtheorem{prop}{Proposition}
\newtheorem{cor}{Corollary}

\theoremstyle{definition}

\def\vbeta{\bm{\beta}}
\def\vd{\bm{d}}
\def\mR{\mathbb{R}}

\DeclareMathOperator{\diag}{diag}
\DeclareMathOperator{\Var}{Var}
\DeclareMathOperator{\Cov}{Cov}
\DeclareMathOperator{\argmin}{argmin}
\DeclareMathOperator{\supp}{supp}

\title{Analysis of Networks via the Sparse $\beta$-Model\footnote{First arXiv version: August 8, 2019. This version: \today. We thank Prof. David Dunson, Prof. Aurore Delaigle, an Associate Editor and two referees for their constructive comments that have led to a much improved paper. Chen is partially supported by a Turing-HSBC-ONS Economic Data Science Award. Kato is partially supported by NSF grants DMS-1952306 and DMS-2014636. Leng's research is partially supported by a Turing Fellowship.}}
\author{Mingli Chen
\footnote{Department of Economics, University of Warwick, Coventry, CV4 7AL, UK. Email: m.chen.3@warwick.ac.uk
} \and
Kengo Kato
\footnote{Department of Statistics and Data Science, Cornell University, 
1194 Comstock Hall, Ithaca, NY 14853. Email: kk976@cornell.edu
} \and
Chenlei Leng
\footnote{Department of Statistics, University of Warwick, Coventry, CV4 7AL, UK. Email: C.Leng@warwick.ac.uk
}}
\date{}
\begin{document}
\maketitle
  
\begin{abstract}
Data in the form of networks are increasingly available in a variety of areas, yet statistical models allowing for parameter estimates with desirable statistical properties for sparse networks remain scarce. To address this, we propose the Sparse $\beta$-Model (S$\beta$M), a new network model that interpolates the celebrated Erd\H{o}s-R\'enyi model and the $\beta$-model that assigns one different parameter to each node. By a novel reparameterization of the $\beta$-model to distinguish global and local parameters, our S$\beta$M can drastically reduce the dimensionality of the $\beta$-model by requiring some of the local parameters to be zero. We derive the asymptotic distribution of the maximum likelihood estimator of the S$\beta$M when the support of the parameter vector is known. When the support is unknown, we formulate a penalized likelihood approach with the $\ell_0$-penalty. Remarkably, we show via a monotonicity lemma that the seemingly combinatorial computational problem due to the $\ell_0$-penalty can be overcome by assigning nonzero parameters to those nodes with the largest degrees. We further show that a $\beta$-min condition guarantees our method to identify the true model and provide excess risk bounds for the estimated parameters. The estimation procedure enjoys good finite sample properties as shown by simulation studies. The usefulness of the S$\beta$M is further illustrated via the analysis of a microfinance take-up example.
\end{abstract}
\noindent


\section{Introduction}
Complex datasets involving multiple units that interact with each other are best represented by networks where nodes correspond to units and edges to interactions. Thanks to the rapid development of measurement and information technology, data in the form of networks are becoming increasingly available in a wide variety of areas including science, health, economics, engineering, and sociology \citep{jackson2010social,Barabasi:2016,depaula:2017,Newman:2018}. 
{Observed networks tend to be sparse}, namely having much fewer edges than the maximum possible numbers of links allowed, and exhibits various degrees of heterogeneity.
One of the major goals of analysis of networks is to understand the generative mechanism of the interconnections among the nodes in  such networks using statistical models. We refer to \cite{Goldenberg:etal:2009} and \cite{Fienberg:2012}  for reviews,  \cite{Kolaczyk:2009} for a comprehensive treatment, and  \cite{Kolaczyk:2017} for foundational issues and emerging challenges. More recent developments on statistical modeling of networks can be found in \cite{li2020network}, \cite{schweinberger2017exponential},  \cite{stewart2020scalable} and references therein.
The study of various statistical properties of a network model is usually conducted by allowing the number of nodes $n$ to go to infinity. 

The earliest, simplest and perhaps the most studied network model is the Erd\H{o}s-R\'enyi model  \citep{erdds1959random,erdos1960evolution,gilbert1959random} where connections between pairs of nodes independently occur  with the same probability $p$. The resulting distribution of the degree of any node is Poisson for large $n$ if $np$ equals a constant. Probabilistically, the simplicity of the Erd\H{o}s-R\'enyi model  has permitted the development of many insights on networks as a mathematical object such as the existence of giant components and phase transition. The Erd\H{o}s-R\'enyi model is also attractive from a theoretical perspective as discussed in  Section \ref{sec:SbetaM}. In particular, the maximum likelihood estimator (MLE) of its parameter is consistent and asymptotically normal for both dense and sparse networks. By a sparse network, we mean that its number of edges scales sub-quadratically with the number of nodes. Similar phenomena are discussed for a closely related model  for directed networks in  \cite{Krivitsky:Kolaczyk:2015} that study the fundamental issue of the effective sample size of a network model. Despite its theoretical attractiveness, however, the Erd\H{o}s-R\'enyi model is not suitable for modeling real networks whose empirical degree distributions are often heavy-tailed  because it tends to produce degree distributions similar to Poisson \citep{clauset2009power, Newman:2018}. 
We refer further to \cite{Caron:Fox:2017} for related discussion and a novel attempt in using exchangeable random measures to model sparse networks. 

In practice, many real network exhibits a certain level of degree heterogeneity, usually having few high degree ``core" nodes with many edges and many low degree individuals with few links \citep{clauset2009power,Newman:2018}. 
Many statistical models have been developed to directly account for degree heterogeneity. Two prominent examples are the stochastic block model  and the $\beta$-model. The former aims to capture degree heterogeneity by clustering nodes into communities with similar connection patterns \citep{Holland:etal:1983,Wang:Wong:1987,Bickel:Chen:2009,Abbe:2018}, {sometimes after adjusting the propensity of each node in participating in network activities \citep{karrer2011stochastic}}. The latter explicitly models degree heterogeneity by using node-specific parameters  \citep{Britton:etal:2006,Chatterjee:etal:2011}. The $\beta$-model can be seen as a generalization of the Erd\H{o}s-R\'enyi model where the probability that two nodes are connected  depends on the corresponding two node parameters.  
It is one of the simplest exponential random graph models \citep{Robins:etal:2007} and a special case of the $p_1$ model  \citep{Holland:Leinhardt:1981}.  
 The recent work of \cite{Mukherjee:etal:2019} studies  sharp thresholds for detecting sparse signals in the $\beta$-model from a hypothesis testing perspective. 

Statistically, however, the $\beta$-model has a limitation when sparse networks are considered. 
Namely, until now,  the MLE of the $\beta$-model parameters is known to be consistent and asymptotically normal only for relatively dense networks \citep{Chatterjee:etal:2011,Yan:Xu:2013}. We refer also to \cite{Rinaldo:etal:2013} and  \cite{Karwa:Slavkovic:2016} for further results concerning the MLE for the $\beta$-model, and \cite{Yan:etal:2016} for similar results on  the MLE of the parameters in the $p_1$ model. The gap between the need for modeling sparse networks that are commonly seen in practice and the theoretical guarantees of the $\beta$-model  that are available for much denser networks thus necessitates the development of new models.

In this paper, we propose a new network model which we call the \textit{Sparse $\beta$-Model} (abbreviated as S$\beta$M) that can capture node heterogeneity and at the same time allows parameter estimates with desirable statistical properties under sparse network regimes, thereby complementing the Erd\H{o}s-R\'enyi and $\beta$-models. Specifically, the S$\beta$M is defined by a novel reparameterization of the $\beta$-model to distinguish parameters  characterizing global sparsity and local {density} of the network. Using a cardinality constraint on the local parameters, the S$\beta$M can effectively interpolate the  Erd\H{o}s-R\'enyi  and $\beta$-models with a continuum of intermediate models while reducing the dimensionality of the latter. Before proceeding further, we emphasize that the word ``sparse'' in S$\beta$M refers to the sparsity of the parameters as often used in high-dimensional statistics, in the sense that many parameters in the S$\beta$M are assumed irrelevant. {We remark that this notion of parameter sparsity should not be confused with network sparsity and it will become clear which sparsity we refer to from the context. Intuitively, a reduction in the number of parameters will enable us to model networks that are sparse. For example, in the extreme case where only one parameter is present, we will show in Section 2 that a notion of statistical inference is possible, as long as the expected total number of nodes goes to infinity.}

We study several statistical properties of the S$\beta$M in the asymptotic setting where the number of nodes tends to infinity. 
We first study parameter estimation in the S$\beta$M. We derive the asymptotic distribution of the maximum likelihood estimator  when the support of the parameter vector is known.
Although this result should be considered as a theoretical benchmark, it leads to the following important properties of the S$\beta$M: 1) the MLE of the parameters in the S$\beta$M can achieve consistency and asymptotic normality under sparse network regimes, and 2) the S$\beta$M can also capture the heterogeneous patterns for the individual nodes. Next, we consider a more practically relevant case where the support is unknown and formulate a penalized likelihood approach with the $\ell_0$-penalty. Remarkably, we show via a monotonicity lemma that the seemingly combinatorial computational problem due to the $\ell_0$-penalty can be overcome by assigning nonzero parameters to those nodes with the largest degrees. We show further that a $\beta$-min condition guarantees our method to identify the true model with high probability and derive excess risk bounds for the estimated parameters. In particular, we show that the $\ell_{0}$-penalized MLE is \textit{persistent} in the sense of \cite{GreenshteinRitov2004} for (dense and) sparse networks under mild regularity conditions. The simulation study confirms that the $\ell_{0}$-penalized MLE with its sparsity level selected by Bayesian Information Criterion (BIC) works well in the finite sample, both in terms of model selection and parameter estimation. 

Our development of the S$\beta$M is practically motivated by the microfinance take-up dataset of 43 rural Indian villages in \cite{Banerjee:etal:2013}. A detailed description of this dataset can be found in Section \ref{Sec:DA}. In Figure \ref{fig:1}, for illustration,  we plotted a sub-network of the dataset corresponding to one of the villages (Village 60) with  $356$ nodes  as well as their empirical degree distribution (the number $356$ is the sample size of Village 60; the total sample size of all 43 villages combined is 9598). The average degree is $7.98$, the maximum degree is $39$, and there are 15 nodes with no connections at all.  From the  left plot, we can see that there are few nodes with many edges and many peripheral nodes with few connections.  The right plot presents the empirical degree distribution on the log-log scale. It is seen that the empirical distribution of the node degrees is heavy tailed.

\begin{figure}[hbt!]
\centering
\begin{subfigure}[c]{.4\textwidth}
\includegraphics[width=\textwidth]{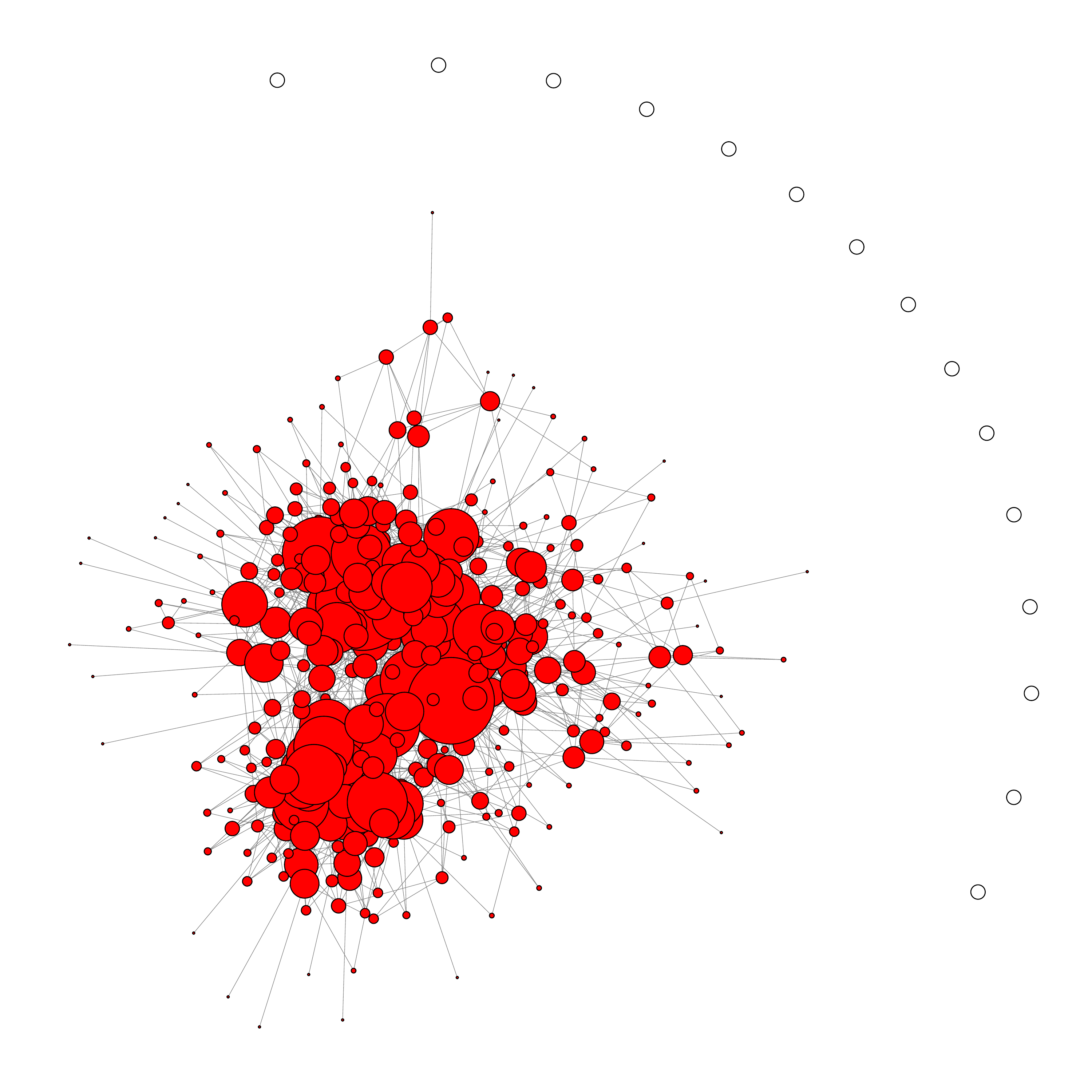} 
\end{subfigure}
\begin{subfigure}[c]{.5\textwidth}
\includegraphics[width=\textwidth]{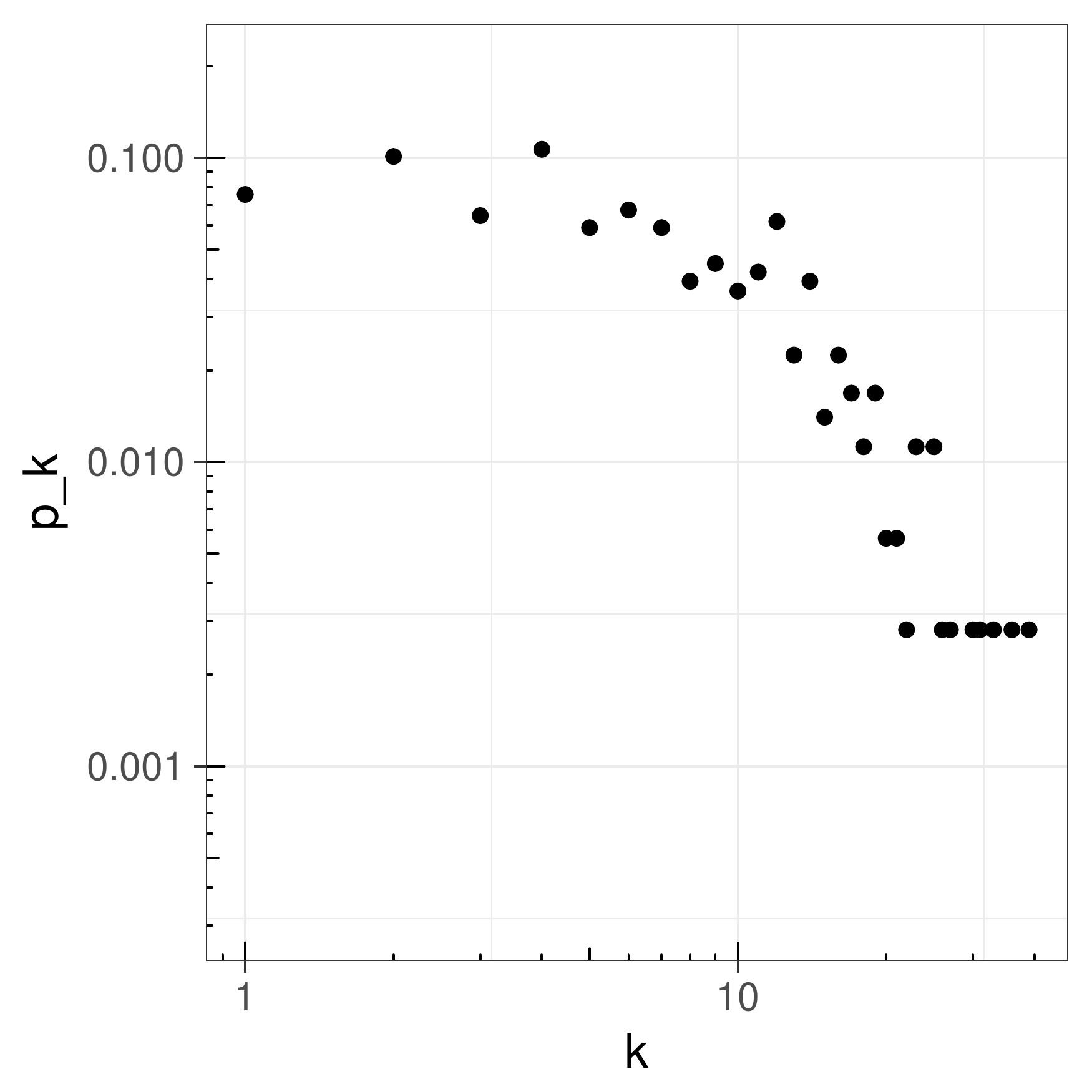}
\end{subfigure}
\caption{\label{fig:1} Left: The network of Village 60. The size of each node is proportional to its degree. Right: The empirical distribution of the node degrees (frequency of degree denoted as $p_k$ versus degree $k$) on the log-log scale.}
\end{figure}

The network structure in Figure \ref{fig:1} depicts features in so-called core-periphery or leaders-followers networks commonly seen in financial economics, due to the presence of one group of core nodes and another group of peripheral nodes. For example, over-the-counter markets for financial assets are dominated by a relatively small number of core intermediaries and a large number of peripheral customers. The core nodes are densely connected with each other and to the peripheral nodes, while  the peripheral nodes are typically only connected to the core nodes but not to each other. This structure has important policy implications. For example, small shocks to those core/hub/leading players will affect the entire network \citep{Acemoglu:etal:2012} because of their roles in facilitating diffusion \citep{Banerjee:etal:2013}. It is thus natural to associate those important core nodes with their individual parameters while leaving the less important peripheral nodes as background nodes without associated parameters.    The S$\beta$M is a model for doing this.

The rest of the paper is organized as follows. 
In Section \ref{sec:SbetaM}, we define the S$\beta$M, establish its connection to the Erd\H{o}s-R\'enyi and $\beta$-models, and discuss its properties. We also derive some auxiliary asymptotic results for the Erd\H{o}s-R\'enyi model.
In Section \ref{sec:estimation}, we consider estimation of the parameters in the S$\beta$M. 
We first consider the ideal situation that the support of the parameter vector is known and  derive consistency and asymptotic normality results for the MLE. 
Next, we consider a more practically relevant situation where the support is unknown and formulate a penalized likelihood approach with the $\ell_0$-penalty  building on a monotonicity lemma, and derive some statistical properties of the estimator. In Section \ref{sec: simulation}, we provide extensive simulation results. In Section \ref{Sec:DA}, we analyze the microfinance take-up example. A summary and discussion on future research are given in Section \ref{sec:conclusion}. All the proofs are relegated to the Appendix. The Appendix is contained in the supplementary material. 

\subsection{Notation}  
Let $\mR_{+} = [0,\infty)$ denote the nonnegative real line. 
For a finite set $F$, let $|F|$ denote its cardinality. 
For a vector $\vbeta \in \mR^{n}$, let $S(\vbeta) = \{ i \in \{ 1,\dots,n \} : \beta_{i} \ne 0 \}$ denote the support of $\vbeta$, and let $\| \vbeta \|_{0}$ denote the number of nonzero elements of $\vbeta$, i.e., $\| \vbeta \|_{0} = |S(\vbeta)|$.  We use $\vbeta_S$ to denote the subvector of $\vbeta$ with indices in $S$ and $S^c$ as the complement of $S$. 
For two sequence of positive numbers $a_{n}$ and $b_{n}$, we write $a_{n} \sim b_{n}$   if $-\infty < \liminf_{n \to \infty} a_{n}/b_{n} \le \limsup_{n \to \infty} a_{n}/b_{n} < \infty$. 

A network with $n$ nodes is represented by a graph $G_{n} = G_n(V, E)$ where $V$ is the set of  nodes or vertices and $E$ is the set of edges or links.   Let $A=(A_{ij})_{i,j=1}^n$ be the adjacency matrix where $A_{ij} \in \{ 0, 1 \}$ is an indicator whether nodes $i$ and $j$ are connected:
\[
A_{ij} =
\begin{cases}
1 & \text{if nodes $i$ and $j$ are connected} \\
0 & \text{if nodes $i$ and $j$ are not connected}
\end{cases}
.
\] 
We focus on undirected graphs with no self loops, so that the adjacency matrix $A$ is symmetric with zero diagonal entries. 
The degree of node $i$ is defined by $d_{i} = \sum_{j=1}^{n} A_{ij} = \sum_{j \ne i} A_{ij}$, and the vector $\bm{d}=(d_1, \dots, d_n)^T$ is called the degree sequence of $G_n$. The total number of edges is denoted by $d_+=\sum_{i=1}^n d_i/2 = \sum_{1 \le i < j \le n} A_{ij}$. 
Modeling a random network or graph is carried out by modeling the entries of $A$ as random variables \citep{Bollobas:etal:2007}.  Denote by $D_{+} = E[d_{+}]$ the expected total number of edges, which is a function of $n$, typically a polynomial. We say that a (random) network  is \textit{dense} if $D_+ \sim n^{2}$   and that it is \textit{sparse} if $D_+ \sim n^{\kappa}$  for some $\kappa\in (0, 2)$ \citep{Bollobas:Riordan:2011}.  Apparently, the smaller $\kappa$ is, the sparser the network is.

\section{Sparse $\beta$-Model}\label{sec:SbetaM}

We first review the Erd\H{o}s-R\'enyi model and the $\beta$-model as a motivation to our S$\beta$M. 
The Erd\H{o}s-R\'enyi model assumes that $A_{ij}$'s are generated as independent  Bernoulli random variables with 
\[
P(A_{ij}=1)=p=\frac{e^\mu}{1+e^\mu},
\]
where $p$ and $\mu$ are parameters possibly dependent on $n$.  Given the graph $G_n$, the MLE of $p$ is 
\[
\hat{p}=\frac{1}{\binom{n}{2}}\sum_{1 \le i<j \le n}A_{ij}=\frac{2d_+}{n(n-1)},
\]
which is also known as the density of the network. The next proposition shows that the MLE $\hat{p}$ retains asymptotic normality even for sparse networks. 
That is, we assume that $p = p_{n}$ may tend to zero as $n \to \infty$ to accommodate  sparse network regimes \citep[see also][for related results]{Krivitsky:Kolaczyk:2015}.
 
\begin{prop}\label{prop:erm}
Consider  the Erd\H{o}s-R\'enyi model. Assume that $n^\gamma p \to p^{\dagger}$ as $n \to \infty$ where $p^{\dagger}>0$ is a fixed constant and $\gamma\in [0,2)$. Then $n^{1+\gamma/2}  (\hat{p}-p) \stackrel{d}{\to} N(0, \sigma^2_{p^{\dagger}})$ {as $n \to \infty$}, 
where  $\sigma^2_{p^{\dagger}} = 2p^{\dagger}(1-p^{\dagger})$ for $\gamma=0$ and $\sigma^2_{p^{\dagger}}= 2p^{\dagger}$ for $\gamma \in (0,2)$. If instead we assume $n^\gamma p = p^{\dagger}$, then  the MLE of $p^{\dagger}$, denoted as $\hat{p}^{\dagger} = n^{\gamma} \hat{p}$, satisfies that $n^{1-\gamma/2}  (\hat{p}^{\dagger}-p^\dagger) \stackrel{d}{\to} N(0, \sigma^2_{p^{\dagger}})$ {as $n \to \infty$}. 
 \end{prop}
 
The expected number of edges for the Erd\H{o}s-R\'enyi model satisfies $D_+ \sim n^{2-\gamma}$ if $n^\gamma p \to p^{\dagger}$ {as $n \to \infty$}.  The proposition shows that as long as $D_+ \to \infty$, which also allows for sparse networks, the MLE of $p$ is asymptotically normal.
If we assume further $n^\gamma p = p^{\dagger}$, then  $p^{\dagger}$ as a non-degenerate constant can be consistently estimated with its MLE being asymptotically normal. In particular, 
for dense networks where $\gamma = 0$,  $\hat{p}^{\dagger}$  is $n$-consistent; in this case, the effective sample size is of order $N=O(n^2)$, so $\hat{p}^{\dagger}$ is $\sqrt{N}$-consistent. For sparse networks where $\gamma=1$,  $n\hat{p}^{\dagger}$  is $\sqrt{n}$-consistent. For a more general $\gamma$, the rate of convergence of $\hat{p}^{\dagger}$ is $n^{1-\gamma/2}$ and the asymptotic variance of $\hat{p}^{\dagger}$ is proportional to $n^{-2+\gamma}$.  Thus $n^{2-\gamma}$  can be seen as the effective sample size for the size invariant parameter $p^{\dagger}$. The notion and importance  of the effective sample size of a network model have been discussed and highlighted by \cite{Krivitsky:Kolaczyk:2015} that study a closely related model for directed networks in the special case when $\gamma=0$ or $1$.  We can also work with the parameter $\mu$ of the Erd\H{o}s-R\'enyi model on the logit scale as follows.

\begin{cor}\label{cor:erm}
Assume that $n^\gamma p \to p^{\dagger}$ as $n \to \infty$ where $p^{\dagger}>0$ is a fixed constant and $\gamma\in [0,2)$.
Define $\mu^{\dagger}=\log [p^{\dagger}/(1-p^{\dagger})]$ for $\gamma=0$ and $\mu^{\dagger}=\log p^{\dagger}$ for $\gamma\in (0,2)$. The MLE of $\mu = \log [p/(1-p)]$ over the parameter space $\mR$ is $\hat\mu=\log[ \hat{p}/(1-\hat{p})]$ and we have
$n^{1-\gamma/2}\left (\hat\mu-\mu \right)  \stackrel{d}{\to} N(0, \sigma^2_{\mu^{\dagger}})$  {as $n \to \infty$}, 
where $\sigma^2_{\mu^{\dagger}}=4+2e^{-\mu^{\dagger}}+2e^{\mu^{\dagger}}$ if $\gamma=0$ and $\sigma_{\mu^{\dagger}}^{2}=2e^{-\mu^{\dagger}}$ if $\gamma \in (0, 2)$.
In addition, we can expand $\mu$ as $\mu = -\gamma \log n + \mu^{\dagger} + o(1)$.
\end{cor}

Again the scaling factor $n^{2-\gamma}$ can be viewed as the effective sample size of the network model. 
From Proposition \ref{prop:erm} and this corollary, the Erd\H{o}s-R\'enyi model has a desirable statistical property that the MLE is asymptotically normal under a wide spectrum of sparsity levels of networks. 
  
With a single parameter, however, the Erd\H{o}s-R\'enyi model  cannot capture heavy tailedness often seen in practice. For example, when $np$ converges to a constant, the degree distribution behaves similarly to  a Poisson law for large $n$. An alternative model specifically designed for capturing degree heterogeneity  is the $\beta$-model that assigns one parameter for each node \citep{Chatterjee:etal:2011}.  In particular, this model assumes that $A_{ij}$'s are independent Bernoulli random variables  with
\begin{equation}
\label{eq:betamodel}
P(A_{ij}=1)=p_{ij}=\frac{ e^{\beta_i + \beta_j} }{ 1 + e^{\beta_i + \beta_j} },
\end{equation}
where $\vbeta=(\beta_1, \dots, \beta_n)^T \in \mR^n$ is an unknown parameter. In this model, $\beta_i$ has a natural interpretation in that it measures the propensity of node $i$ to have connections with other nodes. Namely, the larger $\beta_i$ is, the more likely node $i$ is connected to other nodes. The resulting log-likelihood under the $\beta$-model is easily seen as 
\[ 
\sum_{i=1}^n \beta_i d_i - \sum_{1 \le i<j \le n} \log(1+e^{\beta_i+\beta_j})
\]
and the degree sequence $\vd = (d_{1},\dots,d_{n})^{T}$ is thus a sufficient statistic.  Because of this, the $\beta$-model offers a simple mechanism to describe the probabilistic variation of degree sequences, which serves as an important first step towards understanding the extent to which nodes participate in network connections. More importantly, the $\beta$-model has emerged in recent years as a theoretically tractable model amenable for statistical analysis. In particular,  \cite{Chatterjee:etal:2011} prove the existence and consistency of the MLE of $\vbeta$, while \cite{Yan:Xu:2013} show its asymptotic normality. 

Despite these attractive properties,  the $\beta$-model has a limitation when sparse networks are considered. Up to now, the known sufficient condition for the MLE of the $\beta$-model to be consistent and asymptotically normal is $\max_{1\le i\le n}|\beta_i| =o(\log\log n)$ \citep{Chatterjee:etal:2011,Yan:Xu:2013}, although this condition may not be the best possible. This condition implies that 
\[
\min_{1 \le i < j \le n} p_{ij} \gg \frac{e^{-C\log\log n}}{1+e^{-C\log\log n}} \sim (\log n)^{-C}
\]
for some positive constant $C$. Under this condition,  the expected number of edges of the network should be of order at least $n^2 (\log n)^{-C}$ and hence the network will be dense up to a logarithmic factor.  Part of this requirement stems from the need to estimate $n$ parameters, so we need a sufficient number of connections for each node to estimate all the $\beta$ parameters well.

To conclude, the Erd\H{o}s-R\'enyi model is simple enough to allow desirable asymptotic properties for the MLE under a variety of sparsity levels of the network but too under-parametrized to explain many notable features of the network. On the other hand, the over-parametrized $\beta$-model is more flexible at the expense of a minimal requirement for the density of the network.  Motivated from these observations, we propose the Sparse $\beta$-Model (S$\beta$M) that retains the attractive properties of both. Specifically, the S$\beta$M assumes that $A_{ij}$'s are independent Bernoulli random variables with 
\begin{equation}\label{eq:sbm}
P(A_{ij}=1)=p_{ij}=\frac{ e^{\mu+\beta_i+\beta_j} }{ 1 + e^{\mu+\beta_i+\beta_j} },
\end{equation}
where $\mu \in \mR$ and $\vbeta \in \mR_{+}^{n}$ are both unknown parameters. 
To ensure identifiability, we require that   the elements of $\vbeta$ are nonnegative with at least one element equal to zero, i.e., $\min_{1 \le i \le n} \beta_i= 0$. Hence $\| \vbeta \|_{0} \le n-1$. 
A key assumption we make on the S$\beta$M is that $\vbeta$ is sparse, {hence the name sparse $\beta$-model}. We are mainly  interested in the case where $\| \vbeta \|_{0}  \ll n$.

In this model, $\mu \in \mathbb{R}$ can be understood as the intercept, a baseline term that may tend to $-\infty$ as $n \to \infty$, which allows various sparsity levels for the network similarly to the role of $\mu$ in the Erd\H{o}s-R\'enyi model. Thus $\mu$ is the global parameter characterizing the sparsity of the entire network.
On the other hand, 
 $\vbeta \in \mathbb{R}^n_+$  is a vector of node specific parameters.  It can be understood that node $i$ has no individual effect in forming connections if  $\beta_i=0$, and therefore $\beta_{i}$ controls the local density of the network around node $i$ in addition to its baseline parameter $\mu$.  Such separate treatment of the global and local parameters corresponds to the roles that core and peripheral nodes play in a network. In the context of the microfinance example in Figure \ref{fig:1}, this model allows us to differentially assign parameters only {to certain nodes that are considered ``core''}. In Figure \ref{fig:2}, three simulated examples with $n=50$, $100$ and $200$ are presented to give a general idea of the networks generated from our model, where cores and peripherals are highly visible.

\begin{figure}[hbt!]
\centering
\begin{subfigure}[b]{.32\textwidth}
\includegraphics[width=\textwidth]{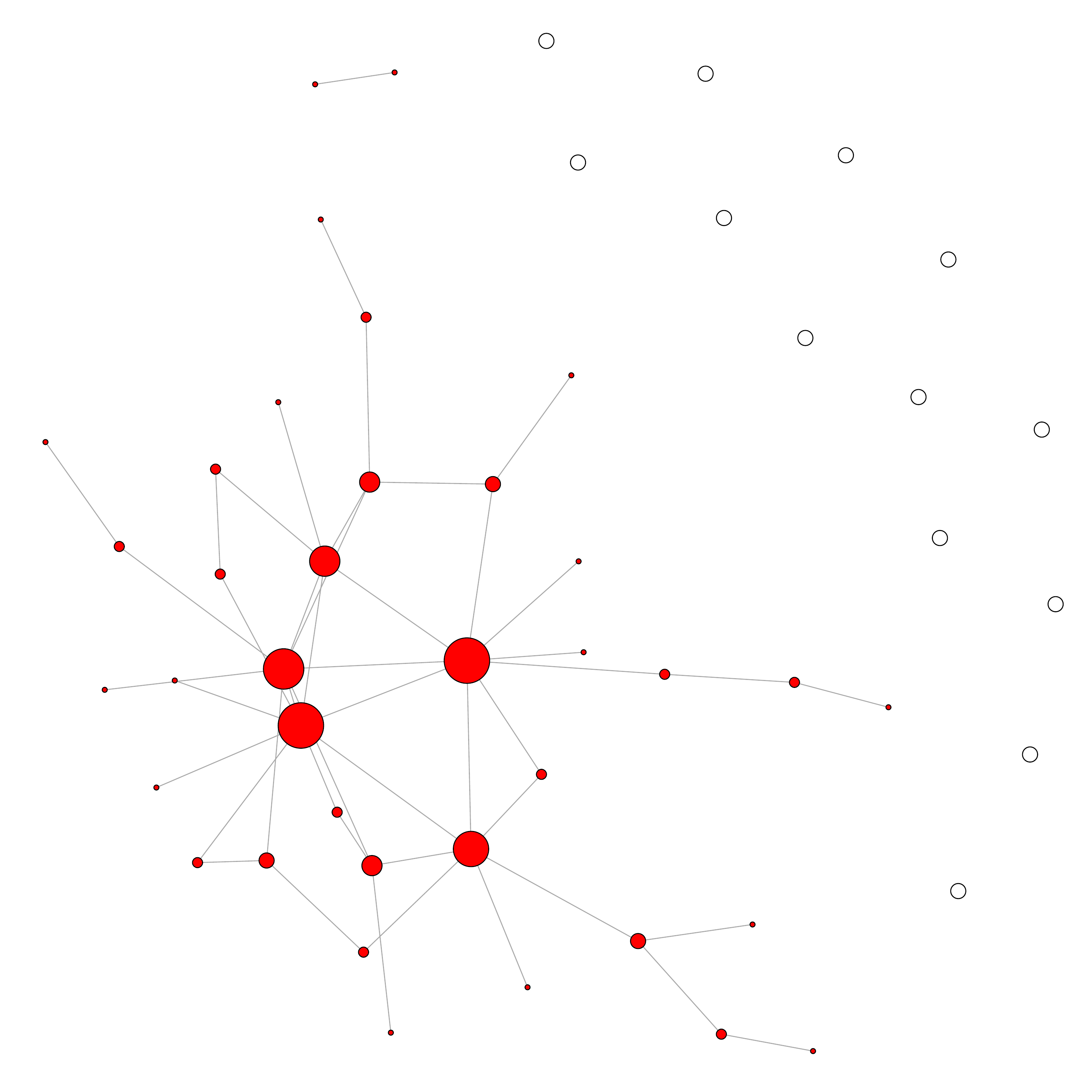}
\end{subfigure}
\begin{subfigure}[b]{.32\textwidth}
\includegraphics[width=\textwidth]{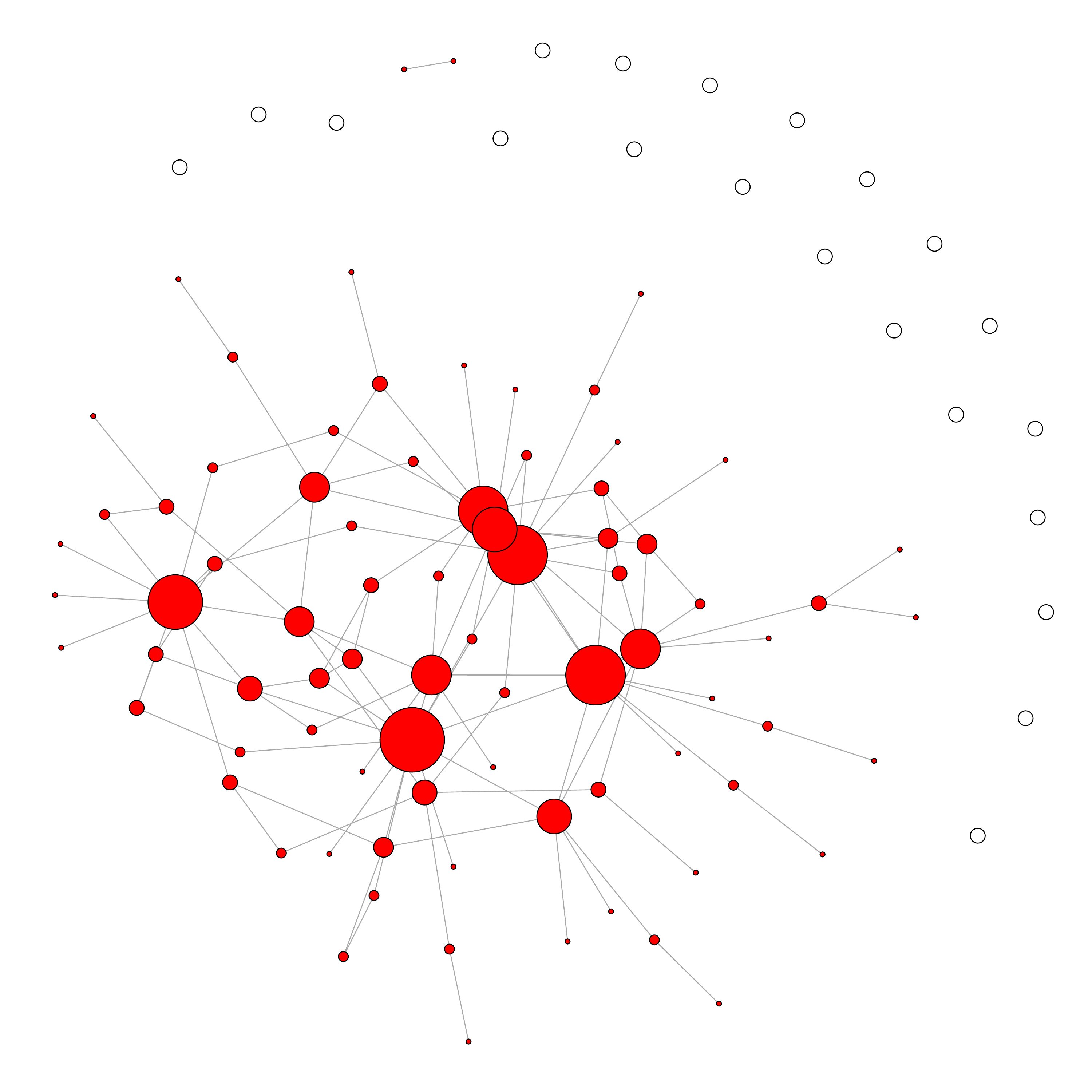}
\end{subfigure}
\begin{subfigure}[b]{.32\textwidth}
\includegraphics[width=\textwidth]{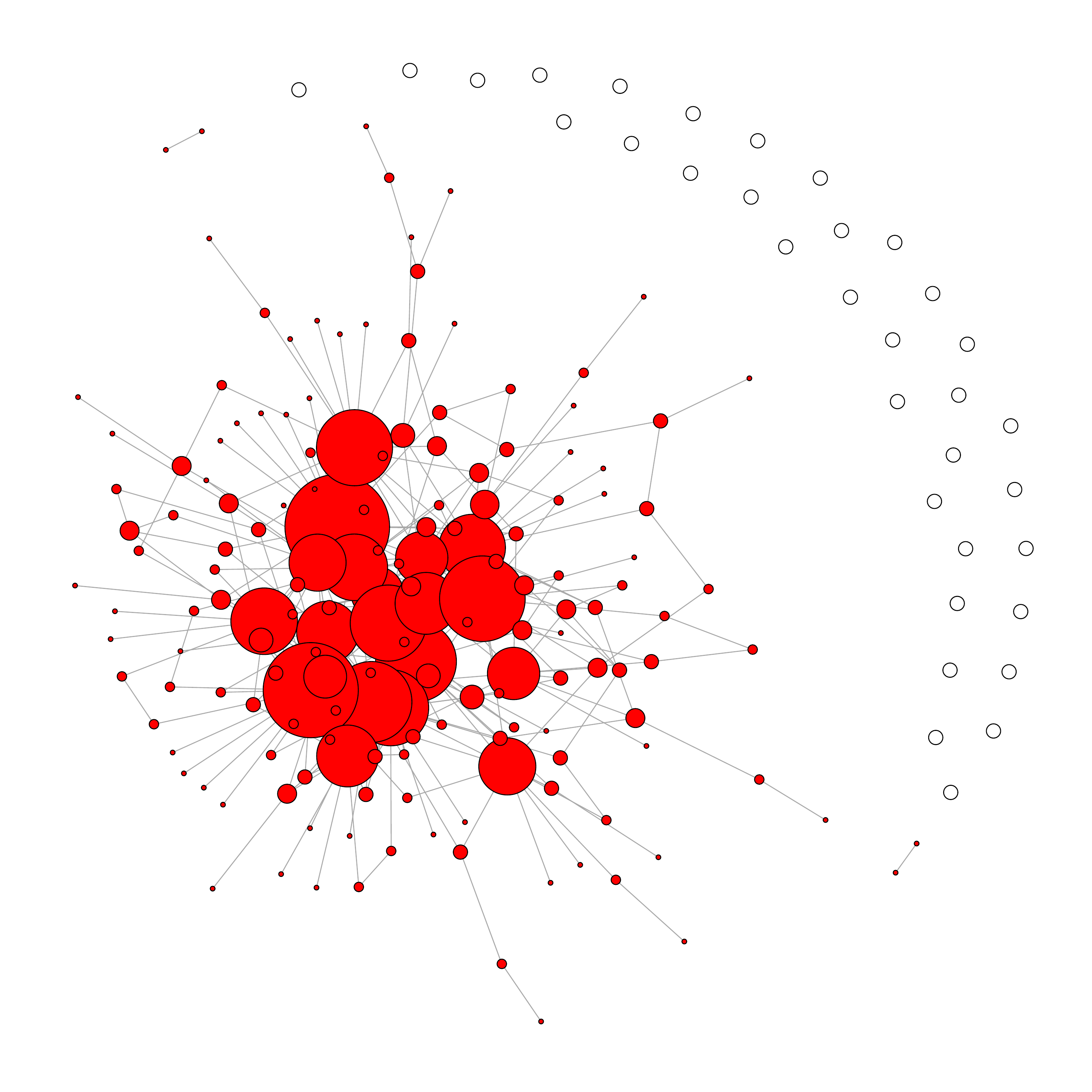}
\end{subfigure}
\caption{\label{fig:2} Some sample networks generated from the S$\beta$M. Left: $n=50$, Middle: $n=100$, Right: $n=200$. The size of the vertex is proportional to its degree. The size of the support of $\vbeta$ is set as $n/10$ with $\beta_i=\sqrt{\log{n}}$ or $0$, while $\mu=-\log n$.}
\end{figure}

Without the sparsity assumption on $\vbeta$, the S$\beta$M reduces to a reparametrized version of the $\beta$-model by shifting $\beta_i$ in the latter by $\mu/2$.  On the other extreme end when $\|\vbeta\|_0=0$, the S$\beta$M reduces to the  Erd\H{o}s-R\'enyi model. Thus, the S$\beta$M interpolates the  Erd\H{o}s-R\'enyi and  $\beta$-models.  By allowing the sparsity level $\| \vbeta \|_{0}$ to be much smaller than $n$, the S$\beta$M can drastically reduce the number of parameters needed in the $\beta$-model, and, as will be discussed in Section \ref{sec:estimation}, allow parameter estimators with desirable statistical properties under sparse network regimes.

We note that \cite{Mukherjee:etal:2019} consider a different reparameterization of the $\beta$-model by inducing a different form of sparsity. Specifically, they consider the model
\[
P_{ij}=P(A_{ij}=1) = \frac{\lambda}{n} \frac{e^{\beta_{i}+\beta_{j}}}{1+e^{\beta_{i}+\beta_{j}}}.
\]
Different from ours, 
the focus of \cite{Mukherjee:etal:2019} is  on testing the hypothesis $H_0: \vbeta=\bm{0}$ against the alternative that $\vbeta$ is nonzero but sparse. 
In addition, they do not consider  estimation of the parameters when $\vbeta$ is sparse, and assume that $\lambda$ is a known constant, which should be contrasted with our S$\beta$M where both $\mu$ and $\vbeta$ are unknown parameters.

An interesting questions arises whether the model in \cite{Mukherjee:etal:2019} can be developed as an alternative to the S$\beta$M in this paper. Noting that for identifiability, their model also requires $\beta_i \ge 0$, we can see that in their model, $P_{ij}$ is always between $\frac{\lambda}{2n}$ and $\frac{\lambda}{n}$. Basically what this says is that the expected number of connections each node can have is of the same order. By contrast, the S$\beta$M allows the expected number of connections that a node can have to differ substantially, which makes the S$\beta$M well suited for networks where there are simultaneously nodes with many connections and nodes with few connections; see the discussion after Theorem 1.

\section{Parameter Estimation in S$\beta$M}
\label{sec:estimation}
In this section, we consider estimation of the parameters in the S$\beta$M. We will denote the true parameter value of $(\mu,\vbeta)$ by $(\mu_0,\vbeta_0)$. 
We first discuss the case where the support of $\vbeta_{0}$ is known. 
We consider the known support case for a theoretical purpose to study the properties of the S$\beta$M. 
Theorem \ref{thm:asymnorm} below reveals two important theoretical properties of the S$\beta$M: 1) the MLE of the parameters in the S$\beta$M can achieve consistency and asymptotic normality under sparse network regimes, and 2) the S$\beta$M can also capture the heterogeneous density patterns for the individual nodes. 
Next, we consider a more practically relevant case where the support is unknown and study the $\ell_{0}$-penalized MLE.

\subsection{MLE with a known support}\label{section:known}
First, we consider the case where $S=S(\vbeta_0)$, the support of $\vbeta_0$, is known and study the asymptotic properties of the MLE for $(\mu_{0},\vbeta_{0S})$. 
The cardinality of the support $s_{0}=|S| = \| \vbeta_{0} \|_{0}$ may grow with the sample size $n$, i.e., $s_{0}=s_{0n} \to \infty$ {as $n \to \infty$}. 
Similarly to \cite{Krivitsky:etal:2011} and \cite{Krivitsky:Kolaczyk:2015}, we also introduce $\log n$ shifts to the parameters to accommodate sparsity of the network in the theoretical setup, and consider the statistical properties of the MLE of the scale-invariant parameters of the S$\beta$M. Specifically, we consider the reparameterization 
\begin{equation}
\label{eq: reparameterization}
\mu = - \gamma \log n + \mu^{\dagger} \quad \text{and} \quad \beta_{i} = \alpha \log n + \beta_{i}^{\dagger} \  \text{for all} \ i \in S
\end{equation}
for some $\gamma \in [0,2)$ and $\alpha \in [0, 1)$ such that $0 \le \gamma-\alpha < 1$. The parameters $\alpha$ and $\gamma$ control the  global sparsity and local density of the network; see the discussion after Theorem \ref{thm:asymnorm}. 
In what follows, for two positive sequences $a_n$ and $b_n$, we write $a_n = \overline{o}(b_n)$ if $a_n = O(n^{-c}b_n)$ for some (sufficiently small) fixed constant $c>0$.  Define $a_n = \overline{o}_{P}(b_n)$ analogously when $a_n$ and $b_n$ are stochastic. 

\begin{theorem}[Consistency and asymptotic normality of MLE with known support]
\label{thm:asymnorm}
Consider the reparameterization (\ref{eq: reparameterization}) for some $\gamma \in [0,2)$ and $\alpha \in [0, 1)$ such that $0 \le \gamma-\alpha < 1$, 
 and let $[-M_1^{\dagger},M_1^{\dagger}] \times [0,M_2^{\dagger}]^{s_{0}}$ be the parameter space for $(\mu^{\dagger},\vbeta_{S}^{\dagger})$ where $M_1^{\dagger} > 0$ and $M_2^{\dagger} > 0$ may depend on $n$ but satisfies $M_1^{\dagger} \vee M_2^{\dagger} = o(\log n)$.  Denote by $(\mu_{0}^{\dagger},\vbeta_{0S}^{\dagger})$ the true parameter value for $(\mu^{\dagger},\vbeta_{S}^{\dagger})$ (the true parameter value may depend on $n$ but has to belong to the parameter space $[-M_1^{\dagger},M_1^{\dagger}] \times [0,M_2^{\dagger}]^{s_{0}}$). Let $(\hat{\mu}^{\dagger},\hat{\vbeta}_{S}^{\dagger})$ be an MLE of $(\mu^{\dagger},\vbeta_{S}^{\dagger})$ over the parameter space $[-M_1^{\dagger},M_1^{\dagger}] \times [0,M_2^{\dagger}]^{s_{0}}$ (the MLE need not be unique). Then: 
\begin{enumerate}
\item[(i)] If in addition $s_0 = \overline{o}(n^{1-\alpha})$, then the MLE $(\hat{\mu}^{\dagger},\hat{\vbeta}_{S}^{\dagger})$  is uniformly consistent in the sense that $\hat{\mu}^{\dagger} = \mu_{0}^{\dagger} + \overline{o}_{P}(1)$ and $\max_{i \in S} | \hat{\beta}^{\dagger}_{i} - \beta_{0i}^{\dagger}| = \overline{o}_{P}(1)$.
\item[(ii)]
If in addition $|\mu_{0}^{\dagger}| \le M_1^{\dagger} -\eta$, $\eta \le \min_{i \in S} \beta_{0i}^{\dagger} \le \max_{i \in S} \beta_{0i}^{\dagger} \le M_2^{\dagger} - \eta$ for some small constant $0 < \eta < M_1^{\dagger} \wedge M_2^{\dagger}$ independent of $n$, and $s_0 = \overline{o}(n^{(1-\alpha)/2})$, then for any fixed subset $F \subset S$, we have 
\[
\Sigma_{F}^{-1/2}
\begin{pmatrix}
n^{1-\gamma/2} (\hat{\mu}^{\dagger}-\mu_{0}^{\dagger}) \\
n^{1/2-(\gamma - \alpha)/2} (\hat{\beta}_{i}^{\dagger} - \beta_{0i}^{\dagger})_{i \in {F}}
\end{pmatrix}
\stackrel{d}{\to} 
N  ( \bm{0}, 
I_{1+|F|}) \quad \text{{as $n \to \infty$}},
\]
where $\Sigma_{F}$ is the diagonal matrix with diagonal entries 
\[
\begin{cases}
2e^{-\mu_{0}^{\dagger}} \ \text{and} \ e^{-\mu_{0}^{\dagger}-\beta_{0i}^{\dagger}} \ \text{for} \ i \in F & \text{if $\alpha < \gamma$} \\
2e^{-\mu_{0}^{\dagger}} \ \text{and} \ 2+e^{-\mu_{0}^{\dagger}-\beta_{0i}^{\dagger}}+e^{\mu_{0}^{\dagger}+\beta_{0i}^{\dagger}} \ \text{for} \ i \in F & \text{if $\gamma = \alpha \in (0,1)$} \\
4 + 2e^{-\mu_{0}^{\dagger}} + 2e^{\mu_{0}^{\dagger}} \ \text{and} \ 2+e^{-\mu_{0}^{\dagger}-\beta_{0i}^{\dagger}}+e^{\mu_{0}^{\dagger}+\beta_{0i}^{\dagger}} \ \text{for} \ i \in F & \text{if $\gamma = \alpha = 0$}
\end{cases}
.
\]
\end{enumerate}
\end{theorem}

Some comments on the theorem are in order. In what follows, we focus on the case where $\alpha < \gamma$ for simplicity of exposition. 
The expected number of edges of the S$\beta$M under the condition of the preceding theorem is
\[ 
\begin{split}
D_{+} &= E[d_{+}] = \sum_{1 \le i < j \le n} p_{ij} \\
&=\binom{n-s_0}{2} \frac{n^{-\gamma}e^{\mu_{0}^{\dagger}}}{1+n^{-\gamma} e^{\mu_{0}^{\dagger}}}+(n-s_0) \sum_{i \in S} \frac{n^{-(\gamma-\alpha)}e^{\mu_{0}^{\dagger}+\beta_{0i}^{\dagger}}}{1+n^{-(\gamma-\alpha)}e^{\mu_{0}^{\dagger}+\beta_{0i}^{\dagger}}}+\sum_{\substack{i,j \in S \\ i < j }} \frac{n^{-(\gamma-2\alpha)}e^{\mu_{0}^{\dagger}+\beta_{0i}^{\dagger}+\beta_{0j}^{\dagger}}}{1+n^{-(\gamma-2\alpha)}e^{\mu_{0}^{\dagger}+\beta_{0i}^{\dagger}+\beta_{0j}^{\dagger}}} \\
&= n^{2-\gamma}e^{\mu_{0}^{\dagger}}/2 + \overline{o}(n^{2-\gamma}),
\end{split}
\]
provided that $s_0 = \overline{o}(n^{1-\alpha})$ (see the proof of Theorem \ref{thm:asymnorm}). 
Theorem \ref{thm:asymnorm} shows that the MLE of the parameters in the S$\beta$M (when the support is known) can achieve consistency and asymptotic normality for sparse networks, as the total number of expected edges is allowed to be of the order $O(n^{2-\gamma})$.  Put another way, in the S$\beta$M, the sparsity in the parameter $\vbeta$ permits statistical inference for sparse networks. 
Our result should be contrasted with the $\beta$-model where the MLE is known to be consistent and asymptotically normal only for relatively dense networks. 

In addition, the S$\beta$M can also capture the heterogeneous density patterns for the individual nodes in the sense that
\[
E[d_{i}] =
\begin{cases}
n^{1-(\gamma-\alpha)} e^{\mu_{0}^{\dagger} + \beta_{0i}^{\dagger}} + \overline{o}(n^{1-(\gamma-\alpha)}) & \text{if $i \in S$} \\
n^{1-\gamma} e^{\mu_{0}^{\dagger}} + \overline{o}(n^{1-\gamma}) & \text{if $i \notin S$}
\end{cases}
.
\]
Intuitively speaking, $\gamma$ (or the magnitude of $\mu$) controls the global sparsity while $\alpha$ (or the magnitude of $\beta_i$) is controlling the local density. Namely, as $\gamma$ increases, all nodes will be less likely to be connected, while as $\alpha$ increases, the nodes in the support of $\vbeta$ will be more likely to be connected. Theorem \ref{thm:asymnorm} shows how these global and local parameters affect the effective sample sizes for $\mu$ (global parameter) and $\vbeta_{S}$ (local parameter). The theorem implies that the effective sample size for the global parameter is $n^{2-\gamma}$ which is decreasing in $\gamma$ similarly to the Erd\H{o}s-R\'enyi graph (see the discussion after Corollary \ref{cor:erm}), while that for the local parameter is $n^{1-\gamma+\alpha}$ which is decreasing in $\gamma$ but increasing in $\alpha$. The expected total number of edges allowed by the S$\beta$M is of the order
\[
 s_0n^{1-\gamma+\alpha}+(n-s_0)n^{1-\gamma} \sim n^{2-\gamma}+s_0(n^{1-\gamma+\alpha}-n^{1-\gamma}),
\]
which is $O(n^{2-\gamma})$ if $s_0 = o (n^{1-\alpha})$. Thus, the S$\beta$M can model very sparse networks. To our best knowledge, we are not aware of any models with estimators enjoying similar consistency and asymptotic normality properties for such sparse networks, apart from the models studied by \cite{Krivitsky:Kolaczyk:2015} that are much simpler. 
%and \cite{Graham:2017}; their models do not involve node-specific parameters
%apart from the model studied by Krivitsky and Kolaczyk (2015) which is much simpler.

Finally, the proof of Theorem \ref{thm:asymnorm} is nontrivial since there are two types of parameters with different rates, one being common to the nodes and the other being node-specific, and the number of parameters $1+s_{0}$ may diverge as $n$ increases, which is reminiscent of the incidental parameter problem \citep{NeymanScott1948, Li:etal:2003,HahnNewey2004}. 
To prove the uniform consistency, we work with the concentrated negative log-likelihoods for $(\mu^{\dagger},\vbeta_{S}^{\dagger})$ and show that they converge in probability to some nonstochastic functions uniformly in $i \in S$. To prove the asymptotic normality, we use iterative stochastic expansions to derive the uniform asymptotic linear representations for $(\hat{\mu}^{\dagger},\hat{\vbeta}_{S}^{\dagger})$. See the proof in Appendix \ref{sec:proof thm1} in the supplementary material for the details.

\subsection{$\ell_{0}$-penalized MLE with an unknown support}\label{section:unknown}
In practice, the support of $\vbeta_{0}$  is  usually unknown. In this section,
we consider and analyze the $\ell_0$-norm constrained maximum likelihood estimator for estimating the parameters of the model when this is the case.  
The negative log-likelihood of the S$\beta$M is given by
\[
\ell_n(\mu,\vbeta)= -d_{+} \mu - \sum_{i=1}^{n}  d_i \beta_i+ \sum_{1 \le i<j \le n} \log\left(1+e^{\mu+\beta_i+\beta_j}\right).
\]
Then, we shall estimate the parameters as 
\begin{equation}
(\hat\mu(s), \hat\vbeta(s)) =\argmin_{\mu \in \mathbb{R},\vbeta \in \mathbb{R}_{+}^{n}}~\ell_n(\mu,\vbeta) \quad \text{subject to} \  \|\vbeta\|_0 \le s,\label{eq:ple}
\end{equation}
where $s \in \{1, 2, \dots, n-1 \}$ is an integer-valued tuning parameter. We restrict $s$ to be less than $n$ so that the identifiability condition $\min_{1 \le i \le n}\beta_i= 0$ is automatically satisfied. If there is a question of the existence of the global optimal solution in (\ref{eq:ple}), we restrict the parameter space to be a (sufficiently large) compact rectangle. We refer to Appendix \ref{sec: MLE existence} in the supplementary material for some discussion on the existence of $\ell_0$-constrained MLE in (\ref{eq:ple}). 

The optimization problem \eqref{eq:ple} is a combinatorial problem that seems difficult to solve. For each $s$, a naive approach to compute the solution of \eqref{eq:ple} is to fit $\binom{n}{s}$ models, each assuming $s$ out of $n$ parameters in the S$\beta$M are nonzero, and then choose the model that gives the smallest negative log-likelihood. This strategy is used routinely in the so-called best subset selection for regression models, which is known for being unsuitable for datasets with a large number of parameters. Remarkably, the S$\beta$M has a property that at most only $n-1$ models need to be examined before the optimal choice is decided, making it attractive computationally. In particular, we have the following monotonicity lemma stating that the entries of $\hat\vbeta (s)= (\hat{\beta}_{1}(s),\dots,\hat{\beta}_{n}(s))^{T}$ are ordered according to those of the degree sequence $\bm{d} = (d_{1},\dots,d_{n})^{T}$.  Before presenting this lemma, we introduce the following notation to handle tied degrees. 
Let 
\begin{equation}
d_{(1)}>d_{(2)}>\cdots>d_{(m)}
\end{equation}
denote the distinctive values of $d_i$'s. Denote by $S_k$ the set of indices of those $d_i$'s that equal to $d_{(k)}$ and by $s_k$  its cardinality; that is, $S_k=\{ i \in \{ 1,\dots, n \}: d_i=d_{(k)}\}$ and $s_{k} = |S_{k}|$. By definition, $\sum_{k=1}^m s_k=n$. If no two degrees are tied, then $m=n$ and $s_k=1$ for any $k=1,\dots,n$.

\begin{lemma}[Monotonicity lemma]\label{lemma:1}
\label{lem: sorting}
The estimate $\hat{\vbeta}(s)$ in \eqref{eq:ple} has the following properties.
\begin{itemize}
\item[(i)] If $d_i<d_j$, then we have $\hat\beta_i (s) \le  \hat\beta_j (s)$ for any $s<n$;
\item[(ii)] If $d_i=d_j$, then we have $\hat\beta_i (s) = \hat\beta_j (s)$ for  any $s$ such that $s=\sum_{k=1}^K s_k$ for some $K\le m-1$.
\end{itemize}
\end{lemma}
The proof of Lemma \ref{lemma:1} and other proofs for Section \ref{section:unknown} can be found in Appendix \ref{sec:proof unknown} in the supplementary material. Lemma \ref{lemma:1} implies that $\hat\vbeta (s)$ as the constrained MLE of \eqref{eq:ple} has the same order as the degree sequence. That is, for the constrained optimization in \eqref{eq:ple} with a penalty parameter $s$, we just assign nonzero $\beta$ to those nodes whose degrees are among the largest $s$ nodes. More precisely,  if $s=\sum_{k=1}^K s_k$ for some $K \le m-1$, then $\hat\beta_i (s)\ge 0$ for $i\in \bigcup_{k=1}^K S_k$ and $\hat\beta_i (s)=0$ for $i \in \bigcup_{K < k \le  m} S_k$. In other words, we can find \textit{a priori} the support of $\hat\vbeta (s)$ from the degree sequence and can compute $\hat\vbeta (s)$ by solving the following optimization problem \textit{without the $\ell_{0}$-penalization}:
\[
( \hat\mu (s), \hat \vbeta(s)  )= \argmin_{\mu \in \mathbb{R}, \vbeta \in \mathbb{R}_{+}^n}~ \ell_{n}(\mu,\vbeta)\quad \text{subject to} \   \beta_{i} = 0 \ \text{for} \ i \in \bigcup_{K < k \le  m} S_k.
\]
This way, we can efficiently compute a solution path of $( \hat\mu (s), \hat \vbeta(s)  )$ as a function of $s \in \{ s_{1},s_{1}+s_{2},\dots,\sum_{k=1}^{m-1}s_{k} \}$ without solving a computationally expensive combinatorial problem. 
We note that  the set $\{1, \dots, n-1\} \setminus \{ s_{1},s_{1}+s_{2},\dots,\sum_{k=1}^{m-1}s_{k} \}$ is excluded from consideration for the tuning parameter $s$, because otherwise the solution to the constrained optimization will not be unique.  

The preceding lemma shows that there will be a sequence of supports 
\begin{equation}
S_{1}, S_{1} \cup S_{2}, \dots,\bigcup_{k=1}^{m-1} S_{k},
\label{eq:support}
\end{equation}
for $\hat{\vbeta}(s)$ with $s \in \{ s_{1},s_{1}+s_{2},\dots,\sum_{k=1}^{m-1}s_{k} \}$. 
Next, we show that as long as the smallest nonzero element of $\vbeta_0$ is above a certain threshold, with high probability the true support $S(\vbeta_0)$ is included in the support sequence (\ref{eq:support}) constructed from the degree sequence $\vd$.

\begin{lemma} 
\label{lem: beta min}
Let $S=S(\vbeta_0)$, and let $\tau \in (0,1)$ be given. Pick any $i \in S$ and $j \in S^{c}$. Suppose that 
\begin{equation}
\beta_{0i} > \log \left (1+c_{n,\tau}(1+e^{\mu^{-}})(1+e^{2\overline{\beta}+\mu^{+}}) \right),
\label{eq: betamin}
\end{equation}
where $c_{n,\tau} = \sqrt{(2/(n-2)) \log (2/\tau)}, \ \overline{\beta} = \max_{1 \le k \le n} \beta_{0k}, \ \mu^{+} = \max \{ \mu_0, 0 \}$, and $\mu^{-} = \max \{ -\mu_0,0 \}$. 
Then $d_{i} > d_{j}$ with probability at least $1-\tau$.
\end{lemma}

By the union bound, Lemma \ref{lem: beta min} immediately yields the following corollary.

\begin{cor}[$\beta$-min condition]\label{cor:betamin}
Pick any $\tau \in (0,1)$. Suppose that the following \textit{$\beta$-min condition} is satisfied:
\begin{equation}\min_{i \in S} \beta_{0i} > \log \left (1+c_{n,\tau/n(n-1)}(1+e^{\mu^{-}})(1+e^{2\overline{\beta}+\mu^{+}}) \right).
\label{eq:betaminUnion}
\end{equation}
Then we have $\min_{i \in S} d_{i} > \max_{j \in S^{c}} d_{j}$ with probability at least $1-\tau$. 
\end{cor}

Corollary \ref{cor:betamin} specifies the minimum magnitude of the nonzero $\beta$'s for the S$\beta$M to include the true support $S(\vbeta_0)$ in the support sequence (\ref{eq:support}). For this reason, we call the condition in (\ref{eq:betaminUnion}) the $\beta$-min condition. Such $\beta$-min conditions are common in the literature on high-dimensional statistics to guarantee support recovery; see, e.g., \cite{meinshausen2006,zhao2006,wainwright2009,buhlmann2013}. With this $\beta$-min condition, if we choose $s=|S(\vbeta_0)|$, then we can identify the support of $\vbeta_0$ by solving the optimization problem in \eqref{eq:ple} with a probability close to one. The issue of determining the sparsity level $s$ will be discussed in Section \ref{sec:BIC}. 
Note that $c_{n,\tau/n(n-1)} \sim \sqrt{(\log n)/n}$, and that the right hand side of (\ref{eq:betaminUnion}) is of constant order as long as $e^{2\overline{\beta} +|\mu_0|} = O(\sqrt{n/\log n})$.

Finally, we evaluate the prediction risk for the estimator $(\hat{\mu}(s),\hat\vbeta(s))$ for a given sparsity level $s$. 
Recall that the true value of $(\mu,\vbeta)$ is denoted by $(\mu_{0},\vbeta_{0})$ with $s_{0} = \| \vbeta_{0} \|_{0}$. In general $s$ and $s_{0}$ may differ. 
Let $\mathcal{R} (\mu,\vbeta)$ be the risk of the parameter value $(\mu,\vbeta)$ which is defined by the expected normalized negative log-likelihood, i.e., 
\[
\mathcal{R} (\mu,\vbeta) = E[D_{+}^{-1}\ell_{n}(\mu,\vbeta)],
\]
where we think of $D_{+} = E[d_{+}]$ as the effective sample size. Normalization by $D_{+}$ is natural since the risk at the true parameter $\mathcal{R}(\mu_{0},\vbeta_{0})$ is of constant order up to logarithmic factors under sparse network scenarios; see the discussion after Theorem \ref{thm: risk bound} (recall that in the linear regression case with squared loss function, the risk at the true parameter is the error variance, which is constant).  
For a given sparsity level $s$, consider the $\ell_{0}$-constrained estimator $(\hat{\mu}(s),\hat{\vbeta}(s))$ as in (\ref{eq:ple}):
\[
(\hat{\mu}(s),\hat{\vbeta}(s)) = \argmin \{ \ell_n (\mu,\vbeta) : (\mu,\vbeta) \in \Theta_{s}  \},
\]
where $\Theta_{s} = \{(\mu,\vbeta)  \in \mR \times \mR_{+}^{n}: | \mu | \le M_1, \vbeta \in [0,M_2]^{n}, \| \vbeta \|_{0} \le s \}$ and $M_1,M_2$ are given positive deterministic numbers. We assume that $M_1$ and $M_2$ are  sufficiently large and may increase with $n$, but suppress the dependence of the parameter space $\Theta_{s}$ on $M_1$ and $M_2$ (the risk bound in Theorem \ref{thm: risk bound} is nonasymptotic so specifying how fast $M_1$ and $M_2$ can grow with $n$ is not necessary;  effectively, however, $M_1$ and $M_2$ should be constrained so that the risk bound is tending to zero).  In addition, both $s_{0}$ and $s$ can depend on $n$. Following the empirical risk minimization literature \citep[see, e.g.,][]{GreenshteinRitov2004,Koltchinskii2011}, we will evaluate the performance of the estimator $(\hat{\mu}(s),\hat{\vbeta}(s))$ by the (local) excess risk relative to the parameter space $\Theta_{s}$
\[
\mathcal{E}_{s} = \mathcal{R}(\hat{\mu}(s),\hat{\vbeta}(s)) - \inf_{(\mu,\vbeta) \in \Theta_{s}} \mathcal{R}(\mu,\vbeta). 
\]
We note that the (global) excess risk relative to the true parameter $(\mu_{0},\vbeta_{0})$ can also be bounded by the decomposition 
\[
\mathcal{R}(\hat{\mu}(s),\hat{\vbeta}(s)) -  \mathcal{R}(\mu_{0},\vbeta_{0}) =  \left [ \inf_{(\mu,\vbeta) \in \Theta_{s}} \mathcal{R}(\mu,\vbeta) - \mathcal{R}(\mu_{0},\vbeta_{0}) \right ] +  \mathcal{E}_{s},
\]
where the first term on the right hand side accounts for the deterministic bias. 
The following theorem derives high-probability upper bounds on the excess risk $\mathcal{E}_{s}$. 
\begin{theorem}[Excess risk bound]
\label{thm: risk bound}
For any given $\tau \in (0,1)$, we have 
\begin{equation}
\begin{split}
\mathcal{E}_{s} &\le  \frac{2}{D_{+}} \Bigg [ M_1 \left \{ \sqrt{2\Var (d_{+}) \log (4/\tau)}+ (\log (4/\tau))/3 \right \} \\
&\qquad \qquad+ M_2 s \left \{  \sqrt{2\max_{1 \le i \le n} \Var (d_{i}) \log (4n/\tau)}+  (\log (4n/\tau))/3  \right \}  \Bigg] 
\end{split}
\label{eq:generic bound}
\end{equation}
with probability at least $1-\tau$. In particular, if $\mu_{0} =  -\gamma \log n +O(1)$ and $\beta_{0i} = \alpha \log n +O(1)$ uniformly in $i \in S(\vbeta_{0})$ for some $\gamma \in [0,2)$ and $\alpha \in [0,1)$ with $0 \le \gamma - \alpha < 1$, and $s_{0} = o(n^{1-\alpha})$, then we have $D_{+} \sim n^{2-\gamma}$ and 
\begin{equation}
\mathcal{E}_{s} = O_{P} \left(\frac{M_{1}}{n^{1-\gamma/2}} + \frac{M_{2}s \sqrt{\log n}}{n^{3/2-(\gamma+\alpha)/2}} \right ).
\label{eq:risk bound}
\end{equation}
\end{theorem}

%{KK note: I have to change the discussion after ``In particular'' in Theorem 2.}

In the latter setting of Theorem \ref{thm: risk bound}, it is not difficult to see that $E[\ell_{n}(\mu_{0},\vbeta_{0})] \sim n^{2-\gamma}$ up to logarithmic factors so that the risk at $(\mu_{0},\vbeta_{0})$ normalized by $D_{+}$ is of constant order $\mathcal{R}(\mu_{0},\vbeta_{0}) \sim 1$ up to logarithmic factors. In addition, if e.g. $M_{1} \sim \log n$ and $M_{2} \sim \log n$, then the bound (\ref{eq:risk bound}) becomes
\[
\mathcal{E}_{s} = O_{P} \left(\frac{\log n}{n^{1-\gamma/2}} + \frac{s (\log n)^{3/2}}{n^{3/2-(\gamma+\alpha)/2}} \right ).
\]
Hence, the estimator $(\hat{\mu}(s),\hat{\vbeta}(s))$ is \textit{persistent} in the sense of \cite{GreenshteinRitov2004}, i.e., $\mathcal{E}_{s} \stackrel{P}{\to} 0$ ({as $n \to \infty$}), as long as 
\begin{equation}
s=o(n^{3/2 - (\gamma+\alpha)/2}/(\log n)^{3/2}),
\label{eq: persistency}
\end{equation}
and  provided that the true sparsity level satisfies $s_{0} = o(n^{1-\alpha})$. Thus, the persistency is more difficult to achieve when $\gamma$ or $\alpha$ is large, i.e., the generated networks tend to be globally sparse or locally dense. The reason why the persistency is more difficult when the networks are locally dense is that while the effective sample size $D_+$ does not depend on the local density (i.e., $\alpha$), the variance of each node degree increases with the local density. Condition (\ref{eq: persistency}) is automatically satisfied if $\gamma + \alpha <1$ since $s$ is at most $n-1$. 
In addition,  the bound can achieve the near parametric rate $(\log n)/n^{1-\gamma/2}$ with respect to the effective sample size $D_{+} \sim n^{2-\gamma}$ as long as $s = o(n^{(1-\alpha)/2}/\sqrt{\log n})$.

\section{Simulation Study}\label{sec: simulation}

\subsection{Selection of sparsity level}
\label{sec:BIC}
In practice, we have to choose the sparsity level $s$ for the $\ell_{0}$-penalized MLE to work. In this simulation study, we will examine the following version of BIC
\begin{equation}
\textsf{BIC}(s) = 2 \ell_n(\hat{\mu}(s),\hat{\vbeta}(s))+  s\log\left(n(n-1)/2\right). 
\label{eq: BIC}
\end{equation}
Recall that we have defined $\ell_n (\mu,\vbeta)$ by the negative log-likelihood. Using the notation in Section \ref{section:unknown}, 
we choose $s$ that minimizes the BIC:
\[
\hat{s} = \argmin \left \{ \textsf{BIC}(s) : s \in \left \{ s_{1},s_{1}+s_{2},\dots,\sum_{k=1}^{\tilde{m}}s_{k} \right \} \right \},
\]
{where $1\le \tilde{m}<m$ is used to constrain the maximum size of the models to be inspected. The simplest choice of $\tilde{m}$ is $m-1$, where $m$ is the number of distinct degrees, corresponding to the $\beta$-model. In practice however, we recommend using an $\tilde{m}$ such that $\sum_{k=1}^{\tilde{m}}s_{k}$ is a loose upper bound of the true model size $s_0$. We note that similar strategies restricting the maximum sizes of candidate models are widely used in choosing high-dimensional models; see, for example, \cite{chen2008}, \cite{wang2009}, and \cite{fan2013}. Otherwise, the value of the corresponding information criterion for a model, especially an over-fitted model, may not be well defined.} 
The final estimator is then given by $(\hat\mu (\hat s)_,\hat\vbeta (\hat s))$. We shall study the performance of the BIC via numerical simulations.

The BIC defined in (\ref{eq: BIC}) uses $n(n-1)/2$ as the sample size. 
In view of our previous discussion on the effective sample size, it would be natural to use $D_{+}$ or its unbiased estimate $d_{+}$ in place of $n(n-1)/2$ by defining a different BIC: 
\begin{equation}
\textsf{BIC}^{*}(s) =2 \ell_n(\hat{\mu}(s),\hat{\vbeta}(s))+  s\log (d_{+}). 
\label{eq: BIC2}
\end{equation}
Preliminary simulation results suggest that, however, the performance of the BIC in (\ref{eq: BIC2}) is similar or slightly worse than the one in (\ref{eq: BIC}) in most cases in terms of model selection and parameter estimation. Hence we only report the simulation results using (\ref{eq: BIC}).

Let us discuss selection consistency of the BIC defined in (\ref{eq: BIC}), i.e., $P(S(\hat{\vbeta}(\hat{s})) = S(\vbeta_0)) \to 1$ {as $n \to \infty$}. 
For given $S \subset \{ 1,\dots, n \}$ with $|S| \le n-1$, let $(\hat{\mu}^{S},\hat{\vbeta}^{S})$ denote the support constrained MLE
$(\hat{\mu}^{S},\hat{\bm{\beta}}^{S}) = \argmin \{ \ell_{n}(\mu,\bm{\beta}) : \mu \in \mR, \bm{\beta} \in \mR_{+}^{n}, \supp (\bm{\beta}) = S \}$. The selection consistency of the BIC in (\ref{eq: BIC}) follows if $P(\min_{S \neq S(\vbeta_{0})} \textsf{BIC}_{S} > \textsf{BIC}_{S(\vbeta_{0})}) \to 1$ {as $n \to \infty$} where $\textsf{BIC}_{S} =2 \ell_n(\hat{\mu}^{S},\hat{\vbeta}^{S})+  |S|\log\left(n(n-1)/2\right)$. Several papers have studied consistency of BIC and its modification for variable selection in linear and generalized linear regression models with increasing numbers of covariates; see, for example, \cite{chen2008, wang2009, fan2013}. 
Importantly, however, none of these results can be adapted to our case (at least directly) since, in addition to the fact that the number of possible models is extremely large, the parameter space for $\vbeta$ is restricted to the positive orthant $\mathbb{R}_{+}^{n}$, in the overfitting case (i.e., $S \supset S(\vbeta_{0})$), and the asymptotic behavior of $\hat{\beta}_{i}^{S}$ for $i \in S \setminus S(\vbeta_{0})$ is nonregular as the corresponding true parameter lies on the boundary of the parameter space \citep{andrews1999}. The fact that the true parameter is on the boundary of the parameter space prevents us from expanding $\hat{\vbeta}^{S}$ into a linear term, which is a crucial step in proving the selection consistency of BIC in \cite{fan2013}. 
Developing formal asymptotic theory for BIC under such nonregular cases (and with diverging number of parameters) is beyond the scope of the present paper and left for future research. In any case, the simulation results below demonstrate good performance of the BIC in terms of model selection. Additional simulations using the information criteria in \cite{chen2008,  fan2013} show similar performance to the BIC we used.

\subsection{Simulation results}

In this simulation study, we consider the following configurations of $(\mu_{0},\vbeta_{0})$:
\begin{enumerate}
\item[(i)] $\mu_{0} = - 1.5$, and $\beta_{0i} = 1.5$, $\sqrt{\log n}$, or $\log n$ for $i \in S(\vbeta_{0})$; 
\item[(ii)] $\mu_{0} = -\sqrt{\log n}$, and $\beta_{0i} = 1.5$, $\sqrt{\log n}$, or $\log n$ for $i \in S(\vbeta_{0})$; 
\item[(iii)] $\mu_{0} = -\log n$, and $\beta_{0i} = 1.5$, $\sqrt{\log n}$, or $\log n$ for $i \in S(\vbeta_{0})$;
\end{enumerate}
where $n=50$, $100$, $200$ or $400$. The sparsity level of $\vbeta_{0}$ is either $s_{0} = |S(\vbeta_{0})| = 2, \lfloor \sqrt{n/2} \rfloor,  \left \lfloor \sqrt{n} \right \rfloor$, or $ \left \lfloor 2\sqrt{n} \right \rfloor$, where $\left \lfloor a \right \rfloor$ denotes the largest integer smaller than $a$.  Since the indices of the nonzero elements of $\vbeta_0$ do not matter  for our estimation procedure, we simply choose the first $s_{0}$ elements of $\vbeta_{0}$ to be nonzero. The number of Monte Carlo repetitions is  $1000$ for each case of simulation. 
To speed up our estimation procedure, in this simulation study, {we restricted the maximum number of sparsity levels $s$ examined to be $\max \{ 40,\lfloor 4 n \rfloor \}$. 
Finally, we used the programming language R \citep{rcore2020} to conduct simulations and real data analysis. To compute support constrained MLEs, we used the \textsf{nlminb} function in R.}
 
The configurations in (i)-(iii) above are chosen to reflect various degrees of sparsity for the overall network globally and for individual nodes locally. Recall that an induced subgraph of a graph is another graph formed from a subset of the vertices of the graph and all of the edges connecting pairs of vertices in that subset. If $\mu_{0} = -\log n$, then the subgraph induced by those nodes with zero $\beta$ parameters will form a sparse Erd\H{o}s-R\'enyi graph with $D_+ \sim n-s_0$. If  $\mu_{0} = - 1.5$, then this subgraph is almost dense in that  $D_+ \sim (n-s_0)^2/\log (n-s_0)$. If $\mu_{0} = -\sqrt{\log n}$, then the induced subgraph lies somewhere between these two cases. The specification  $\beta_{0i} =\log n$ is guided by the reparameterization in Theorem \ref{thm:asymnorm}. By specifying  $\beta_{0i} = 1.5$ or $\sqrt{\log n}$, we want to consider those local parameters that are much smaller than $\log n$.

\begin{figure}[p]
\begin{minipage}{\textwidth}
\centering

\begin{subfigure}[b]{0.3\textwidth}
     \caption[ ]
        {{\footnotesize  $\beta_{0i} = 1.5, \mu_{0} = - 1.5$}}
 	\centering
	\includegraphics[width=1.1\textwidth, height = \textwidth]{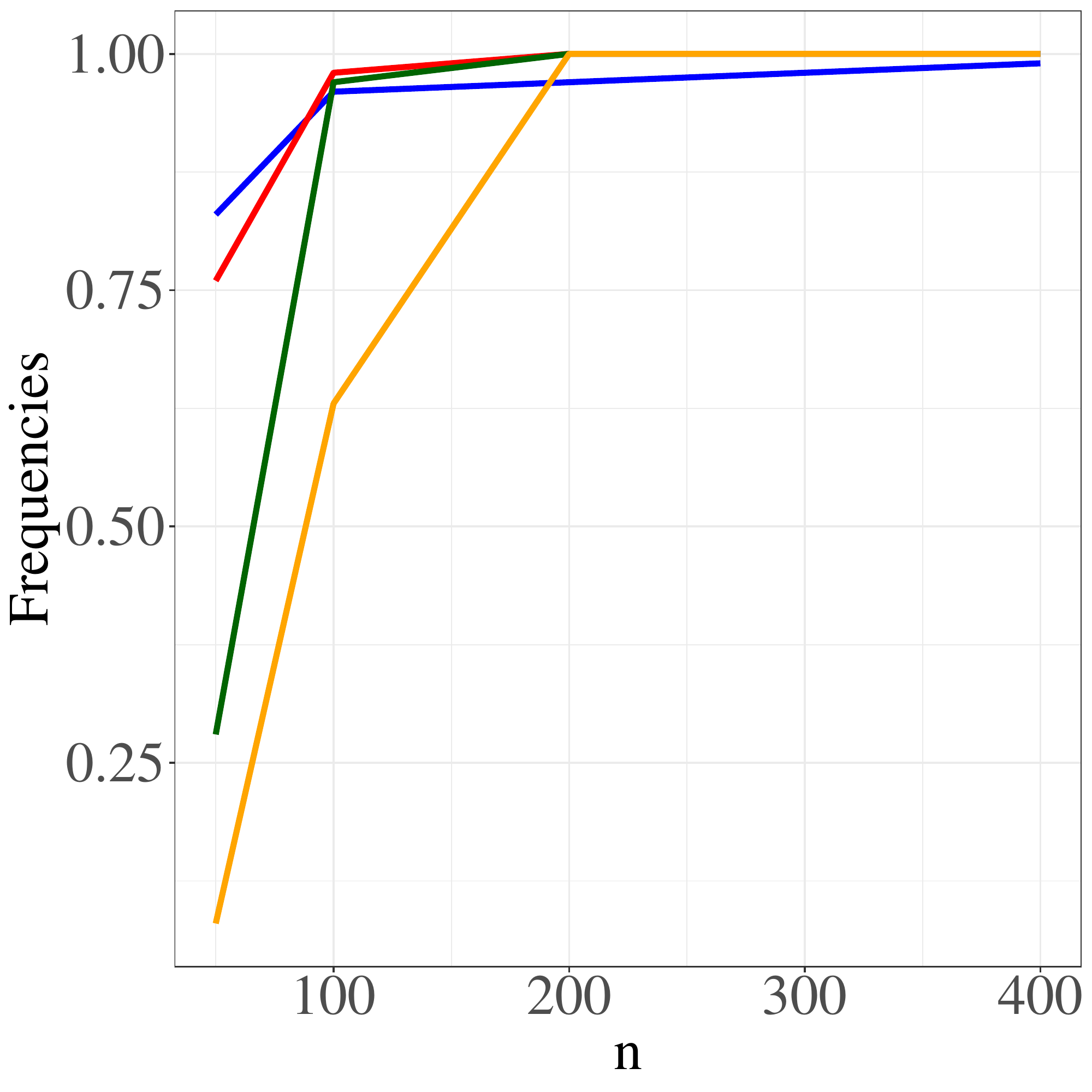}
\end{subfigure}
 \hfill \hfill \hfill \hspace{-5mm}
\begin{subfigure}[b]{0.3\textwidth}
   \caption[]
       {{\footnotesize $\beta_{0i} = \sqrt{\log n}, \mu_{0} = - 1.5$}}
 	\centering
	\includegraphics[width=\textwidth]{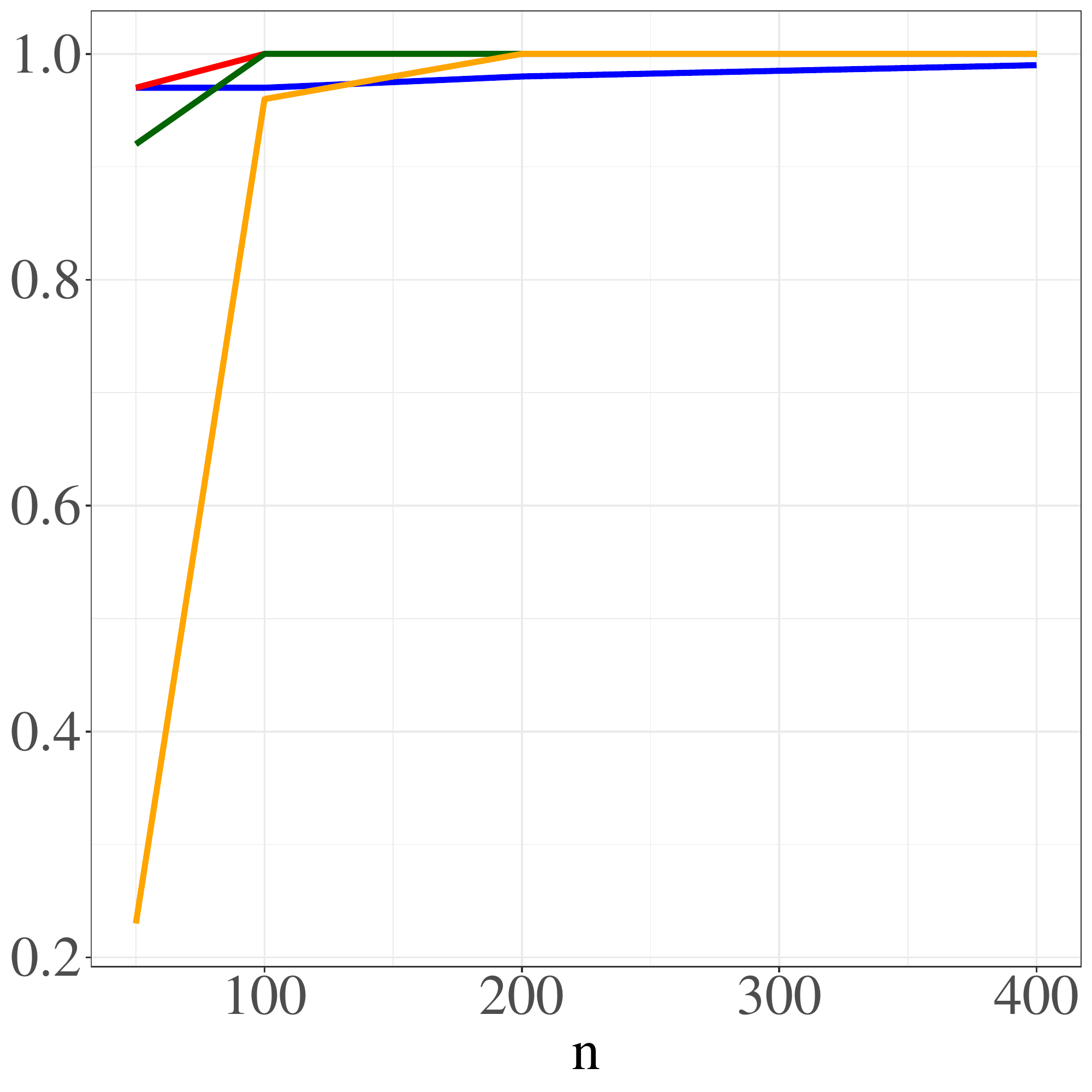}
\end{subfigure}
\hfill
\begin{subfigure}[b]{0.3\textwidth}
   \caption[]
      {{\footnotesize $\beta_{0i} = \log n, \mu_{0} = - 1.5$}}
 	\centering
	\includegraphics[width=\textwidth]{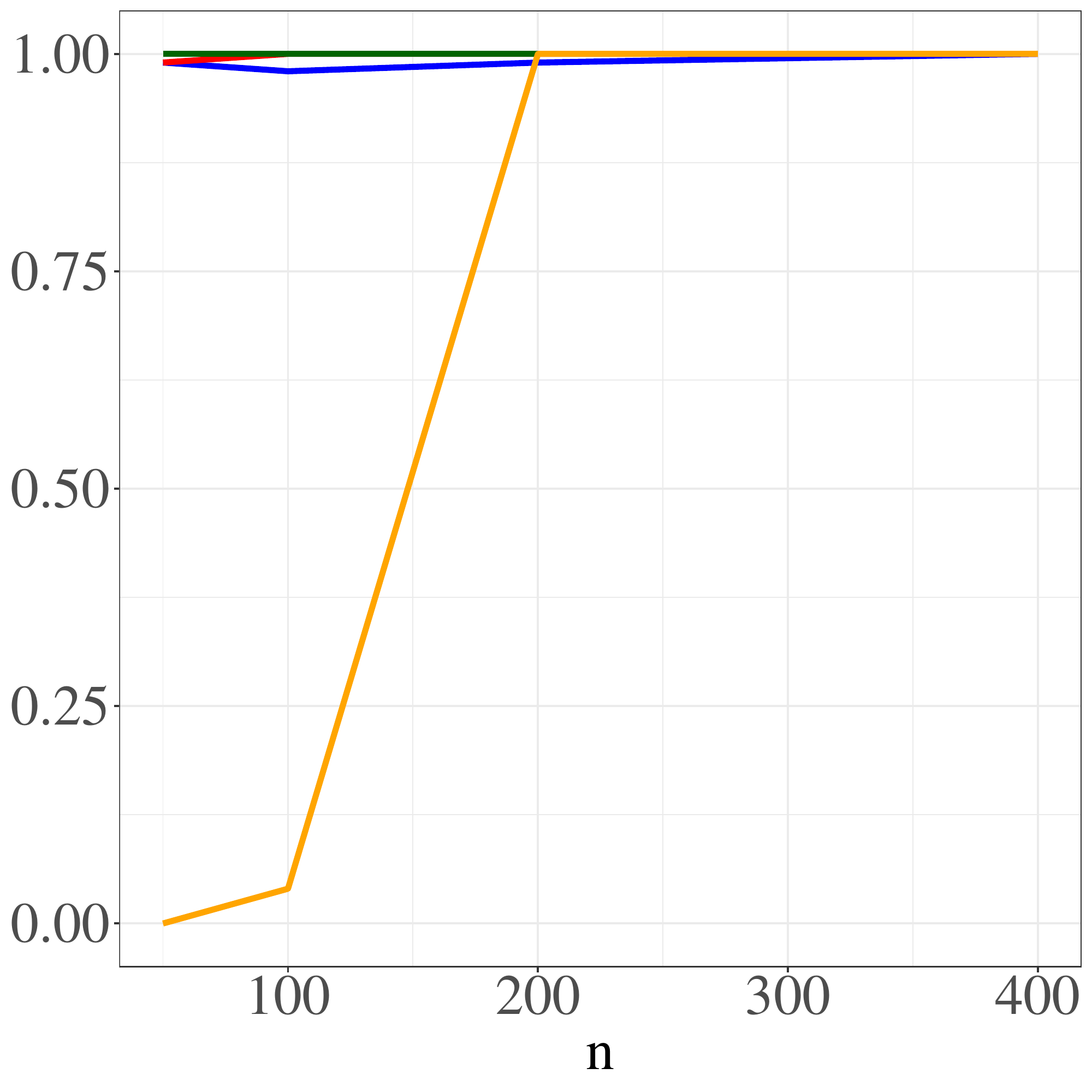}
\end{subfigure}

\vskip\baselineskip
\begin{subfigure}[b]{0.3\textwidth}
     \caption[]
       {{\footnotesize $\beta_{0i} = 1.5, \mu_{0} = - \sqrt{\log n}$}}
 	\centering
	\includegraphics[width=1.1\textwidth, height = \textwidth]{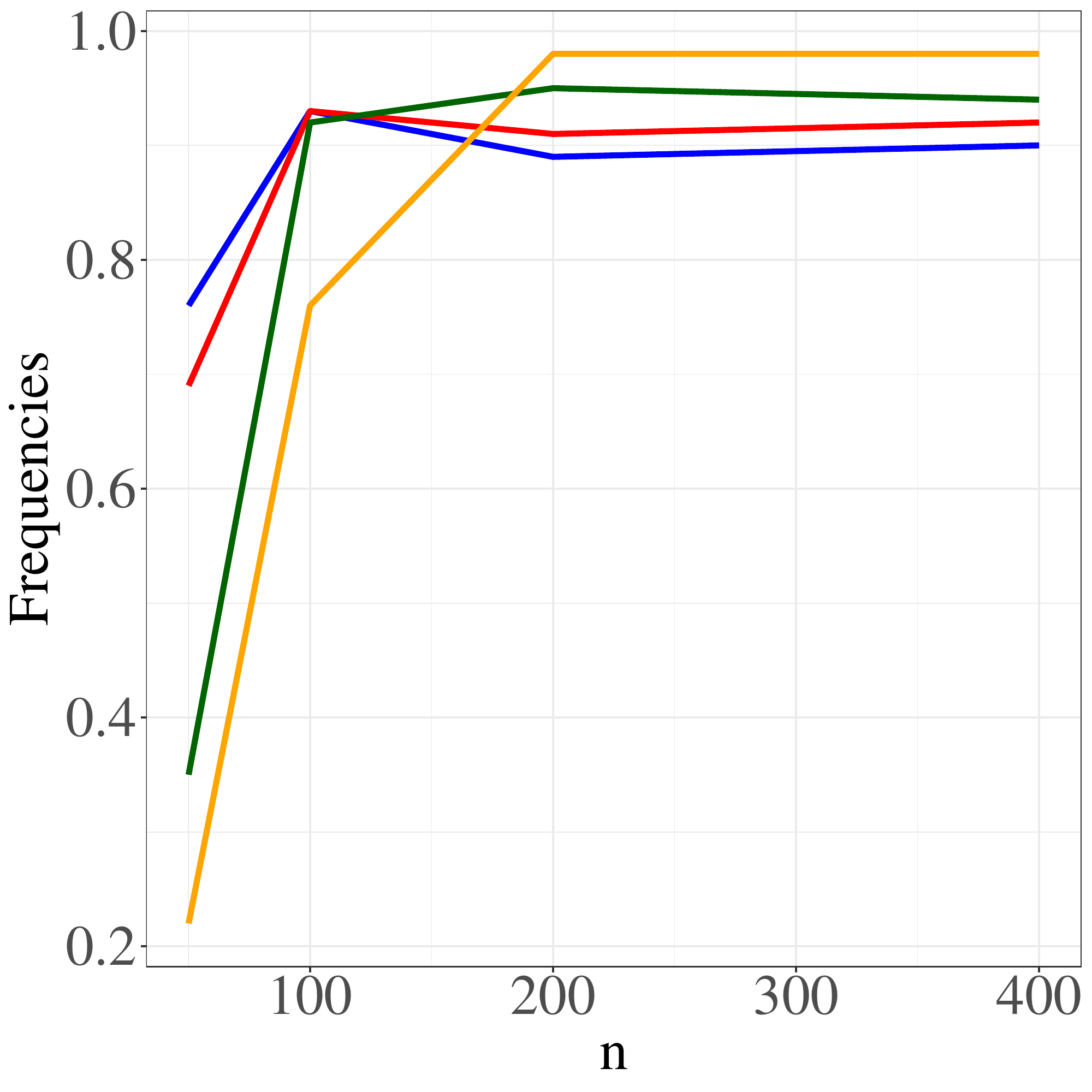}
\end{subfigure}
 \hfill \hfill \hfill \hspace{-5mm}
\begin{subfigure}[b]{0.3\textwidth}
   \caption[]
       {{\footnotesize $\beta_{0i} = \sqrt{\log n}, \mu_{0} = - \sqrt{\log n}$}}
 	\centering
	\includegraphics[width=\textwidth]{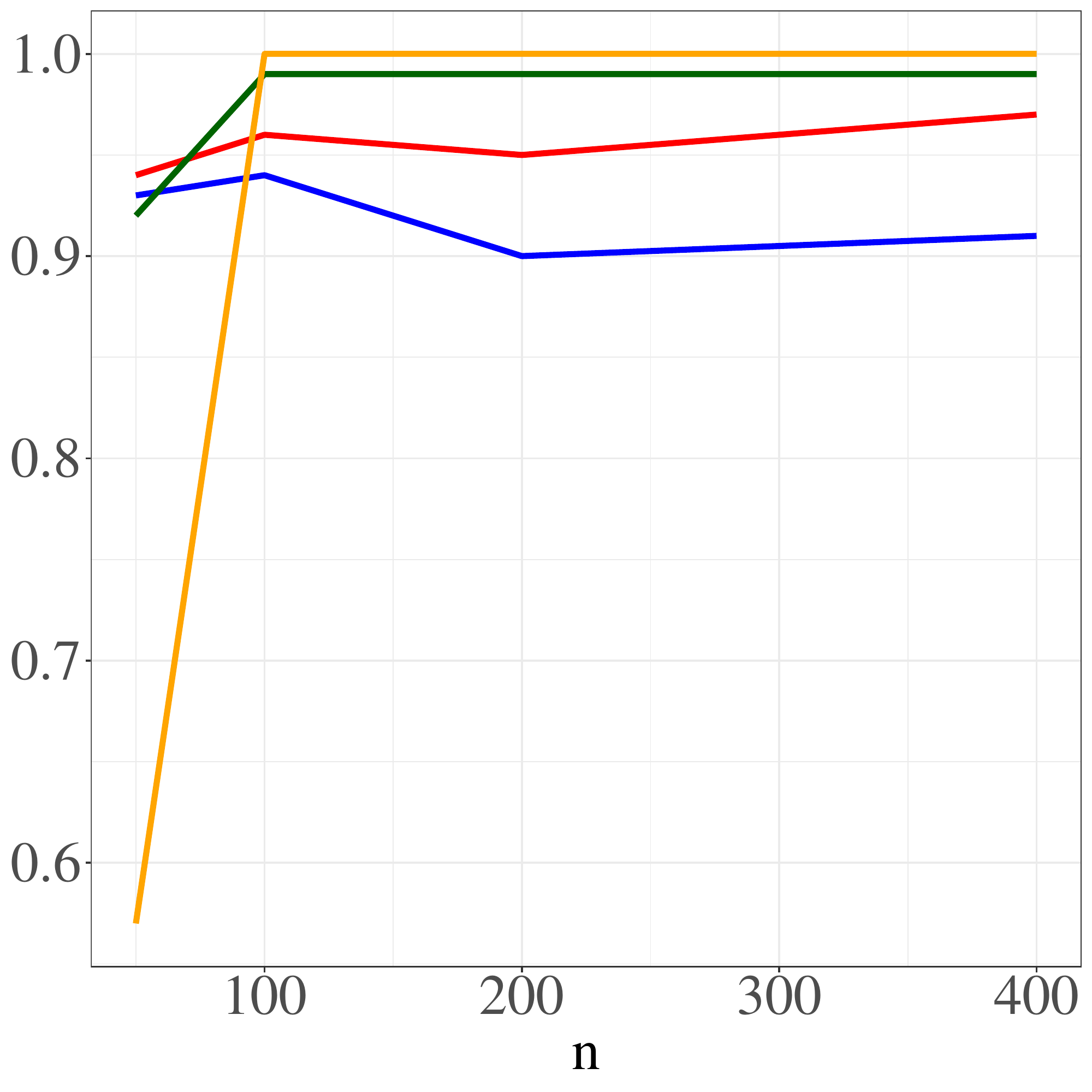}
\end{subfigure}
\hfill
\begin{subfigure}[b]{0.3\textwidth}
    \caption[]
      {{\footnotesize $\beta_{0i} = \log n, \mu_{0} = - \sqrt{\log n}$}}
 	\centering
	\includegraphics[width=\textwidth]{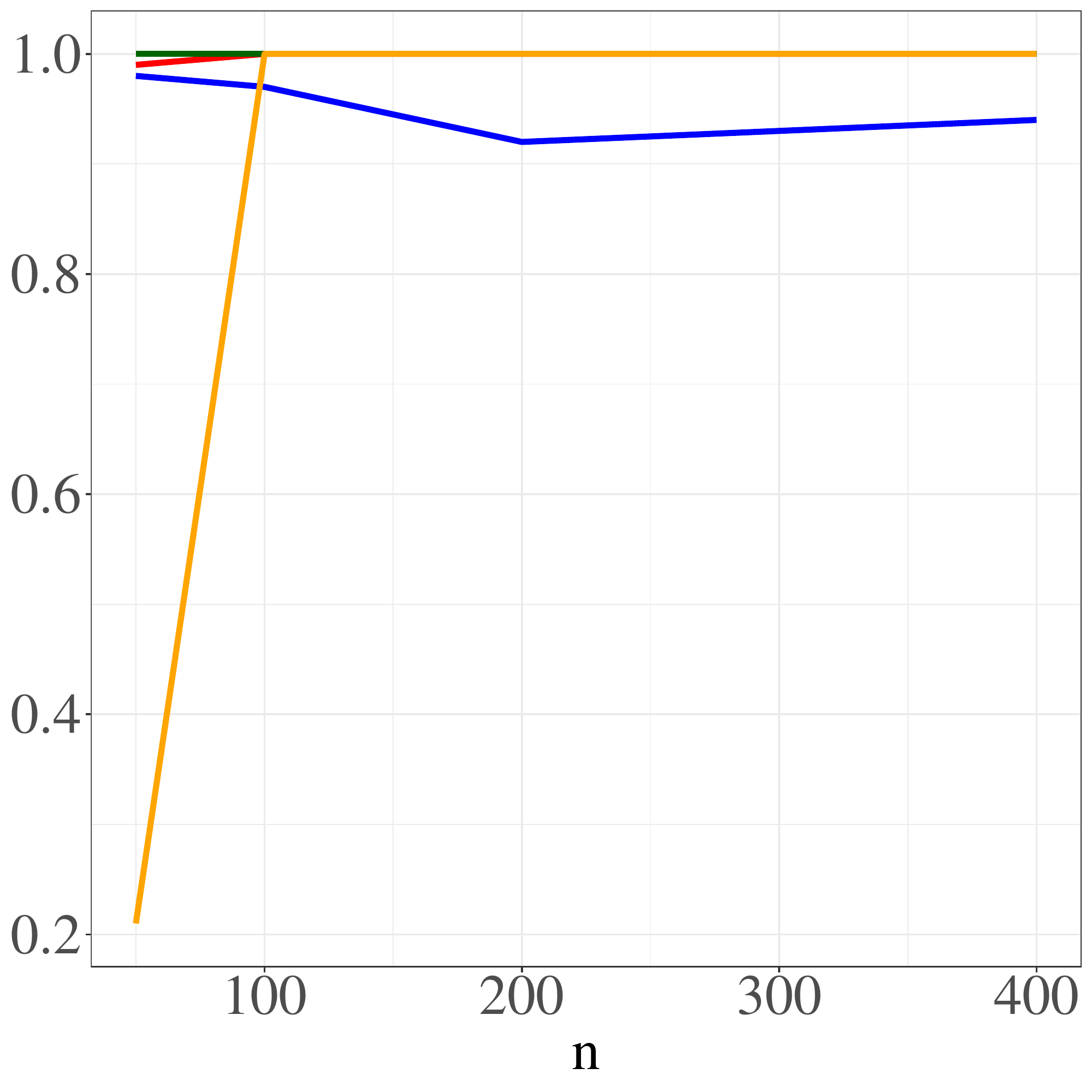}
\end{subfigure}

\vskip\baselineskip
\begin{subfigure}[b]{0.3\textwidth}
     \caption[]
       {{\footnotesize $\beta_{0i} = 1.5, \mu_{0} = - \log n$}}
 	\centering
	\includegraphics[width=1.1\textwidth, height = \textwidth]{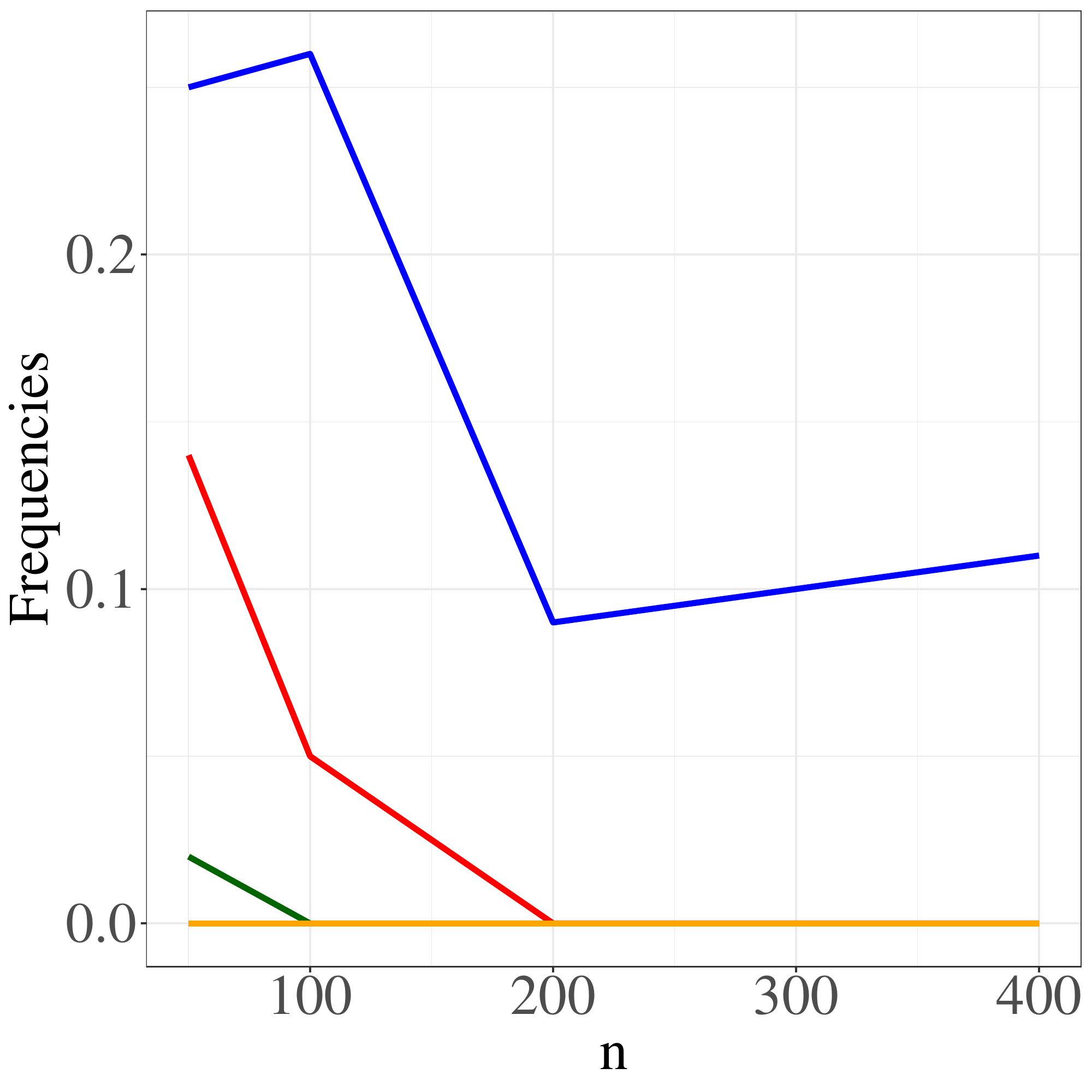}
\end{subfigure}
 \hfill \hfill \hfill \hspace{-5mm}
\begin{subfigure}[b]{0.3\textwidth}
   \caption[]
      {{\footnotesize $\beta_{0i} = \sqrt{\log n}, \mu_{0} = - \log n$}}
 	\centering
	\includegraphics[width=\textwidth]{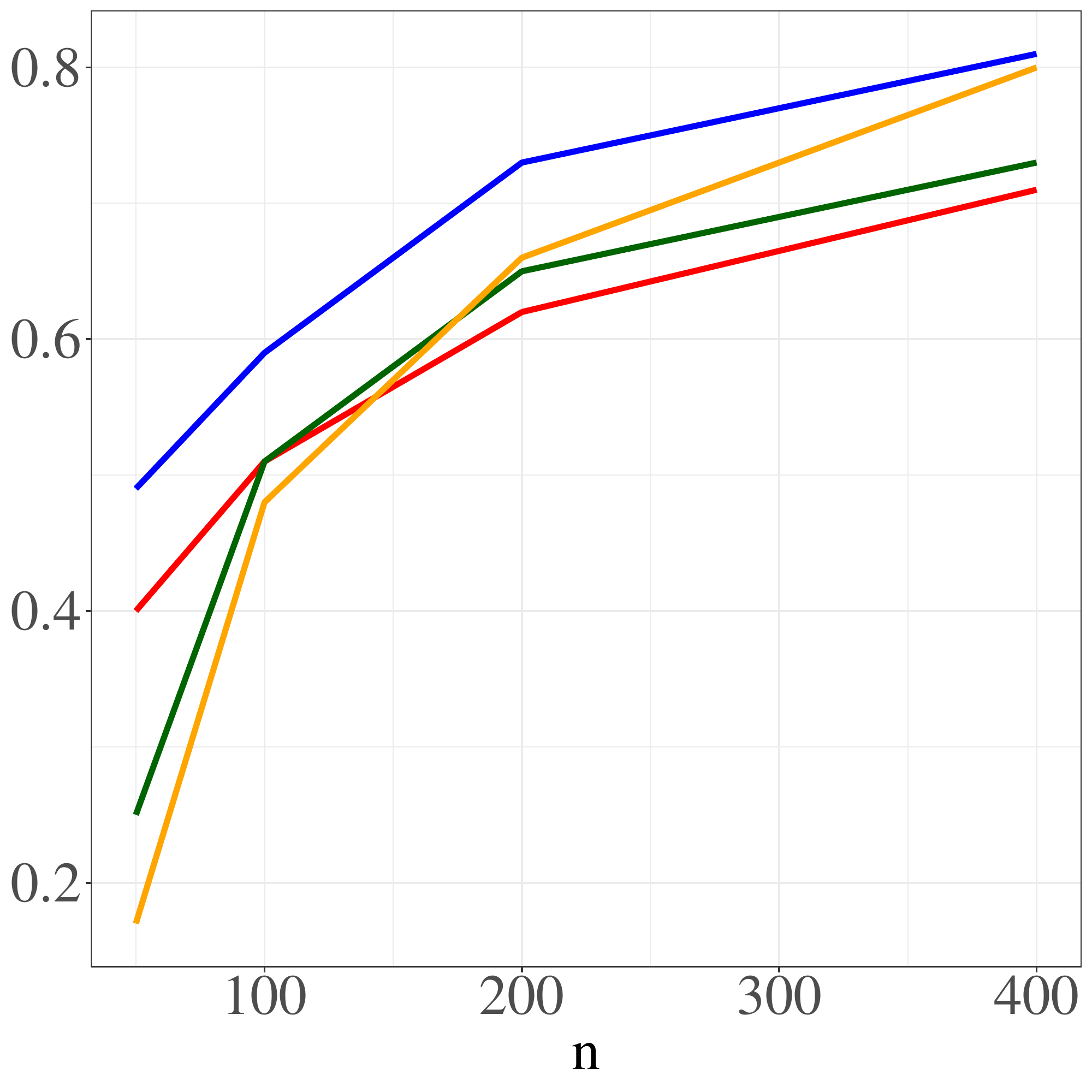}
\end{subfigure}
\hfill
\begin{subfigure}[b]{0.3\textwidth}
    \caption[]
       {{\footnotesize $\beta_{0i} = \log n, \mu_{0} = - \log n$}}
 	\centering
	\includegraphics[width=\textwidth]{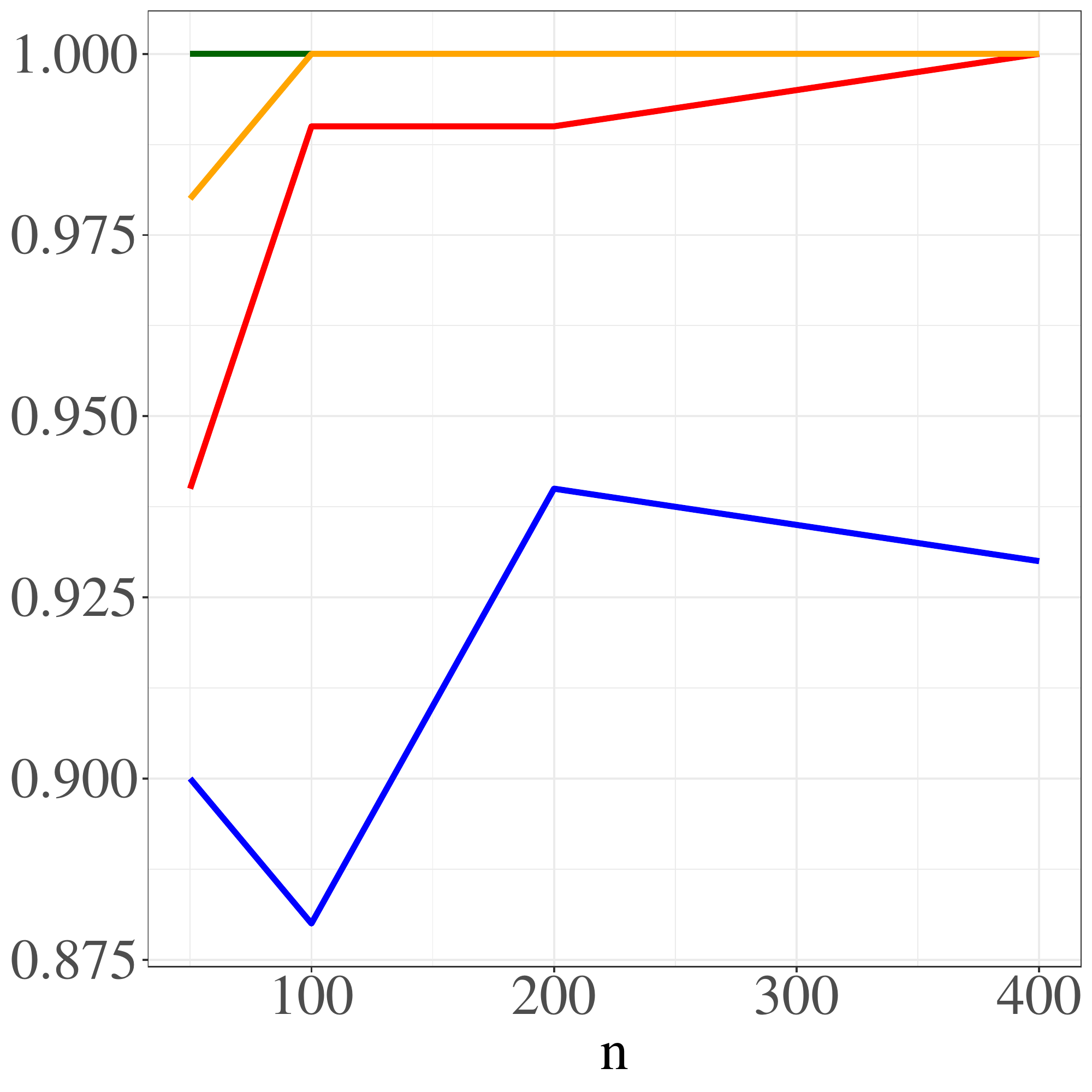}
\end{subfigure}

 \caption[ ]
        {\small Simulation results on frequencies of the true support selected by BIC. \legendsquare{blue}~$s_0 = 2$,
  \legendsquare{red}~$s_0 = \lfloor \sqrt{n/2} \rfloor$,
  \legendsquare{darkgreen}~$s_0 = \left \lfloor \sqrt{n}\right \rfloor$, \legendsquare{orange}~$s_0 = \left \lfloor 2\sqrt{n}\right \rfloor$.  }        
         \label{Figure-Frequencies-of-Selection}
\end{minipage}
\end{figure}

\begin{figure}[p]
\begin{minipage}{\textwidth}
\centering

\begin{subfigure}[b]{0.3\textwidth}
     \caption[ ]
        {{\footnotesize $\beta_{0i} = 1.5, \mu_{0} = - 1.5$}}
 	\centering
	\includegraphics[width=1.1\textwidth, height = \textwidth]{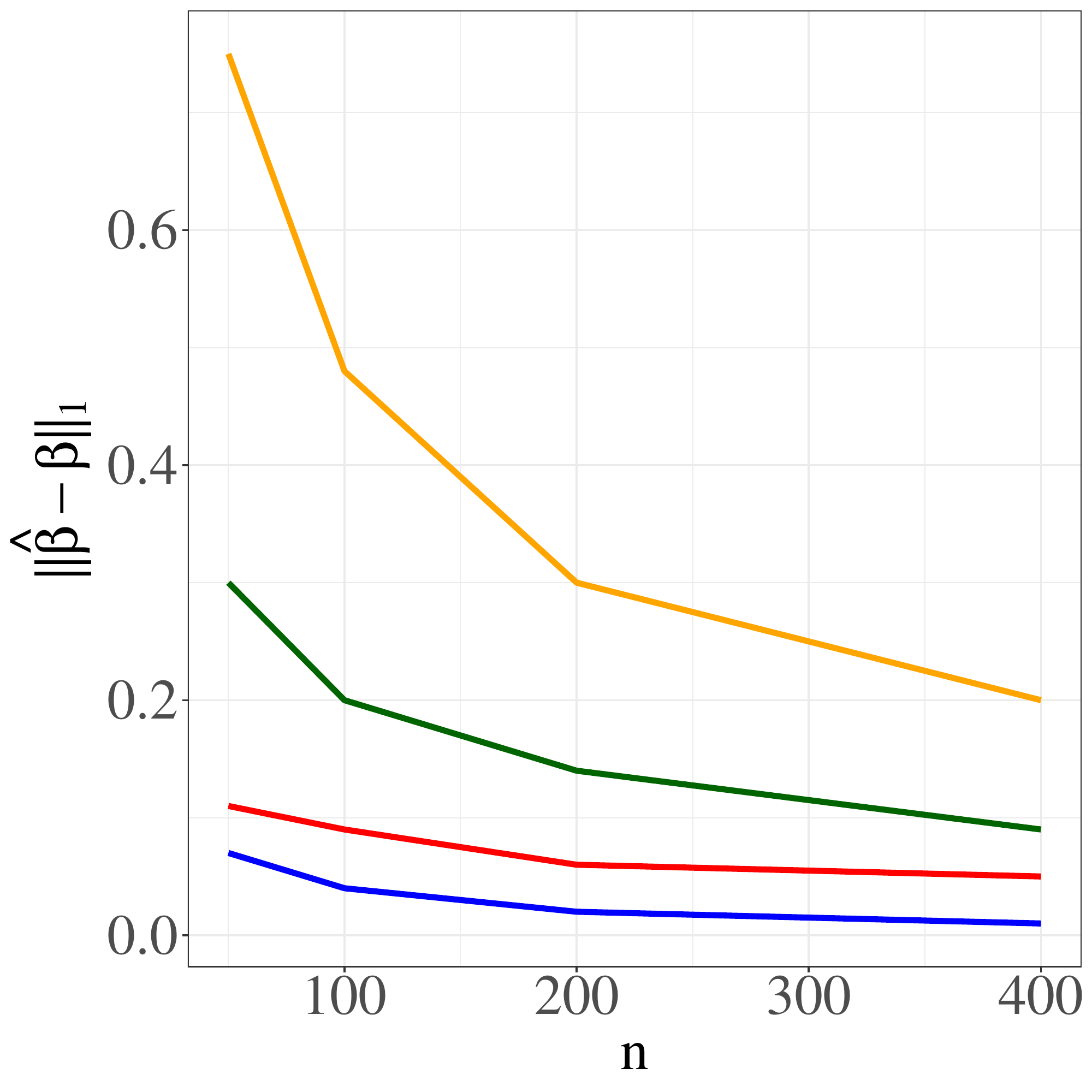}
\end{subfigure}
 \hfill \hfill \hfill \hspace{-5mm}
\begin{subfigure}[b]{0.3\textwidth}
   \caption[]
        {{\footnotesize $\beta_{0i} = \sqrt{\log n}, \mu_{0} = -1.5$}}
 	\centering
	\includegraphics[width=\textwidth]{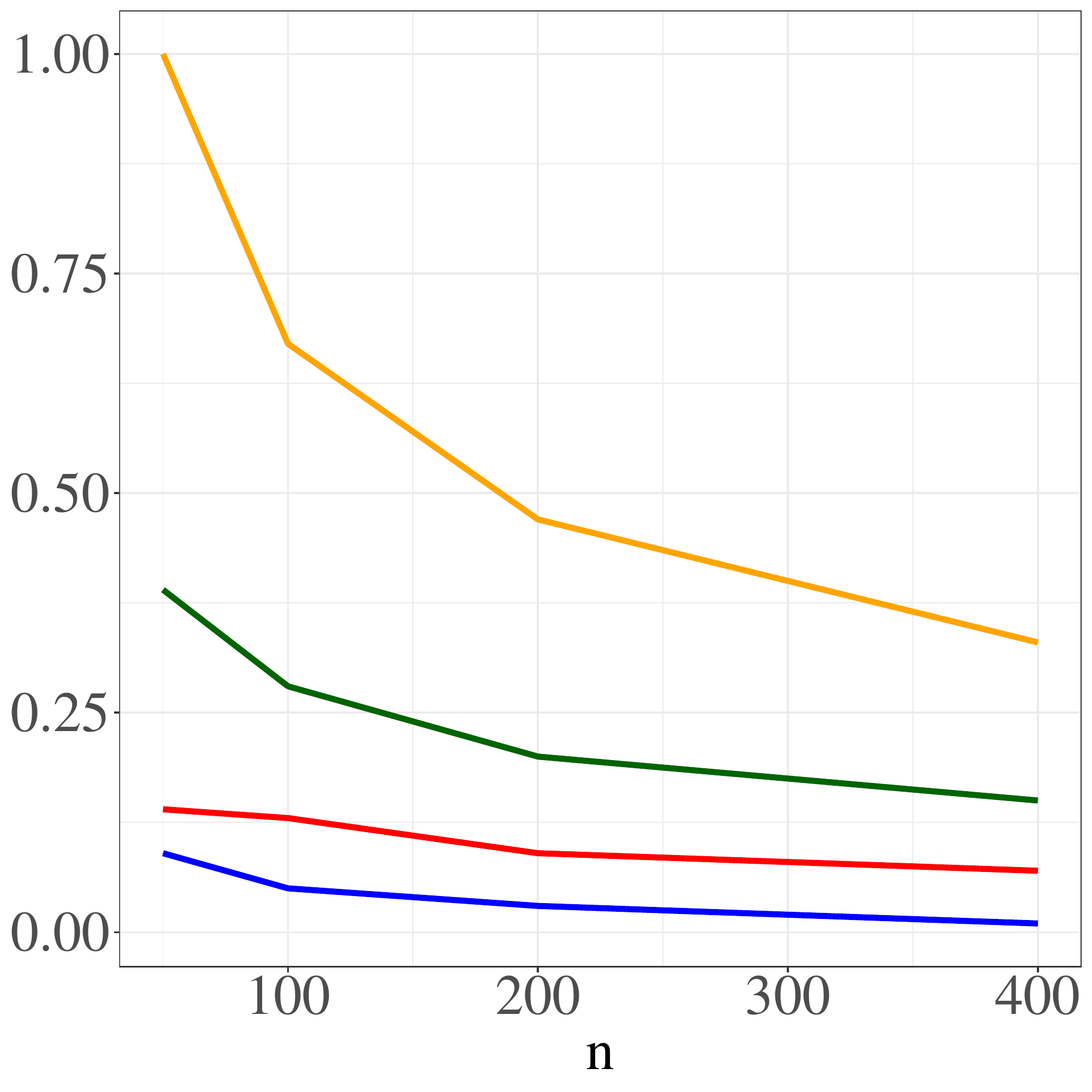}
\end{subfigure}
\hfill
\begin{subfigure}[b]{0.3\textwidth}
   \caption[]
        {{\footnotesize $\beta_{0i} = \log n, \mu_{0} = - 1.5$}}
 	\centering
	\includegraphics[width=\textwidth]{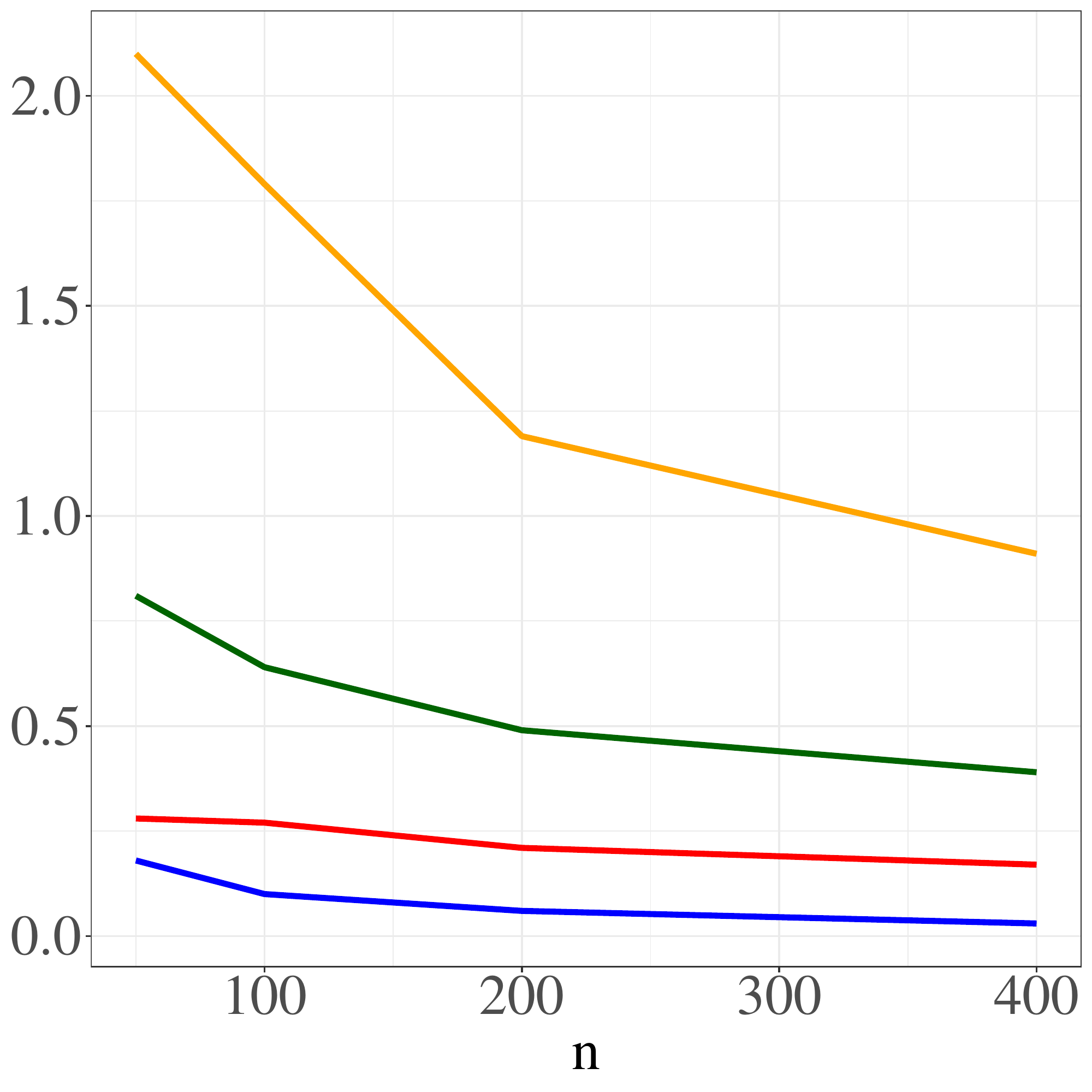}
\end{subfigure}

\vskip\baselineskip
\begin{subfigure}[b]{0.3\textwidth}
     \caption[]
        {{\footnotesize $\beta_{0i} = 1.5, \mu_{0} = - \sqrt{\log n}$}}
 	\centering
	\includegraphics[width=1.1\textwidth, height = \textwidth]{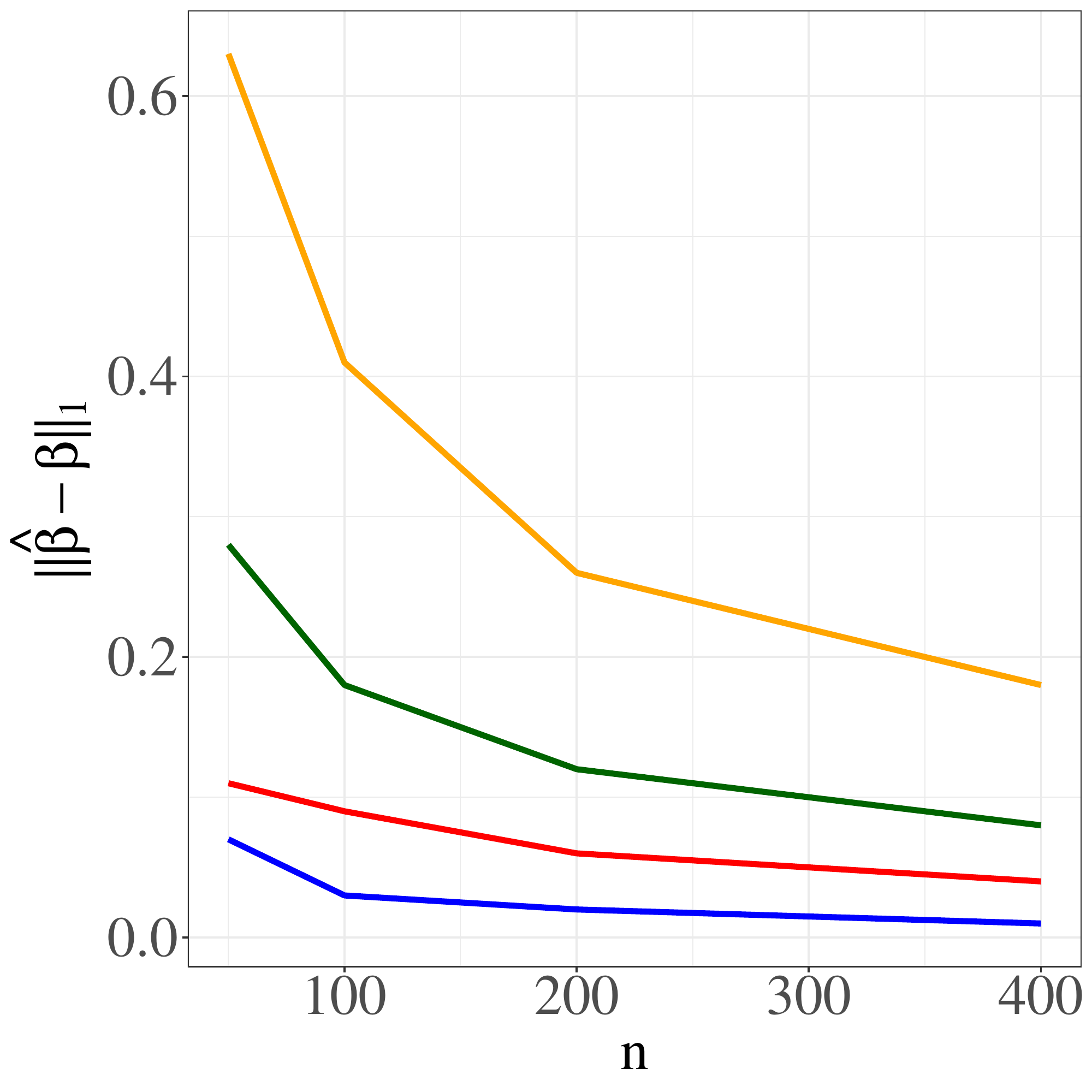}
\end{subfigure}
 \hfill \hfill \hfill \hspace{-5mm}
\begin{subfigure}[b]{0.3\textwidth}
   \caption[]
        {{\footnotesize $\beta_{0i} = \sqrt{\log n}, \mu_{0} = - \sqrt{\log n}$}}
 	\centering
	\includegraphics[width=\textwidth]{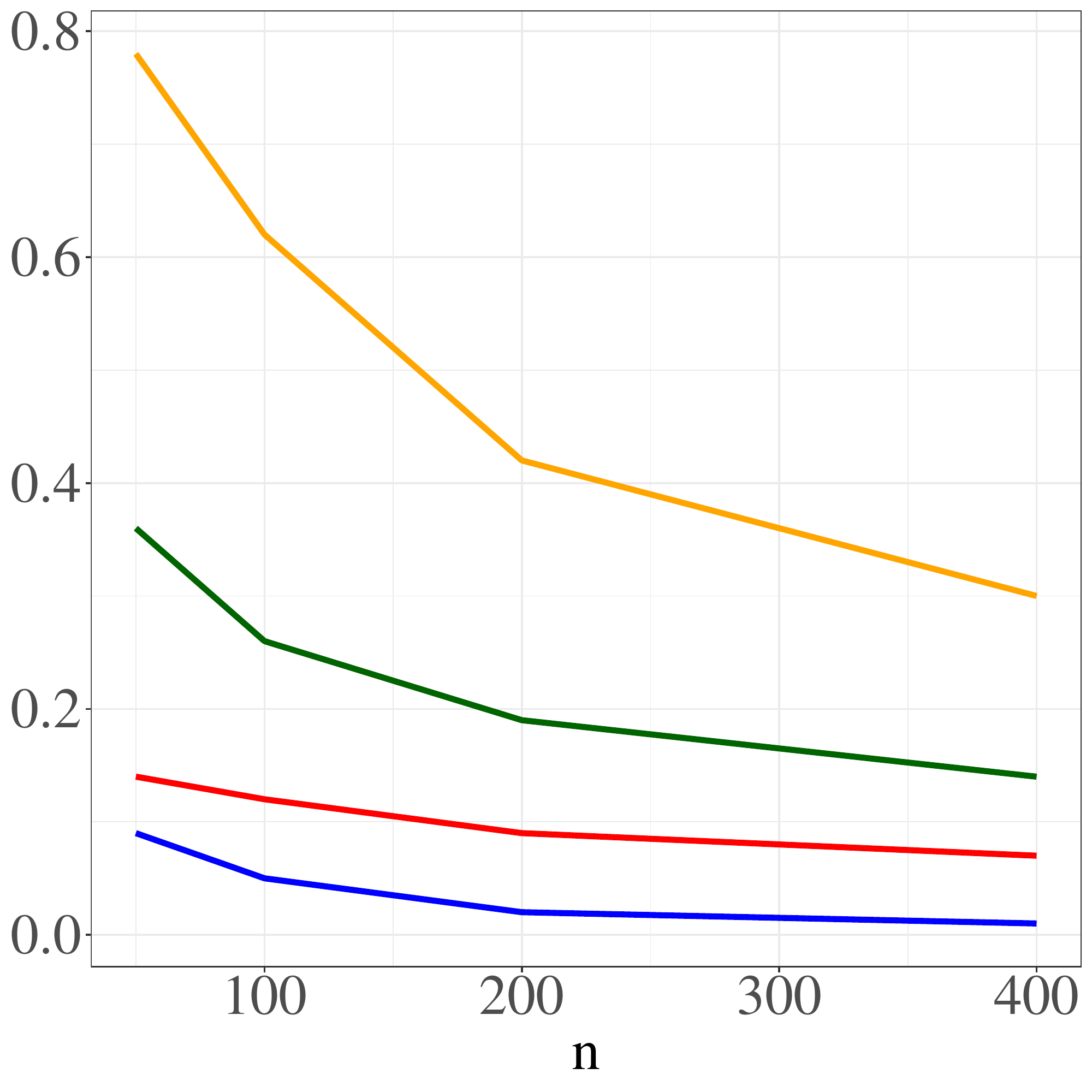}
\end{subfigure}
\hfill
\begin{subfigure}[b]{0.3\textwidth}
    \caption[]
        {{\footnotesize $\beta_{0i} = \log n, \mu_{0} = - \sqrt{\log n}$}}
 	\centering
	\includegraphics[width=\textwidth]{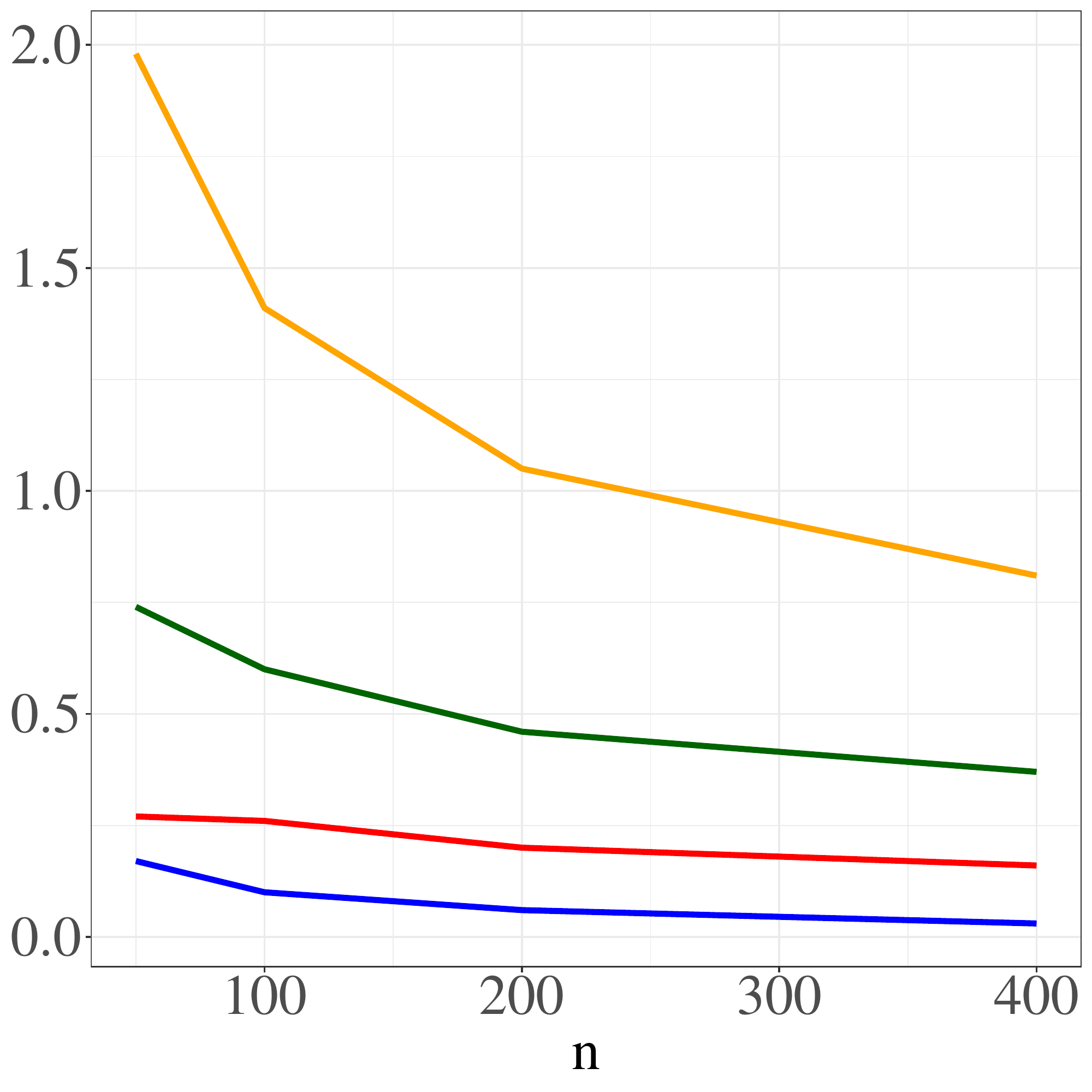}
\end{subfigure}

\vskip\baselineskip
\begin{subfigure}[b]{0.3\textwidth}
     \caption[]
        {{\footnotesize $\beta_{0i} = 1.5, \mu_{0} = - \log n$}}
 	\centering
	\includegraphics[width=1.1\textwidth,  height = \textwidth]{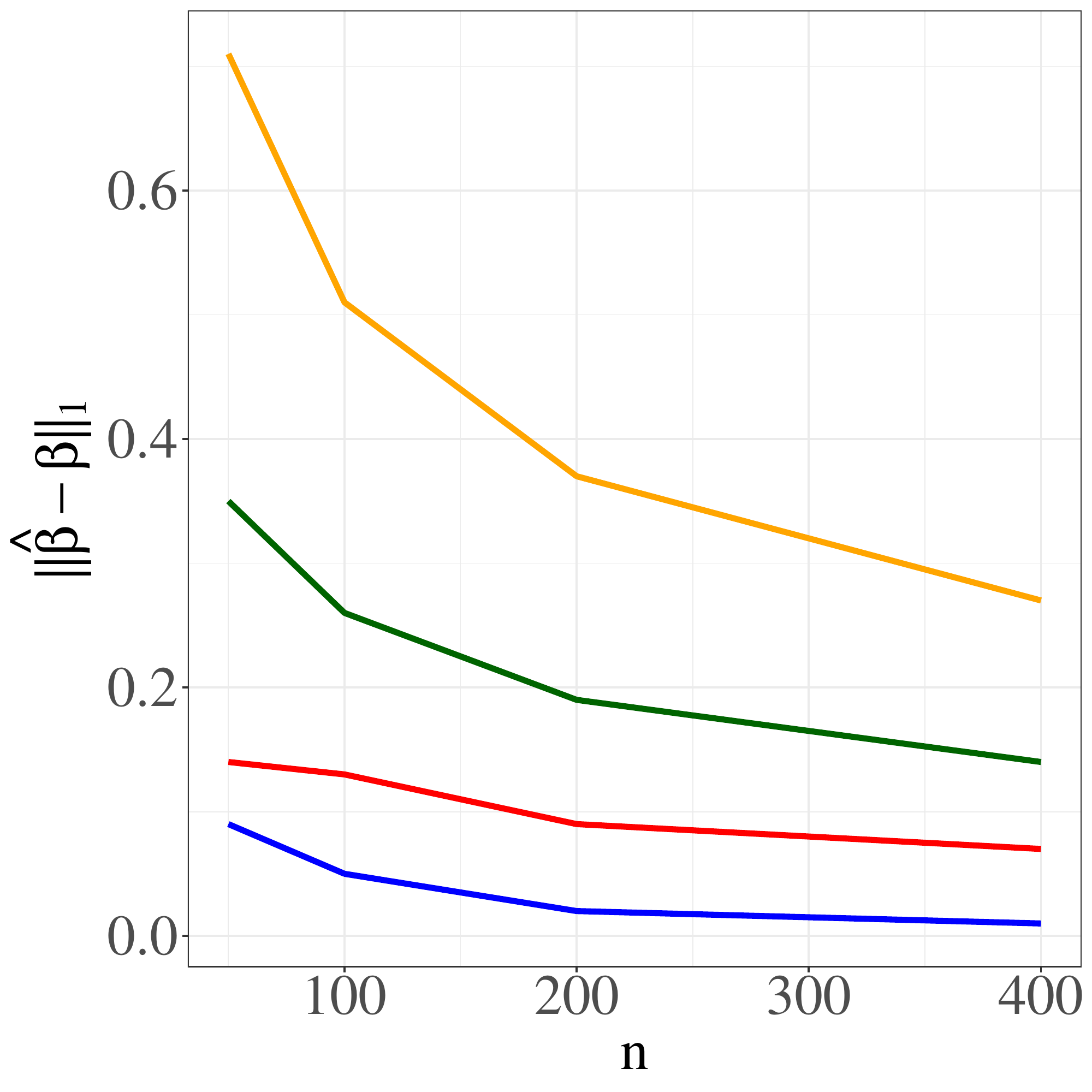}
\end{subfigure}
 \hfill \hfill \hfill \hspace{-5mm}
\begin{subfigure}[b]{0.3\textwidth}
   \caption[]
        {{\footnotesize $\beta_{0i} = \sqrt{\log n}, \mu_{0} = - \log n$}}
 	\centering
	\includegraphics[width=\textwidth]{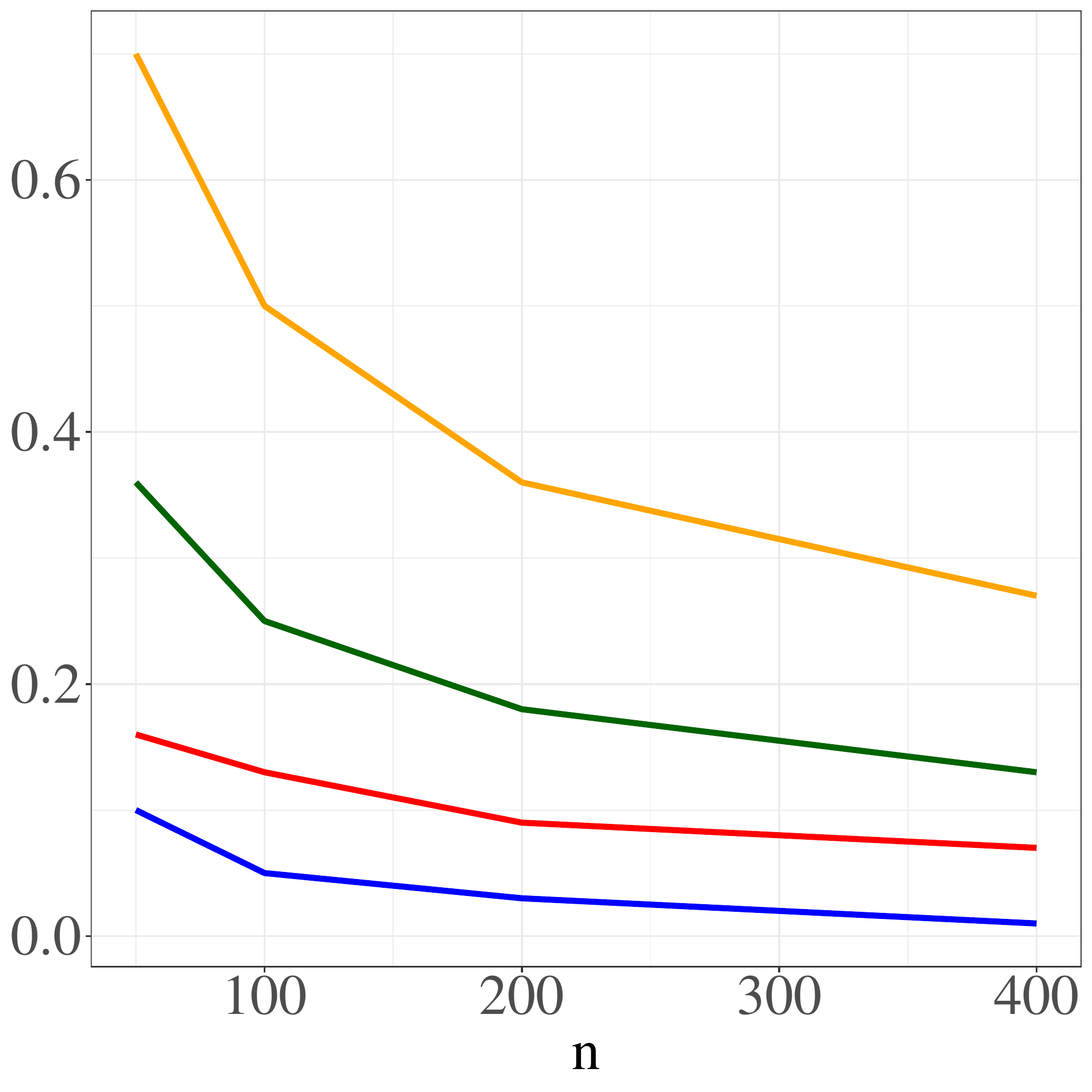}
\end{subfigure}
\hfill
\begin{subfigure}[b]{0.3\textwidth}
    \caption[]
        {{\footnotesize $\beta_{0i} = \log n, \mu_{0} = - \log n$}}
 	\centering
	\includegraphics[width=\textwidth]{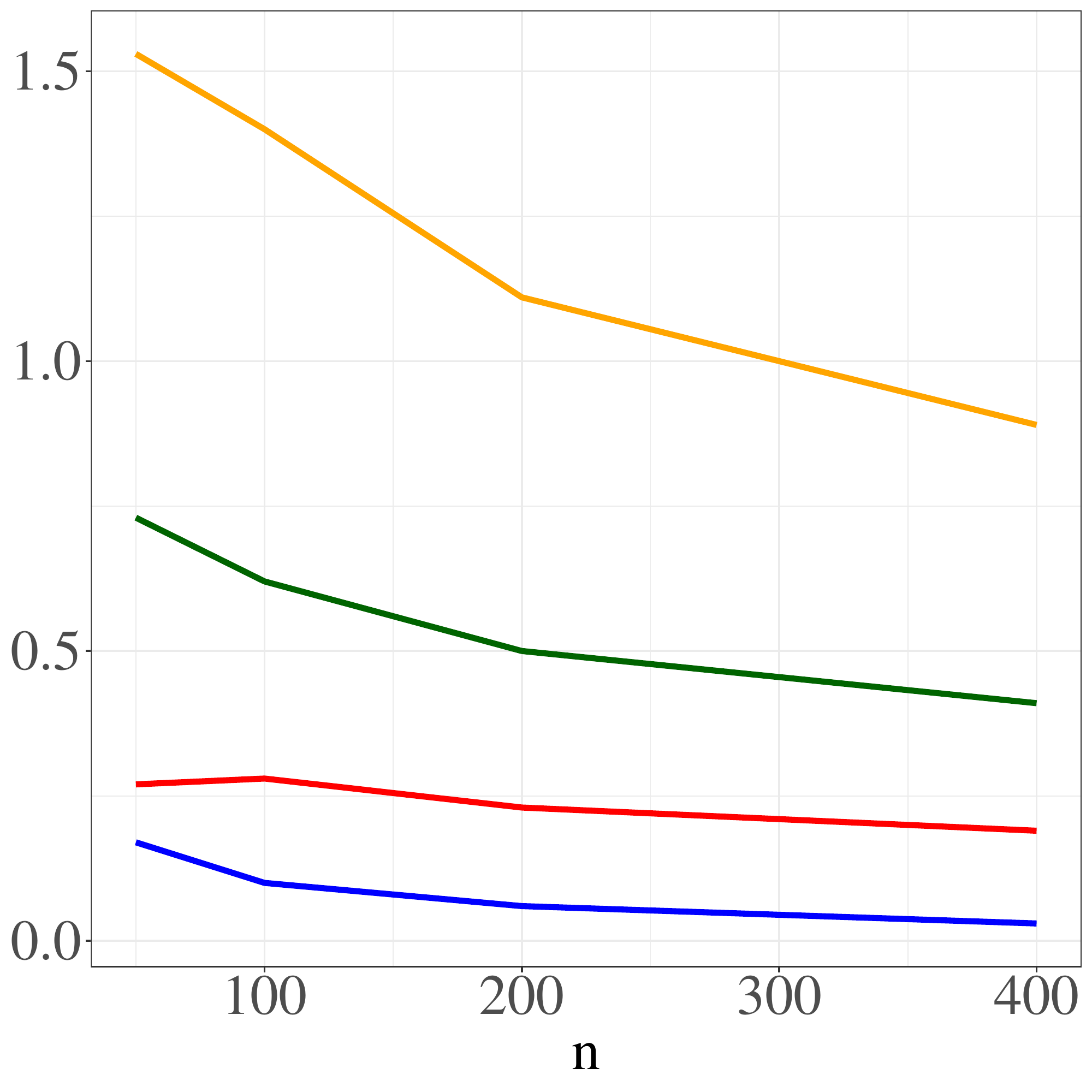}
\end{subfigure}

 \caption[ ]
      {\footnotesize Simulation results on the $\ell_1$-norm of $\hat{\vbeta}(\hat{s}) - \vbeta_{0}$ with $\hat{s}$ selected by BIC. \ \legendsquare{blue}~$s_0 = 2$,
  \legendsquare{red}~$s_0 = \lfloor \sqrt{n/2} \rfloor$,
  \legendsquare{darkgreen}~$s_0 = \left \lfloor \sqrt{n}\right \rfloor$, \legendsquare{orange}~$s_0 = \left \lfloor 2\sqrt{n}\right \rfloor$.    }        
         \label{Figure-Norm-of-Beta}
\end{minipage}
\end{figure}  

Figure \ref{Figure-Frequencies-of-Selection} reports the frequencies when the  support of $\vbeta_{0}$ is correctly identified for different settings (Figure \ref{Figure-Number-of-Regressors} in Appendix \ref{Appendix-Section-Simulation} in the supplementary material provides the results on the average number of nonzero $\beta$  selected based on our procedure in comparison to $s_0$). Figure \ref{Figure-Norm-of-Beta} reports the simulation results on the $\ell_{1}$-norm of $\hat{\vbeta}(\hat{s})-\vbeta_{0}$   (Figure \ref{Figure-Norm-of-Mu} in Appendix \ref{Appendix-Section-Simulation} in the supplementary material reports the simulation results on $| \hat{\mu}(\hat{s}) - \mu_{0}|$).

Figure \ref{Figure-Frequencies-of-Selection} shows that, in general, our estimation procedure works well in terms of model selection. In most  cases, as $n$ increases, BIC tends to correctly identify the support of $\vbeta_0$. The only exception is the case when $\mu_0 = -\log n$ and $\beta_{0i}$ is of lower magnitude as compared to $\mu_0$, especially for the case when $\beta_{0i} = 1.5$. In that case, the model selection results are worse than the other cases when the magnitude of $\mu_0$ is smaller than that of $\beta_{0i}$. This is partly because the smaller the local parameter $\beta_{0i}$ is, the harder it is to distinguish it from noise.  
When $\mu_0$ is less negative, i.e., when the network is globally denser, BIC has better model selection results at all levels of $\beta_{0i}$. 
Figure \ref{Figure-Frequencies-of-Selection} also shows that, given $\mu_{0}$, the larger the magnitude of $\beta_{0i}$ and the larger heterogeneity (more neighbors for nonzero $\beta$'s) are,  the better model selection results are.  
Figure \ref{Figure-Norm-of-Beta} shows that the
the estimation accuracy of $\hat{\vbeta}(\hat{s})$ generally improves as $n$ increases while it worsens as $s_0$ increases, reflecting the difficulty of estimating more parameters. Additional simulation results concerning the difference between $\hat{s}$ and $s_{0}$,  and the estimation  accuracy of $\hat{\mu}(\hat{s})$, found in Appendix \ref{Appendix-Section-Simulation},  support our findings.

From these simulation results, we may conclude that our estimation procedure in practice works well for a wide variety of networks with varying degrees of local density and global sparsity. {At the same time, our estimation procedure works well even when there are $O(\sqrt{n})$ nonzero true $\beta$ parameters (i.e., $s_0 = O(\sqrt{n})$), allowing for much more heterogeneity than the Erd\H{o}s-R\'{e}nyi model. These observations support our motivating discussion made in Section \ref{sec:SbetaM}  
on the S$\beta$M in comparison with the Erd\H{o}s-R\'{e}nyi model and the $\beta$-model.}

\section{Data Analysis}\label{Sec:DA}
 
 In this section, we analyze the microfinance take-up example in \cite{Banerjee:etal:2013} to illustrate the usefulness of our model and estimation procedure. \cite{Banerjee:etal:2013} investigated the role of social networks, especially the role of those pre-identified as ``leaders" (e.g., teachers, shopkeepers, savings group leaders), on households microfinancing decisions, and modelled the microfinancing decisions using a logit model.  

\paragraph*{Data.} In 2006, data were collected for $75$ rural villages in Karnataka, a state in southern India. A census of households was conducted, and a subset of individuals was asked detailed questions about the relationships they had with others in the village. This information was used to create network graphs for each village. 

The social network data were collected along $12$ dimensions in terms of whether individuals borrowed money from, gave advice to, helped with a decision, borrowed kerosene or rice from, lent kerosene or rice to, lent money to, obtained medical advice from, engaged socially with, were related to, went to temple with, invited to one's home,  or visited another's home. A relationship between households exists if any household member indicated a relationship with members from the other household.  It should be noted that no relation exists between any two households in different villages due to the nature of data collection.

In 2007, a microfinancing institution, Bharatha Swamukti
Samsthe (BSS), began operations in these villages, and collected data on the households who participated in
the microfinancing program. For the $43$ villages that BSS entered, the total number of households is $9598$; the average number of households for each village is $223$ with a standard deviation of $56$; and the average take-up rate for BSS is $18\%$, with a cross-village standard deviation of $8\%$. 

For the data used in our paper, we considered the network in which two households are linked if and only if any of the $12$ dimensions of social contact occurred between them.  The adjacency matrix of this network is block diagonal as there exists no link between any two households in different villages. All the data sets are available through the Harvard Dataverse Network \url{https://dataverse.harvard.edu/dataset.xhtml?persistentId=hdl:1902.1/21538}. 
%The data sets used include "adj_allVillageRelationships_HH_vilno" and "MF". 

\paragraph*{Method.}

Following \cite{Banerjee:etal:2013}, we study how social importance, in a network sense, will affect microfinance take-up decision. Building on the S$\beta$M, we identified ``leaders" as those households whose $\beta$ parameters are estimated as nonzero. 

Since households in different villages are not connected, we allowed village-dependent $\mu$ parameters to capture the individual village effects in fitting the S$\beta$M. More precisely, for each village, we first fitted the S$\beta$M to the observed network of the village by choosing $s$ via BIC. {In this data analysis, we examined $s \leq n_m/2$ for village $m$ ($n_m$ is the sample size of village $m$), i.e., the maximum value of $s$ examined is half the sample size of the village.}  Having obtained the parameter estimates denoted as $\hat{\vbeta}_m$ and $\hat{\mu}_m$ for village $m$, 
 we used $\hat{\beta}^{\star}_{mi} =  \hat{\beta}_{mi} + \hat{\mu}_m/2$ as a covariate for household $i$ in village $m$ to model 
 the probability that this household participated in the microfinancing program
\begin{equation}
\label{eq:parti}
\mathrm{Parti}_{mi} =  \Lambda (c + \theta \cdot \hat{\beta}^{\star}_{mi} ),
\end{equation}
where $\Lambda$ is the logistic function such that $\Lambda(a)=\log (a)/ \log(1-a)$ for $a\in(0,1)$, and $c \in \mR$ and $\theta \in \mR$ are two unknown parameters. The role of these estimated $\beta$'s will be referred to as $\beta$-centrality hereafter. As an alternative measure of leadership, we also examined the use of an  indicator variable ``Leader", defined as $\mathrm{Leader}_{mi} := 1\{\hat{\beta}_{mi} >0\}$. Below we suppress the dependence of $\hat{\vbeta}$ and $\hat\mu$ on $m$ for simplicity.

For comparison,  degree centrality and eigenvector centrality, two widely used measures of the influence of a node, were also investigated. In the context of the data analysis, the degree centrality of  household $i$ is $d_{i}$, the number of links that this household has. This is a measure of how well-connected a household is in the network.  In graph theory, eigenvector centrality is a recursively defined notion of importance by associating high scores to those nodes that are connected to high-scoring nodes. Mathematically, the eigenvector centrality of the $i$th household is the $i$th element of $\bm{x}$, where $\bm{x}=(x_{1}, \dots, x_n)^T$ is the nonnegative eigenvector associated with the largest eigenvalue of the adjacency matrix, normalized to have Euclidean norm  $n$. Considering various combinations of these measures of influence, we examine the following models: 
\[ (1)~\mathrm{Parti}_i = \Lambda (c + \theta \cdot d_{i} ); 
(2)~\mathrm{Parti}_i  = \Lambda (c + \theta \cdot x_{i} );
(3)~\mathrm{Parti}_i  = \Lambda (c + \theta \cdot \hat{\beta}^{\star}_i  ); \]
\[(4)~\mathrm{Parti}_i = \Lambda (c + \theta \cdot (1\{\hat{\beta}_i>0 \} +\hat{\mu}/2) );\]
\[(5)~\mathrm{Parti}_i = \Lambda (c + \theta_1 \cdot d_i + \theta_2  \cdot \hat{\beta}^{\star}_i  );
(6)~\mathrm{Parti}_i = \Lambda (c + \theta_1 \cdot d_i + \theta_2 \cdot (1\{\hat{\beta}_i>0 \} +\hat{\mu}/2) ); \]
\[(7)~\mathrm{Parti}_i = \Lambda (c + \theta_1 \cdot x_i + \theta_2  \cdot \hat{\beta}^{\star}_i ); 
(8)~\mathrm{Parti}_i = \Lambda (c +  \theta_1 \cdot x_i + \theta_2  \cdot (1\{\hat{\beta}_i>0 \} +\hat{\mu}/2));\]
where $c\in \mR, \theta \in \mR, \theta_1 \in \mR$ and $\theta_2 \in \mR$ are unknown parameters. Note that in these models, $\hat\mu$ can not be absorbed into $c$ because it is a parameter dependent on the village $m$. In examining these models, we wanted to assess the effects of different centralities in models (1)--(4), and to compare their relative merits when competing with each other in models (5)--(8). Finally the parameters in  models (1)--(8) were estimated via the method of maximum likelihood for a logistic regression model.

\paragraph*{Results.}
Using BIC, the S$\beta$M gave a fit with an average $26 \%$ of the households having nonzero $\beta$ parameter. To assess how the model fits the data graphically, in Figure \ref{fig:1all}, we plotted the empirical distribution of the degrees of the observed network (black solid points) and the degree distribution after fitting the S$\beta$M (red open dots). The latter was obtained by averaging the empirical degree distributions of $100$ randomly generated networks from the S$\beta$M with the estimated parameters. For reference, we  also included the degree distribution of the Erd\H{o}s-R\'enyi model fit.  It can be seen that the empirical degree distribution of the data  in the upper tail follows roughly a straight line, suggesting that a power law may be appropriate. However, the huge discrepancy  between the empirical distribution of the data degrees and the Erd\H{o}s-R\'enyi model fit (black dash dotted line) implies that the Erd\H{o}s-R\'enyi model does not fit the data. In contrast, the S$\beta$M fit tracks the empirical distribution of the data very closely in the upper tail, thus providing a much better fit to capture the heavy upper tail of the empirical degree distribution. Interestingly, it can also be seen that the S$\beta$M fit yields a Poisson curve for those points with small degrees  and a different pattern for those with larger degrees. This pattern can be loosely understood as the result of assigning nonzero $\beta$ parameters to the nodes with large degrees.  In this sense, the S$\beta$M fit mimics a mixture of the Erd\H{o}s-R\'enyi and $\beta$-models which echos our point made previously that the S$\beta$M interpolates these two.  We remark further that the presence of many isolated nodes and many nodes with a small number of links prevented fitting the $\beta$-model to the network.
 
 \begin{figure}[hbt!]
\centering
 \includegraphics[width=0.9\textwidth]{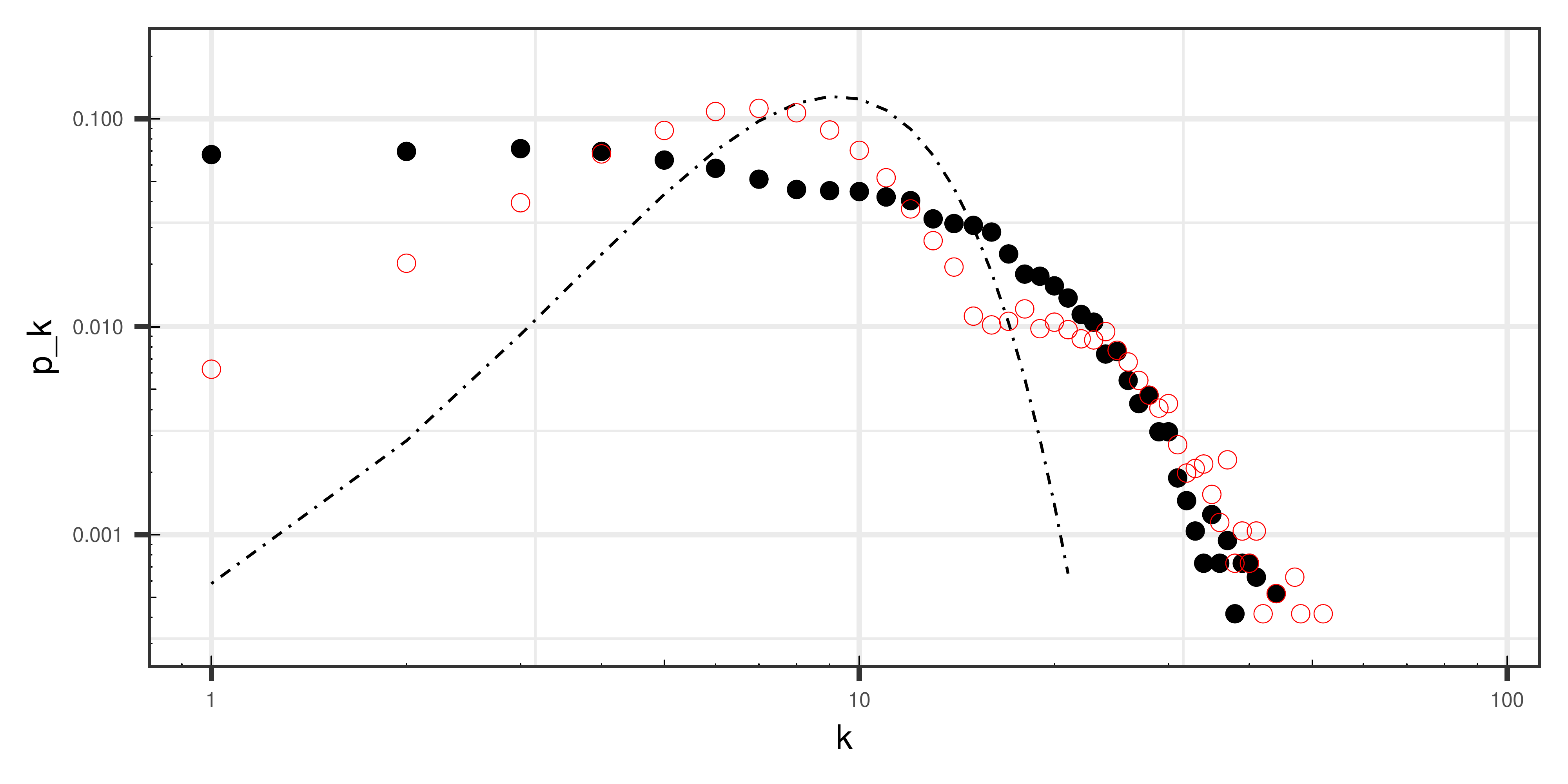}
\caption{\label{fig:1all} The black solid dots represent the degree distribution (frequency of degree denoted as $p_k$ versus degree $k$) on the log-log scale using observed data from 43 villages. The red open dots correspond to the averaged empirical degree distributions of $100$ randomly generated networks from the S$\beta$M with the estimated parameters. The fitted degree distribution assuming the frequencies follow a Poisson distribution is plotted as the black dash dotted  line.}
\end{figure}

Table \ref{Table: Degree or Beta Logit} provides the effects of degree centrality and eigenvector centrality on microfinance take-up, along with the effects of $\beta$-centrality and how being a ``Leader"  can influence take-up. From the results on models (1)--(4), we can see that the effect on microfinance participation is much higher when $\beta$ is larger (or when the node is identified as ``Leader") for the S$\beta$M. On the other hand, although the effect of degree centrality is also statistically significant, the magnitude is much smaller when compared with eigenvector centrality or $\beta$-centrality or being a ``Leader".

 \begin{table}[hbt!] \centering 
 \begin{threeparttable}
  \caption{Effect of Different Network Statistics on Take-Up. Standard errors are in parentheses.} 
  \label{Table: Degree or Beta Logit} 
\begin{tabular}{@{\extracolsep{\fill}}lcccccccc} 
\\[-1.8ex]\hline 
\hline \\[-1.8ex] 
 & \multicolumn{8}{c}{\textit{Dependent variable: take-up}} \\ 
\cline{2-9} 
 
\\[-1.8ex] & (1) & (2) & (3) & (4) & (5) & (6) & (7) & (8)\\ 
\hline \\[-1.8ex] 

Degree & 0.010$^{***}$ &  &  &  & $-$0.001 & $-$0.004 &  &  \\ 
  & (0.003) &  &  &  & (0.005) & (0.005) &  &  \\ 
  & & & & & & & & \\ 
 Eigenvector &  & 0.575$^{***}$ &  &  &  &  & 0.442$^{**}$ & 0.239 \\ 
  &  & (0.131) &  &  &  &  & (0.193) & (0.178) \\ 
  & & & & & & & & \\ 
 Beta &  &  & 0.198$^{***}$ &  & 0.212$^{**}$ &  & 0.071 &  \\ 
  &  &  & (0.052) &  & (0.084) &  & (0.076) &  \\ 
  & & & & & & & & \\ 
 Leader &  &  &  & 0.316$^{***}$ &  & 0.366$^{***}$ &  & 0.239$^{***}$ \\ 
  &  &  &  & (0.063) &  & (0.088) &  & (0.085) \\ 
  & & & & & & & & \\ 
 
\hline 
\hline \\[-1.8ex] 
 
\end{tabular} 
\begin{tablenotes}
      \small
       \item Note: $\ensuremath{^{*}}\ensuremath{p<0.1};\ensuremath{^{**}}\ensuremath{p<0.05};\ensuremath{^{***}}\ensuremath{p<0.01}$

    \end{tablenotes}
  \end{threeparttable}
\end{table}

Table \ref{Table: Degree or Beta Logit} also provides effects of degree centrality or eigenvector centrality on microfinance take-up when controlling for $\beta$-centrality or being a ``Leader".  The regression results show that, after controlling for $\beta$-centrality or being a ``Leader", the effects of degree centrality or eigenvector centrality are smaller, with the effect of degree centrality also becoming not statistically significant. Overall, we find the magnitude of $\beta$ is significantly related to eventual microfinance participation. In particular, whether the household
 plays a leader role in the village  is even more significantly related to eventual microfinance participation. Additional results  can be found in Appendix \ref{sec:addSimu} where  a probit link and an identity link in (\ref{eq:parti}) were used. All the additional results are consistent with the conclusions made from Table \ref{Table: Degree or Beta Logit}.

In our analysis, we did not distinguish causal or correlation effects in social networks. We note that other factors such as exogenous variation in the injection points could be useful for causal effects analysis. 
As the main objective of the current analysis is to provide insights on the role of social importance on program participation through the use of the S$\beta$M by defining new centrality measures such as $\beta$-centrality, we leave further investigation on dissecting causal and correlational effects or results from a structural economics model to future study.

\section{Conclusion}
\label{sec:conclusion}
We have proposed the Sparse $\beta$-Model (S$\beta$M) as a new generative model that can explicitly characterize global and local sparsity. We have shown that conventional asymptotic results including consistency and asymptotic normality  results of its MLE are readily available for a wide variety of networks that are dense or sparse, when the support of the parameter is known. When it is unknown, we have developed an $\ell_0$-norm penalized likelihood approach for estimating the parameters and their support. We overcome the seemingly combinatorial nature of the optimization algorithm for computing the penalized estimator by fitting at maximum $n-1$ nested models with their support read from the degree sequence, thanks to a novel monotonicity lemma used to develop the solution path. A sufficient condition on the signal strength which is referred to as the $\beta$-min condition guarantees that, with high probability, the S$\beta$M chooses the correct model along its solution path. Therefore, the S$\beta$M represents a new class of models that are computationally fast, theoretically tractable, and intuitively attractive.

The computational tractability of the $\ell_0$ penalized estimation approach for the S$\beta$M depends on the monotonicity lemma which exploits the unique feature of the $\beta$-model. There are several recent  generalizations of the $\beta$-model to which this lemma is not applicable. The first class of models is the $\beta$-model for directed graphs where, for each node,  incoming and  outgoing parameters are used for capturing the directional effect \citep{Holland:Leinhardt:1981}.  We can adopt a similar strategy assuming that these parameters are sparse possibly after a reparametrization as developed for the S$\beta$M. It is not difficult to see, however,  that the monotonicity lemma no longer holds. Therefore, when the support of these parameters is unknown, the $\ell_0$-penalty based estimator is no longer computationally feasible. In view of this, we may develop $\ell_1$-norm penalized likelihood estimation, immediately connecting this methodology to the vast literature on penalized likelihood methods for binary regression. A disadvantage of the $\ell_1$ penalized approach  is that the resulting estimators will be biased, sometimes substantially so if the amount of shrinkage needs to be excessive for very sparse models. Another future direction for research is to include covariate information at the nodal or link level. Progress has been made in this vein by extending the $\beta$-model \citep{Graham:2017} and its generalization to directed networks  \citep{Yan:etal:2018}. At the moment, however, these generalizations are not known to work for relatively sparse networks if the interest is on the node-specific parameters. The methodology proposed in this paper can be studied in this wider context. We note again that, where the support of the parameters is unknown, the $\ell_0$-penalty based estimation is no longer tractable. \cite{Stein:Leng:2020} reported encouraging preliminary results for the S$\beta$M with covariates using an $\ell_1$-norm based penalization method. The results for other future directions will be reported elsewhere.

\bibliographystyle{chicago}
\bibliography{SparseBeta}

\newpage
\appendix
\appendixpage

The supplementary material contains all the proofs, discussion on the existence of the $\ell_0$-constrained MLE, and additional simulation results. 

\section{Proofs}

{In what follows, all limits are taken as $n \to \infty$.}

\subsection{Proofs for Section \ref{sec:SbetaM}}
\begin{proof}[Proof of Proposition \ref{prop:erm}] 
Since $A_{ij}, 1\le i<j \le n$ are i.i.d. Bernoulli random variables with 
\[ 
E[A_{ij}]=p \quad\text{and}\quad  \Var(A_{ij}) = p(1-p),
\]
we have
\[
 E\left [\sum_{i<j} A_{ij}\right]=\binom{n}{2}p \quad\text{and}\quad \Var \left (\sum_{i<j} A_{ij} \right)=\binom{n}{2}p(1-p)\coloneqq s_n^2.
\]
We will prove that
\begin{equation}\label{eq:prop1:clt}
 \frac{\sum_{i<j} A_{ij}-\binom{n}{2}p}{\sqrt{\binom{n}{2}p(1-p)}} \stackrel{d}{\to} N(0,1)
\end{equation}
whenever $n^2 p(1-p) \to \infty$. To this end, it suffices to verify 
 following Lindeberg condition:
 \[ 
\forall \epsilon > 0, \quad  \frac{1}{s_n^2} \sum_{i<j} E[B_{ij}^2 I(|B_{ij}|\ge \epsilon s_n)]  \to 0, 
\]
where $B_{ij}=A_{ij}-E[A_{ij}]$. Since $|B_{ij}| \le 1$, the left hand side will be zero whenever $\epsilon s_n >1$, which is immediate as $n^2 p(1-p) \to \infty$ implies that $s_n \to \infty$. We can rewrite  the left hand side of \eqref{eq:prop1:clt} as
 \[ 
n^{\gamma/2} \sqrt{\binom{n}{2}} \frac{\sum_{i<j}  A_{ij}/\binom{n}{2}- p}{\sqrt{n^{\gamma}p(1-p)}} =n^{\gamma/2} \sqrt{\binom{n}{2}} \frac{\hat{p}- p}{\sqrt{n^{\gamma}p(1-p)}}.
 \]
So we conclude that
\[ 
 n^{1+\gamma/2} (\hat{p}-p) \stackrel{d}{\to} 
 \begin{cases}
 N(0,2p^{\dagger}(1-p^{\dagger})) & \text{if} \ \gamma = 0 \\
 N\left(0, 2p^{\dagger}\right) & \text{if} \ \gamma \in (0,2)
 \end{cases}
.
 \]
 This completes the proof. 
\end{proof}

\begin{proof}[Proof Corollary \ref{cor:erm}]
The result is a simple application of the delta method when $\gamma = 0$, and so we focus on the case where $\gamma \in (0,2)$. 
Observe that $\hat{\mu} = \log [n^{-\gamma} n^{\gamma} \hat{p}/(1-\hat{p})] = -\gamma \log n + \log n^{\gamma} \hat{p} - \log (1-\hat{p})$. We have 
\[
\hat{\mu}-\mu = (\log n^{\gamma} \hat{p} - \log n^{\gamma} p )  -( \log(1-\hat{p}) - \log (1-p) ).
\]
By Proposition \ref{prop:erm}, we have $\sqrt{n^{2-\gamma}}(n^{\gamma} \hat{p} -n^{\gamma} p) \stackrel{d}{\to} N(0,2p^{\dagger})$, so that by the delta method (or the Taylor expansion) we have 
\[
\sqrt{n^{2-\gamma}} (\log n^{\gamma} \hat{p} - \log n^{\gamma} p ) = \frac{1}{p^{\dagger}} \sqrt{n^{2-\gamma}} (n^{\gamma} \hat{p} - n^{\gamma}p) + o(1) \stackrel{d}{\to} N(0,2/p^{\dagger}) = N(0,2e^{-\mu^{\dagger}}).
\]
Likewise, we have 
\[
\log (1-\hat{p}) - \log (1-p) = (-1+o_{P}(1)) (\hat{p}-p) = o(n^{-1+\gamma/2}).
\]
Conclude that $\sqrt{n^{2-\gamma}} (\hat{\mu} - \mu) \stackrel{d}{\to} N(0,2e^{-\mu^{\dagger}})$. 
\end{proof}

\subsection{Proof of Theorem \ref{thm:asymnorm}}
\label{sec:proof thm1}

The proof of Theorem \ref{thm:asymnorm} uses Bernstein's inequality. We state Bernstein's inequality for the reader's convenience. See \cite{Boucheron:etal:2013} Theorem 2.10.

\begin{lemma}[Bernsten's inequality]
\label{lem:Bernstein}
Let $X_{1},\dots,X_{n}$ be independent random variables with mean zero such that $|X_{i}| \le b$ a.s. for all $i=1,\dots,n$. 
Then 
\[
P\left ( \left | \sum_{i=1}^{n}X_{i} \right | \ge \sqrt{2t\sum_{i=1}^{n}E[X_{i}^{2}]} +bt/3 \right ) \le 2 e^{-t}
\]
for every $t > 0$. 
\end{lemma}

\begin{proof}[Proof of Theorem \ref{thm:asymnorm}]
Recall the definitions of $\overline{o}(\cdot)$ and $\overline{o}_{P}(\cdot)$. 
In this proof, we focus on the case where $\alpha < \gamma$. The proofs for the other cases are analogous. 
In addition, to simplify the notation, below we use $s$ in place of $s_0$ as the cardinality of $S=S(\vbeta_{0})$. 
We may assume without loss of generality that $S=\{ 1,\dots,s \}$. 
Then the likelihood function for $(\mu,\vbeta_{S})$ is 
\[
\begin{split}
&\prod_{1 \le i < j \le s} \left ( \frac{e^{\mu + \beta_{i}+\beta_{j}}}{1+e^{\mu + \beta_{i}+\beta_{j}}} \right)^{A_{ij}} \left (  \frac{1}{1+e^{\mu + \beta_{i}+\beta_{j}}} \right)^{1-A_{ij}} \\
&\quad \times \prod_{\substack{1 \le i \le s \\ s+1 \le j \le n}} \left ( \frac{e^{\mu + \beta_{i}}}{1+e^{\mu + \beta_{i}}} \right)^{A_{ij}} \left (  \frac{1}{1+e^{\mu + \beta_{i}}} \right)^{1-A_{ij}} 
\times \prod_{s+1 \le i < j \le n} \left ( \frac{e^{\mu}}{1+e^{\mu}} \right )^{A_{ij}} \left ( \frac{1}{1+e^{\mu}} \right)^{1-A_{ij}}. 
\end{split}
\]
The negative log-likelihood for $(\mu,\vbeta_{S})$  is 
\[
\begin{split}
\ell_{n}(\mu,\vbeta_{S}) &= -\mu \underbrace{\sum_{1 \le i<j \le n}A_{ij}}_{=d_{+}} - \underbrace{\sum_{1 \le i<j \le s} (\beta_{i} + \beta_{j} )A_{ij}}_{=\sum_{i=1}^{s} \beta_{i} \sum_{\substack{1 \le j \le s \\ j \ne i}} A_{ij}} -\sum_{i=1}^{s} \beta_{i} \sum_{j=s+1}^{n}A_{ij} + \binom{n-s}{2} \log (1+e^{\mu}) \\
&\quad  + (n-s) \sum_{i=1}^{s} \log (1+e^{\mu + \beta_{i}}) + \sum_{1 \le i < j \le s} \log (1+e^{\mu+\beta_{i}+\beta_{j}}) \\
&= -\mu d_{+} - \sum_{i=1}^{s} \beta_{i} d_{i} + \binom{n-s}{2} \log (1+e^{\mu})  + (n-s) \sum_{i=1}^{s} \log (1+e^{\mu + \beta_{i}}) \\
&\quad + \sum_{1 \le i < j \le s} \log (1+e^{\mu+\beta_{i}+\beta_{j}}).
\end{split}
\]
Recall the reparameterization $\mu =  - \gamma \log n + \mu^{\dagger}$ and $\beta_{i} =\alpha \log n + \beta_{i}^{\dagger}$. 

\textbf{Part (i)}. 
We first prove the uniform consistency of the MLE $(\hat{\mu}^{\dagger},\hat{\vbeta}_{S}^{\dagger})$ in the sense that $\hat{\mu}^{\dagger} = \mu^{\dagger}_{0} + \overline{o}_{P}(1)$ and $\max_{1 \le i \le s} | \hat{\beta}^{\dagger}_{i} - \beta^{\dagger}_{0i}| = \overline{o}_{P}(1)$. 
Consider the concentrated negative log-likelihood for $\mu^{\dagger}$:
\[
\begin{split}
\ell_{n}^{c} (\mu^{\dagger}) &= -\mu^{\dagger} d_{+} + \binom{n-s}{2} \log (1+n^{-\gamma}e^{\mu^{\dagger}})  + (n-s) \sum_{i=1}^{s} \log (1+n^{-(\gamma - \alpha)}e^{\mu^{\dagger} +\hat{\beta}_{i}^{\dagger}}) \\
&\quad + \sum_{1 \le i < j \le s} \log (1+n^{-(\gamma-2\alpha)}e^{\mu^{\dagger}+\hat{\beta}_{i}^{\dagger}+\hat{\beta}_{j}^{\dagger}}),
\end{split}
\]
which is minimized at $\mu^{\dagger} = \hat{\mu}^{\dagger}$ on $[-M_{1}^{\dagger},M_{1}^{\dagger}]$.
Since $\hat{\beta}_{i} \in [0,M_{2}^{\dagger}]$ for $i \in S$ and $M_{1}^{\dagger} \vee M_{2}^{\dagger} = o(\log n)$, we see that 
\[
\begin{split}
&\sup_{|\mu^{\dagger}| \le M_{1}^{\dagger}} \left | \ell_{n}^{c} (\mu^{\dagger})  - \left (- \mu^{\dagger} d_{+} + \binom{n-s}{2} \log (1+n^{-\gamma} e^{\mu^{\dagger}}) \right ) \right | \\
&\quad \le O(e^{o(\log n)} (sn^{1-(\gamma-\alpha)} + s^{2}n^{-(\gamma-2\alpha)})) = \overline{o}(n^{2-\gamma}).
\end{split}
\]
In addition, we have 
\[
\begin{split}
E[d_{+}] &= \sum_{1 \le i < j \le n} p_{ij} = \binom{n-s}{2} \frac{n^{-\gamma}e^{\mu^{\dagger}_{0}}}{1+n^{-\gamma} e^{\mu^{\dagger}_{0}}} + O(e^{o(\log n)} (sn^{1-(\gamma-\alpha)} + s^{2}n^{-(\gamma-2\alpha)})) \\
& =(n^{2-\gamma}/2) e^{\mu_{0}^{\dagger}} + \overline{o}(n^{2-\gamma}) \quad \text{and} \\
\Var (d_{+}) &= \sum_{i < j} p_{ij}(1-p_{ij}) = (n^{2-\gamma}/2) e^{\mu_{0}^{\dagger}} + \overline{o}(n^{2-\gamma}),
\end{split} 
\]
so that $d_{+} = (n^{2-\gamma}/2) e^{\mu_{0}^{\dagger}} +  \overline{o}_{P}(n^{2-\gamma})$. 
Conclude that 
\[
2n^{-2+\gamma} \ell_{n}^{c} (\mu^{\dagger}) =  -\mu^{\dagger} e^{\mu_{0}^{\dagger}} + e^{\mu^{\dagger}} + \overline{o}_{P}(1)
\]
uniformly in $| \mu^{\dagger} | \le M_{1}^{\dagger}$. We will show that $\hat{\mu}^{\dagger} = \mu^{\dagger}_{0} + \overline{o}_{P}(1)$ by mimicking the proof of Theorem 5.7 in \cite{vdV1998}. Some extra care is needed since $\mu_{0}^{\dagger}$ and $M_1^{\dagger}$ may grow with $n$. To this end, let $f(x) = -xe^{x_{0}} + e^{x}$ for $x \in \mathbb{R}$ where $x_0$ is given, and observe from the Taylor expansion that $f(x) = f(x_{0}) + (x-x_{0})^{2} e^{\tilde{x}}$ where $\tilde{x}$ is between $x$ and $x_{0}$. 
Thus, if $| \hat{\mu}^{\dagger} - \mu_{0}^{\dagger} | > \delta$ for some $\delta > 0$, then 
\[
2n^{-2+\gamma} \{ \ell_{n}^{c} (\hat{\mu}^{\dagger}) - \ell_{n}^{c}(\mu_{0}^{\dagger}) \} \ge e^{-M_{1}^{\dagger}}  \delta^{2} + \overline{o}_{P}(1),
\]
but the left hand side is nonpositive by the definition of the MLE. Conclude that 
\[
P(| \hat{\mu}^{\dagger} - \mu_{0}^{\dagger} | > \delta) \le P(e^{M_1^{\dagger}} \overline{o}_{P}(1) \ge \delta^{2}).
\]
This implies that $\hat{\mu}^{\dagger} = \mu_{0}^{\dagger} + \overline{o}_{P}(1)$.

Next, consider the concentrated negative log-likelihood for $\beta_{i}^{\dagger}$:
\[
\ell_{n}^{c} (\beta_{i}^{\dagger}) = -\beta_{i}^{\dagger} d_{i}  +  (n-s) \log (1+n^{-(\gamma - \alpha)}e^{\hat{\mu}^{\dagger} +\beta_{i}^{\dagger}})  + \sum_{\substack{1 \le j \le s \\ j \ne i}} \log (1+n^{-(\gamma-2\alpha)}e^{\hat{\mu}^{\dagger}+\beta_{i}^{\dagger}+\hat{\beta}_{j}^{\dagger}}),
\]
which is minimized at $\beta_{i}^{\dagger}  = \hat{\beta}_{i}^{\dagger}$ on $[0,M_{2}^{\dagger}]$. The last term on the right hand side is $O(e^{o(\log n)}sn^{-(\gamma-2\alpha)}) = \overline{o}(n^{1-(\gamma-\alpha)})$ uniformly in $| \beta_{i}^{\dagger} | \le M_{2}^{\dagger}$ and $1 \le i \le s$.
We note that 
\[
\begin{split}
E[d_{i}] &= \sum_{j \ne i} p_{ij} = \sum_{j > s} p_{ij} +  \underbrace{\sum_{\substack{1 \le j \le s \\ j \ne i}} p_{ij}}_{=O(sn^{-(\gamma-2\alpha)})} = n^{1-(\gamma-\alpha)} e^{\mu_{0}^{\dagger} + \beta_{0i}^{\dagger}} + \overline{o}(n^{1-(\gamma - \alpha)}), \\
\Var (d_{i}) &= \sum_{j \ne i} p_{ij} (1-p_{ij}) =n^{1-(\gamma-\alpha)} e^{\mu_{0}^{\dagger} + \beta_{0i}^{\dagger}} + \overline{o}(n^{1-(\gamma - \alpha)}),
\end{split}
\]
where the $\overline{o}$ terms are uniform in $1 \le i \le s$. 
Since $d_{i} = \sum_{j \ne i} A_{ij}$ is the sum independent random variables with $|A_{ij} - E[A_{ij}]| \le 1$, applying Bernstein's inequality (Lemma \ref{lem:Bernstein}) to $d_{i}$, we have 
\[
P\left (|d_{i} - E[d_{i}]| > \sqrt{2t\Var (d_{i})} + t/3 \right) \le 2e^{-t}
\]
for every $t > 0$. Choosing $t=2\log n$ and using the union bound, we have 
 \[
 \max_{1 \le i \le s}|d_{i} - E[d_{i}]| \le 2\sqrt{\max_{1 \le j \le s} \Var(d_{j}) \log n)} + 2(\log n)/3 
 \]
 with probability approaching one. Using the preceding evaluation of $\Var (d_{i})$, we have
 \[
  \max_{1 \le i \le s}|d_{i} - E[d_{i}]|  = O_{P}(e^{o(\log n)}n^{1/2-(\gamma-\alpha)/2}),
  \]
which implies that $d_{i} =  n^{1-(\gamma-\alpha)} e^{\mu_{0}^{\dagger} + \beta_{0i}^{\dagger}} + \overline{o}_{P}(n^{1-(\gamma - \alpha)})$ uniformly in $1 \le i \le s$. 
Together with the  consistency of $\hat{\mu}^{\dagger}$, we have
\[
n^{-1+(\gamma-\alpha)} \ell_{n}^{c}(\beta_{i}^{\dagger}) =\underbrace{ -\beta_{i}^{\dagger} e^{\mu_{0}^{\dagger} + \beta_{0i}^{\dagger}}  + e^{\mu_{0}^{\dagger} + \beta_{i}^{\dagger}} }_{= e^{\mu_{0}^{\dagger}} (-\beta_{i}^{\dagger} e^{\beta_{0i}^{\dagger}} + e^{\beta_{i}^{\dagger}})}+ \overline{o}_{P}(1)
\]
uniformly in $\beta_{i}^{\dagger} \in [0,M_{2}^{\dagger}]$ and $1 \le i \le s$.
Pick any $\delta > 0$. It is not difficult to show that 
\[
\min_{1 \le i \le s} \min_{|\beta_{i}^{\dagger} - \beta_{0i}^{\dagger}| > \delta} e^{\mu_{0}^{\dagger}} \{ -(\beta_{i}^{\dagger}-\beta_{0i}^{\dagger}) e^{\beta_{0i}^{\dagger}} + e^{\beta_{i}^{\dagger}} - e^{\beta_{0i}^{\dagger}} \} \ge e^{\mu_{0}^{\dagger}}\delta^{2} \ge e^{-M_{1}^{\dagger}} \delta^{2}. 
\]
Now, if $|\hat{\beta}_{i}^{\dagger} - \beta_{0i}^{\dagger}| > \delta$ for some $1 \le i \le s$, then 
\[
\begin{split}
&n^{-1+(\gamma-\alpha)} \ell_{n}^{c}(\hat{\beta}_{i}^{\dagger}) - n^{-1+(\gamma-\alpha)} \ell_{n}^{c}(\beta_{0i}^{\dagger}) \\
&\ge e^{-M_1^{\dagger}} \delta^{2}-  \underbrace{2 \max_{1 \le j \le s} \sup_{|\beta_{j}^{\dagger}| \le M_{2}^{\dagger}}\left | n^{-1+(\gamma-\alpha)} \ell_{n}^{c}(\beta_{j}^{\dagger}) -  e^{\mu_{0}^{\dagger}} (-\beta_{j}^{\dagger} e^{\beta_{0j}^{\dagger}} + e^{\beta_{j}^{\dagger}}) \right |}_{=\overline{o}_{P}(1)},
\end{split}
\]
but by the definition of the MLE, the left hand side is nonpositive.
Conclude that 
\[
P\left ( \max_{1 \le i \le s}| \hat{\beta}_{i}^{\dagger} - \beta_{0i}^{\dagger}| > \delta \right ) \le P(e^{M_{1}^{\dagger}}\overline{o}_{P}(1) \ge  \delta^{2}).
\]
This implies that $\max_{1 \le i \le s} |\hat{\beta}_{i}^{\dagger} - \beta_{0i}^{\dagger}| = \overline{o}_{P}(1)$. 

\textbf{Part (ii)}. 
Next, we will derive the limiting distribution of $(\hat{\mu}^{\dagger} - \mu_{0}^{\dagger},\hat{\vbeta}_{F}^{\dagger} - \vbeta_{0F}^{\dagger})$ for any fixed subset $F \subset S$. 
Since the true parameter vector $(\mu_{0}^{\dagger},\vbeta_{0S}^{\dagger})$ is bounded away from the boundary of the parameter space, the MLE satisfies the first order condition with probability approaching one by the uniform consistency. The first order condition is described as follows: 
\begin{equation}
\begin{split}
&-d_{+} + \binom{n-s}{2} \frac{n^{-\gamma}e^{\mu^{\dagger}}}{1+n^{-\gamma} e^{\mu^{\dagger}}} + (n-s) \sum_{i=1}^{s}  \frac{n^{-(\gamma-\alpha)}e^{\mu^{\dagger}+\beta_{i}^{\dagger}}}{1+n^{-(\gamma-\alpha)} e^{\mu^{\dagger}+\beta_{i}^{\dagger}}}  \\
&\qquad \qquad+  \sum_{1 \le i < j \le s}  \frac{n^{-(\gamma-2\alpha)}e^{\mu^{\dagger}+\beta_{i}^{\dagger}+\beta_{j}^{\dagger}}}{1+n^{-(\gamma-2\alpha)} e^{\mu^{\dagger}+\beta_{i}^{\dagger}+\beta_{j}^{\dagger}}} = 0, \\
&-d_{i} +  (n-s) \frac{n^{-(\gamma-\alpha)}e^{\mu^{\dagger}+\beta_{i}^{\dagger}}}{1+n^{-(\gamma-\alpha)} e^{\mu^{\dagger}+\beta_{i}^{\dagger}}} +  \sum_{\substack{1 \le j \le s \\ j \ne i}} \frac{n^{-(\gamma-2\alpha)}e^{\mu^{\dagger}+\beta_{i}^{\dagger}+\beta_{j}^{\dagger}}}{1+n^{-(\gamma-2\alpha)} e^{\mu^{\dagger}+\beta_{i}^{\dagger}+\beta_{j}^{\dagger}}} = 0, \ i \in S.
\end{split}
\label{eq: FOC}
\end{equation}
The left hand sides are $-d_{+} + E[d_{+}]$ and $-\vd_{S} + E[\vd_{S}]$ at $(\mu^{\dagger},\vbeta_{S}^{\dagger}) = (\mu_{0}^{\dagger},\vbeta_{0S}^{\dagger})$, where $\vd_{S} = (d_{1},\dots,d_{s})^{T}$. 
We will derive the joint limiting distribution for $(d_{+} - E[d_{+}],\vd_{F} - E[\vd_{F}])$. 
Decompose $d_{+}$ as 
\[
d_{+} = \sum_{s < i < j \le n} A_{ij} + \sum_{\substack{1 \le i \le s \\ s < j \le n}} A_{ij} + \sum_{1 \le i < j \le s} A_{ij}. 
\]
The variance of the first term on the right hand side is $(n^{2-\gamma}/2) e^{\mu_{0}^{\dagger}} + \overline{o}(n^{2-\gamma})$, while the variances of the last two terms are $\overline{o}(n^{2-\gamma})$. Hence we have 
\[
d_{+} - E[d_{+}] = \sum_{s <  i < j \le n} (A_{ij} -p_{ij}) + \overline{o}_{P}(n^{1-\gamma/2}) \quad \text{and} \quad n^{-1+\gamma/2} \sum_{s < i < j \le n} \frac{A_{ij} -p_{ij}}{ (e^{\mu_{0}^{\dagger}}/2)^{1/2}} \stackrel{d}{\to} N(0,1).
\]
On the other hand, for $i,j \in S$, we have 
\[
\begin{split}
\Var (d_{i}) &= \sum_{k \ne i} p_{ik} (1-p_{ik}) = n^{1-(\gamma-\alpha)} e^{\mu_{0}^{\dagger} + \beta_{0i}^{\dagger}} + \overline{o}(n^{1-(\gamma-\alpha)}), \\
\Cov (d_{i},d_{j}) &= p_{ij} (1-p_{ij}) = \overline{o}(n^{1-(\gamma-\alpha)}),
\end{split}
\]
so that we have
\[
n^{-1/2+(\gamma - \alpha)/2} \Lambda_{F}^{-1/2}(\vd_{F} - E[\vd_{F}]) \stackrel{d}{\to} N(\bm{0},I_{|F|}),
\]
where $\Lambda_{F} = \diag \{ e^{\mu_{0}^{\dagger} + \beta_{0i}^{\dagger}} : i \in F \}$. 
Since $\sum_{s+1 \le i < j \le n}A_{ij}$ and $\vd_{F}$ are independent, we conclude that 
\begin{equation}
\begin{pmatrix}
n^{-1+\gamma/2} (e^{\mu_{0}^{\dagger}}/2)^{-1/2} (d_{+} - E[d_{+}]) \\
n^{-1/2+(\gamma - \alpha)/2}\Lambda_{F}^{-1/2} (\vd_{F} - E[\vd_{F}])
\end{pmatrix}
\stackrel{d}{\to} 
N \left (\bm{0}, 
I_{1+|F|}
 \right ).
 \label{eq: asymptotic normality}
\end{equation}

Let $\hat{\delta} = \max_{1 \le i \le s} |\hat{\beta}_{i}^{\dagger} - \beta_{0i}^{\dagger}|$.  Applying the Taylor expansion
 to the first equation in (\ref{eq: FOC}), we have 
 \begin{equation}
- d_{+} + E[d_{+}] + (n^{2-\gamma}/2)(e^{\mu_{0}^{\dagger}} + \overline{o}_{P}(1))(\hat{\mu}^{\dagger} - \mu_{0}^{\dagger}) + O_{P}(e^{o(\log n)}sn^{1-(\gamma-\alpha)} \hat{\delta}) = 0.
\label{eq: FOC_mu}
\end{equation}
In particular, this implies that 
\[
\hat{\mu}^{\dagger} - \mu_{0}^{\dagger} = O_{P}(e^{o(\log n)} (n^{-1+\gamma/2} +sn^{-1+\alpha} \hat{\delta})).
\]
Likewise, applying the Taylor expansion
 to the second equation in (\ref{eq: FOC}), we have 
 \begin{equation}
 \begin{split}
 -d_{i} + E[d_{i}] + n^{1-(\gamma-\alpha)}( e^{\mu_{0}^{\dagger} + \beta_{0i}^{\dagger}} + &\overline{o}_{P}(1)) \Big \{ \hat{\beta}_{i}^{\dagger} - \beta_{0i}^{\dagger} \\
&\quad  +O_{P}(e^{o(\log n)}n^{-1+\gamma/2})+ \underbrace{O_{P}(e^{o(\log n)}sn^{-1+\alpha} \hat{\delta})}_{=o_{P}(\hat{\delta})}  \Big \} = 0
 \end{split}
 \label{eq: FOC_beta}
 \end{equation}
 uniformly in $1 \le i \le s$. Since $\max_{1 \le i \le s}|d_{i} - E[d_{i}]| = O_{P}(e^{o(\log n)}n^{1/2-(\gamma-\alpha)/2})$, we have 
 \[
 \hat{\delta} =  O_{P}(e^{o(\log n)}n^{-1/2+(\gamma-\alpha)/2}).
 \]
 Plugging this evaluation into (\ref{eq: FOC_beta}), we have
 \begin{equation}
 n^{1/2-(\gamma-\alpha)/2} (\hat{\beta}_{i}^{\dagger} - \beta_{0i}^{\dagger}) = n^{-1/2+(\gamma-\alpha)/2}  e^{-\mu_{0}^{\dagger} -\beta_{0i}^{\dagger}}(d_{i} - E[d_{i}])+\underbrace{O_{P}(e^{o(\log n)}sn^{-1+\alpha})}_{=\overline{o}_{P}(1)}
 \label{eq: Bahadur_beta}
 \end{equation}
 uniformly in $1 \le i \le s$. 
 Likewise, plugging the preceding evaluation of $\hat{\delta}$ into (\ref{eq: FOC_mu}), we have 
 \[
 - d_{+} + E[d_{+}] + (n^{2-\gamma}/2)(e^{\mu_{0}^{\dagger}} + \overline{o}_{P}(1))(\hat{\mu}^{\dagger} - \mu_{0}^{\dagger}) + O_{P}(e^{o(\log n)}sn^{1/2-(\gamma-\alpha)/2}) = 0
 \]
and the last term on the left hand side is $\overline{o}_{P}(n^{1-\gamma/2})$ under our assumption that $s = \overline{o}(n^{(1-\alpha)/2})$.
Hence we have 
\begin{equation}
n^{1-\gamma/2}(\hat{\mu}^{\dagger} - \mu_{0}^{\dagger}) = 2e^{-\mu_{0}^\dagger}n^{-1+\gamma/2}(d_{+} - E[d_{+}]) + \overline{o}_{P}(1). 
\label{eq: Bahadur_mu}
\end{equation}
The desired conclusion follows from combining the expansions (\ref{eq: Bahadur_mu}) and (\ref{eq: Bahadur_beta}) with (\ref{eq: asymptotic normality}). 
\end{proof}

\subsection{Proofs for Section \ref{section:unknown}}
\label{sec:proof unknown}

\begin{proof}[Proof of Lemma \ref{lem: sorting}]
In this proof, we omit the argument $s$ and write $(\hat{\mu}(s),\hat{\vbeta}(s)) = (\hat{\mu},\hat{\vbeta})$. 

\textbf{Part (i)}. 
Suppose on the contrary that there exist $i$ and $j$ such that $d_i< d_j$ but $\hat\beta_i > \hat\beta_j$. Define $\tilde\vbeta$ by 
\[
\tilde{\beta}_{k} = 
\begin{cases}
\hat\beta_k & \text{if} \ k \neq i,j \\
\hat{\beta}_{j} & \text{if} \ k=i \\
\hat{\beta}_{i} & \text{if} \ k=j
\end{cases}
.
\]
Now, since $-(d_i \hat\beta_j + d_j \hat\beta_i)<-(d_i \hat\beta_i + d_j \hat\beta_j)$,
we have $\ell_n(\hat\mu,\tilde\vbeta)< \ell_n(\hat\mu,\hat\vbeta)$, 
which contradicts the fact that $(\hat\mu,\hat\vbeta)$ is an optimal solution to (\ref{eq:ple}).

\textbf{Part (ii)}. Suppose on the contrary  that there exist $i$ and $j$ such that $d_i = d_j$ but $\hat\beta_i \not= \hat\beta_j$. 
Define $\tilde\vbeta$ by 
\[
\tilde\beta_k=
\begin{cases}
\hat\beta_k  & \text{if} \ k \neq i,j \\
(\hat\beta_i+\hat\beta_j)/2 & \text{if} \ k=i,j
\end{cases}
.
\]
It is not difficult to see that $\tilde\vbeta  \in \mathbb{R}^n_+$ and $ \|\tilde\vbeta\|_0 \le s$ if $s=\sum_{k=1}^K s_k$ for some $K\le m-1$. Since for any $k \neq i,j$,
\[ 
2\log\left(1+e^{\mu+(\beta_i+\beta_j)/2+\beta_k}\right)<\log\left(1+e^{\mu+\beta_i+\beta_k}\right)+\log\left(1+e^{\mu+\beta_j+\beta_k}\right),
\]
we have $\ell_n(\hat\mu,\tilde\vbeta)< \ell_n(\hat\mu,\hat\vbeta)$, which contradicts the fact that $(\hat\mu,\hat\vbeta)$ is an optimal solution to (\ref{eq:ple}). 
\end{proof}

The proof of Lemma \ref{lem: beta min} relies on Hoeffding's inequality; for the reader's convenience, we state it as the following lemma. For its proof, see, e.g., Theorem 2.8 in \cite{Boucheron:etal:2013}. 

\begin{lemma}[Hoeffding's inequality]
\label{lem: Hoeffding}
Let  $X_{1},\dots,X_{n}$ be independent random variables such that each $X_{i}$ takes values in $[a_i,b_i]$ for some $-\infty  < a_i < b_i < \infty$. Then 
\[
P \left ( \sum_{i=1}^{n} (X_{i} - E[X_{i}]) > t \right) \le \exp \left \{ -\frac{2t^{2}}{\sum_{i=1}^{n} (b_i - a_i)^{2}} \right \}
\]
for every $t > 0$.
\end{lemma}

\begin{proof}[Proof of Lemma \ref{lem: beta min}]
For the sake of notational convenience, we use  $(\mu,\vbeta)$ for $(\mu_{0},\vbeta_0)$. 
Recall that $i \in S$ and $j \in S^{c}$, i.e., $\beta_{i} \ne 0$ and $\beta_{j} = 0$. 
Observe that 
\[
d_{i}  =\sum_{k \neq i} A_{ik}=\sum_{k \neq i,j} A_{ik}+A_{ij}, \quad d_{j}  =\sum_{k \neq j} A_{jk}=\sum_{k \neq i,j} A_{jk}+A_{ij},
\]
where 
\[
A_{ik} \sim Ber \left ( \frac{e^{\mu+\beta_{i}+\beta_{k}}}{1+e^{\mu+\beta_{i}+\beta_{k}}} \right ), \quad A_{jk}  \sim Ber \left(  \frac{e^{\mu+\beta_{k}}}{1+e^{\mu+\beta_{k}}} \right).
\]
Then, 
\begin{align*}
d_{i}-d_{j} & =\sum_{k \neq i,j} (A_{ik}-A_{jk})  =\sum_{k \neq i,j} (A_{ik}-E[A_{ik}])-\sum_{k \neq i,j} (A_{jk}-E[A_{jk}])+\sum_{k \neq i,j} (E[A_{ik}]-E[A_{jk}])\\
 & =\sum_{k \neq i,j} (A_{ik}-E[A_{ik}])-\sum_{k \neq i,j} (A_{jk}-E[A_{jk}])+\sum_{k \neq i,j} \left ( \frac{e^{\mu+\beta_{i}+\beta_{k}}}{1+e^{\mu+\beta_{i}+\beta_{k}}} - \frac{e^{\mu+\beta_{k}}}{1+e^{\mu+\beta_{k}}}  \right ).
\end{align*}
Define 
\[
\varepsilon_{ij} = \min_{k \ne i,j} \left ( \frac{e^{\mu+\beta_{i}+\beta_{k}}}{1+e^{\mu+\beta_{i}+\beta_{k}}} - \frac{e^{\mu+\beta_{k}}}{1+e^{\mu+\beta_{k}}} \right ) > 0,
\]
and observe that 
\[
d_{i} - d_{j} \ge \sum_{k \neq i,j} (A_{ik}-E[A_{ik}])-\sum_{k \neq i,j} (A_{jk}-E[A_{jk}]) + (n-2) \varepsilon_{ij}.
\]
Now, by Hoeffding's inequality (Lemma \ref{lem: Hoeffding}), for every $t > 0$,
\[
P \left (\sum_{k \neq i,j} (A_{ik}-E[A_{ik}]) < -t \right ) \le e^{-2t^{2}/(n-2)},
\]
and so with probability at least $1-\tau/2$, 
\[
\sum_{k \neq i,j} (A_{ik}-E[A_{ik}]) \ge -\sqrt{\frac{(n-2)}{2} \log (2/\tau)}.
\]
Likewise, with probability at least $1-\tau/2$, 
\[
\sum_{k \neq i,j} (A_{jk}-E[A_{jk}])  \le \sqrt{\frac{(n-2)}{2} \log (2/\tau)}.
\]
Hence, with probability at least $1-\tau$, 
\[
d_{i} - d_{j} \ge (n-2) \Big \{ -\underbrace{\sqrt{\frac{2}{n-2}\log(2/\tau)}}_{= c_{n,\tau}} + \varepsilon_{ij} \Big \}.
\]
Next, we establish a lower bound on $\varepsilon_{ij}$. Observe that 
\[
\begin{split}
\frac{e^{\mu+\beta_{i}+\beta_{k}}}{1+e^{\mu+\beta_{i}+\beta_{k}}} - \frac{e^{\mu+\beta_{k}}}{1+e^{\mu+\beta_{k}}}  = \frac{e^{\mu+\beta_{k}}}{1+e^{\mu+\beta_{k}}} \cdot \frac{e^{\beta_{i}}-1}{1+e^{\mu+\beta_{i}+\beta_{k}}}  
 \ge \frac{e^{-\mu^{-}}}{1+e^{-\mu^{-}}} \cdot \frac{e^{\beta_{i}}-1}{1+e^{2\overline{\beta}+\mu^{+}}}, 
\end{split}
\]
so that 
\[
\varepsilon_{ij} \ge \frac{1}{1+e^{\mu^{-}}} \cdot \frac{e^{\beta_{i}}-1}{1+e^{2\overline{\beta}+\mu^{+}}}.
\]
The right hand side is larger than $c_{n,\tau}$ under Condition (\ref{eq: betamin}). This completes the proof. 
\end{proof}

\begin{proof}[Proof of Theorem \ref{thm: risk bound}]
For the sake of notational simplicity, we will write $(\hat{\mu}(s),\hat{\vbeta}(s)) = (\hat\mu,\hat\vbeta)$. 
We begin with noting that 
\begin{equation}
\label{eq: step1}
\begin{split}
\mathcal{E}_{s} &\le \mathcal{R} (\hat\mu,\hat\vbeta) - \inf_{(\mu,\vbeta) \in \Theta_{s}} D_{+}^{-1}\ell_{n}(\mu,\vbeta) + \sup_{(\mu,\vbeta) \in \Theta_{s}} | D_{+}^{-1}\ell_{n}(\mu,\vbeta) - \mathcal{R}(\mu,\vbeta)| \\
&=\mathcal{R} (\hat\mu,\hat\vbeta) - D_{+}^{-1}\ell_{n}(\hat\mu,\hat\vbeta) + \sup_{(\mu,\vbeta) \in \Theta_{s}} | D_{+}^{-1}\ell_{n}(\mu,\vbeta) - \mathcal{R}(\mu,\vbeta)| \\
&\le 2 \sup_{(\mu,\vbeta) \in \Theta_s} | D_{+}^{-1}\ell_{n}(\mu,\vbeta) - \mathcal{R}(\mu,\vbeta)|.
\end{split}
\end{equation}
Next, observe that 
\begin{equation}
\label{eq: step2}
\begin{split}
|D_{+}^{-1}\ell_{n}(\mu,\vbeta) - \mathcal{R}(\mu,\vbeta) | &\le  D_{+}^{-1} |\mu| \left |  d_{+} - E[d_{+}] \right | + D_{+}^{-1}\left | \sum_{i=1}^{n} \beta_i (d_i - E[d_i]) \right |  \\
&\le D_{+}^{-1} \left ( M_2 \left |    d_{+} - E[d_{+}] \right |  + M_1 s \max_{1 \le i \le n} | d_{i} - E[d_{i}] |   \right )
\end{split}
\end{equation}
for $(\mu,\vbeta) \in \Theta_{s}$ where we have used the fact that $\sum_{i=1}^{n} \beta_i \le M_1 s$. 
Now, using Bernstein's inequality (Lemma \ref{lem:Bernstein}) and the union bound, we have
\begin{equation}
\label{eq: step3}
 \max_{1 \le i \le n} | d_{i} - E[d_{i}] | \le \sqrt{2\max_{1 \le j \le n} \Var(d_{j}) \log (4n/\tau)} + (\log (4n/\tau))/3
\end{equation}
with probability at least $1-\tau/2$. Likewise, by Bernstein's inequality, we have 
\begin{equation}
\label{eq: step4}
\left |  d_{+} - E[d_{+}] \right | \le \sqrt{2\Var(d_{+}) \log (4/\tau)} + (\log (4/\tau))/3
\end{equation}
with probability at least $1-\tau/2$. Combining (\ref{eq: step1})--(\ref{eq: step4}), we obtain the bound (\ref{eq:generic bound}). 

Finally, if $\mu_{0} = -\gamma \log n+ O(1)$ and $\beta_{0i} = \alpha \log n + O(1)$ for $i \in S(\vbeta_{0})$, then from the proof of Theorem \ref{thm:asymnorm}, we know that $D_{+} \sim n^{2-\gamma}$, $\Var (d_{+}) \sim n^{2-\gamma}$, and $\max_{1 \le i \le n} \Var (d_{i}) \sim n^{1-(\gamma-\alpha)}$. This leads to the second bound (\ref{eq:risk bound}). 
\end{proof}

\section{Existence of $\ell_{0}$-constrained MLE}
\label{sec: MLE existence}

In this appendix, we discuss the existence of the $\ell_{0}$-constrained MLE (\ref{eq:ple}) under the unrestricted parameter space. The aim of this appendix is to derive an analogous result to Theorem 3.1 in \cite{Rinaldo:etal:2013} on the existence of the unconstrained MLE for the $\beta$-model. We first note that the optimization problem (\ref{eq:ple}) can be split into two parts. 

Part 1. For each given $S \subset \{ 1,\dots, n \}$ with $|S| \le s$, find the support-constrained MLE 
\begin{equation}
(\hat{\mu}^{S},\hat{\bm{\beta}}^{S}) = \argmin_{\substack{\mu \in \mR, \bm{\beta} \in \mR_{+}^{n} \\ \supp (\bm{\beta}) = S}} \ell_{n}(\mu,\bm{\beta}).
\label{eq:sple}
\end{equation}

Part 2. Find $S$ that minimizes the negative log-likelhood at $(\hat{\mu}^{S},\hat{\vbeta}^{S})$:
\[
\hat{S}(s) = \argmin_{S \subset \{ 1,\dots,n \} , |S| \le s} \ell_{n}(\hat{\mu}^{S},\hat{\vbeta}^{S}). 
\]
Then, we have $(\hat{\mu}(s),\hat{\vbeta}(s)) = (\hat{\mu}^{\hat{S}(s)},\hat{\vbeta}^{\hat{S}(s)})$. 

Part 2 is comparing the negative log-likelihoods over a finite number of competitors, so that the existence of a solution to Part 2 is always guaranteed. Thus, we focus on finding conditions under which the support-constraint MLE (\ref{eq:sple}) exists with given $S$. Fix $S \subset \{ 1,\dots, n \}$ with $|S| \le s$, and think of the parameter vector as $(\mu,\vbeta_{S}) \in \mR \times \mR_{+}^{|S|}$. The negative log-likelihood when the support of $\vbeta$ is restricted to $S$ is given by 
\[
\ell_{n}^{S} (\mu,\vbeta_{S}) =  -\mu d_{+} - \sum_{i \in S} \beta_{i} d_{i} + \psi_{n}^{S} (\mu,\vbeta_{S}),
\]
where $\psi_{n}^{S} (\mu,\vbeta_{S}) = \binom{n-|S|}{2} \log (1+e^{\mu}) + (n-|S|) \sum_{i \in S} \log (1+e^{\mu + \beta_{i}}) + \sum_{i,j \in S, i < j} \log (1+e^{\mu + \beta_{i} + \beta_{j}})$. 
The corresponding probability mass function belongs to an exponential family, and we can use e.g. Theorem 5.7 in \cite{Brown1986} to derive conditions under which a solution to (\ref{eq:sple}) exists (note that the parameter space for $\vbeta_{S}$ is restricted to the positive orthant). We may write the vector of sufficient statistics $T=\binom{d_{+}}{(d_{i})_{i \in S}} \in \mR^{1+|S|}$ as a function of $(A_{i,j})_{i<j} = (a_{i,j})_{i < j}$;
\[
T(\bm{a}) = B \bm{a}, \ \bm{a} = (a_{i,j})_{i < j}  \in \{ 0,1 \}^{\binom{n}{2}} =: \mathcal{S}_{n}
\]
for some $(1+|S|) \times \binom{n}{2}$-matrix $B$. For example, if $n=3$ and $S = \{ 1,2 \}$, we have 
\[
\begin{pmatrix}
d_{+} \\
d_{1} \\
d_{2}
\end{pmatrix}
=
\underbrace{
\begin{pmatrix}
1 & 1 & 1 \\
1 & 1 & 0 \\
1 & 0 & 1
\end{pmatrix}
}_{=B}
\begin{pmatrix}
a_{12} \\
a_{13} \\
a_{23}
\end{pmatrix}
.
\]
Then, Theorem 5.7 in \cite{Brown1986} leads to the following lemma. 
\begin{lemma}
If $T=\binom{d_{+}}{(d_{i})_{i \in S}}$ lies in the interior of the convex hull of $\{ B \bm{a} : \bm{a} \in \mathcal{S}_{n} \}$, then a solution to (\ref{eq:sple}) exists. 
\end{lemma}

\newpage
\section{Additional Simulations}\label{Appendix-Section-Simulation}

\begin{figure}[hbt!]
\begin{minipage}{\textwidth}
\centering

\begin{subfigure}[b]{0.3\textwidth}
     \caption[ ]
        {{\footnotesize $\beta_{0i} = 1.5, \mu_{0} = - 1.5$}}
 	\centering
	\includegraphics[width=1.1\textwidth, height = \textwidth]{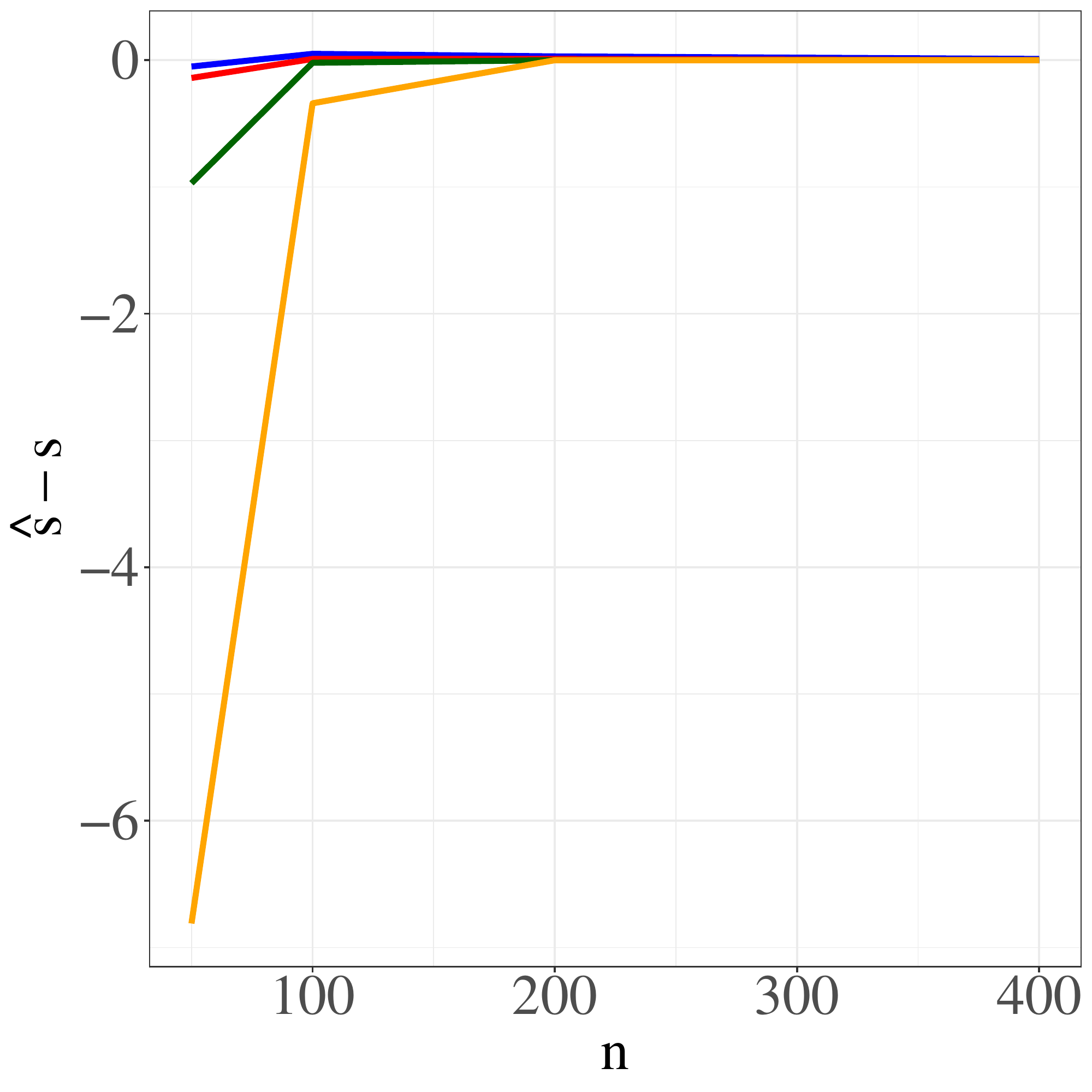}
\end{subfigure}
 \hfill \hfill \hfill \hspace{-5mm} 
\begin{subfigure}[b]{0.3\textwidth}
   \caption[]
        {{\footnotesize $\beta_{0i} = \sqrt{\log n}, \mu_{0} = - 1.5$}}
 	\centering
	\includegraphics[width=\textwidth]{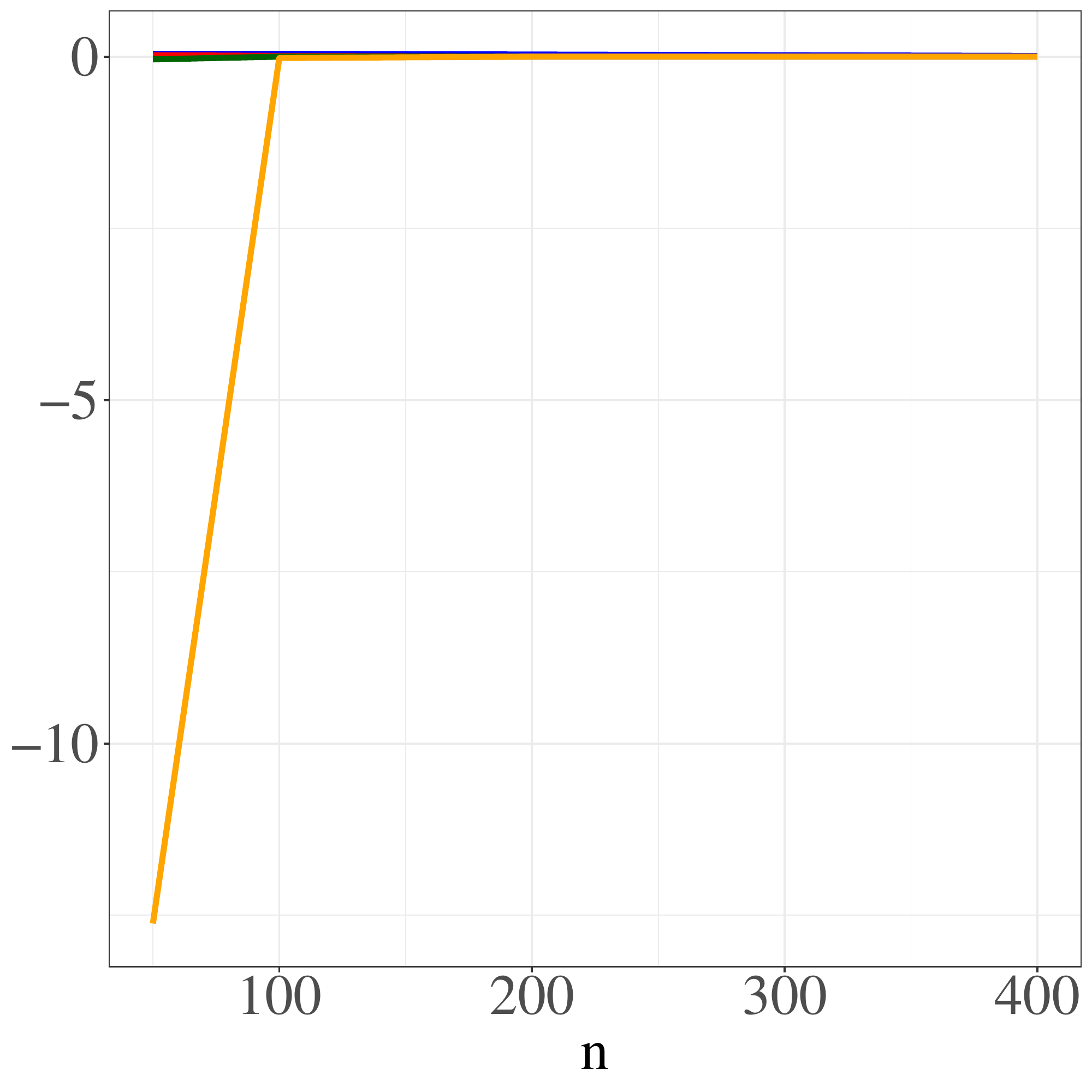}
\end{subfigure}
\hfill  
\begin{subfigure}[b]{0.3\textwidth}
   \caption[]
        {{\footnotesize $\beta_{0i} = \log n, \mu_{0} = - 1.5$}}
 	\centering
	\includegraphics[width=\textwidth]{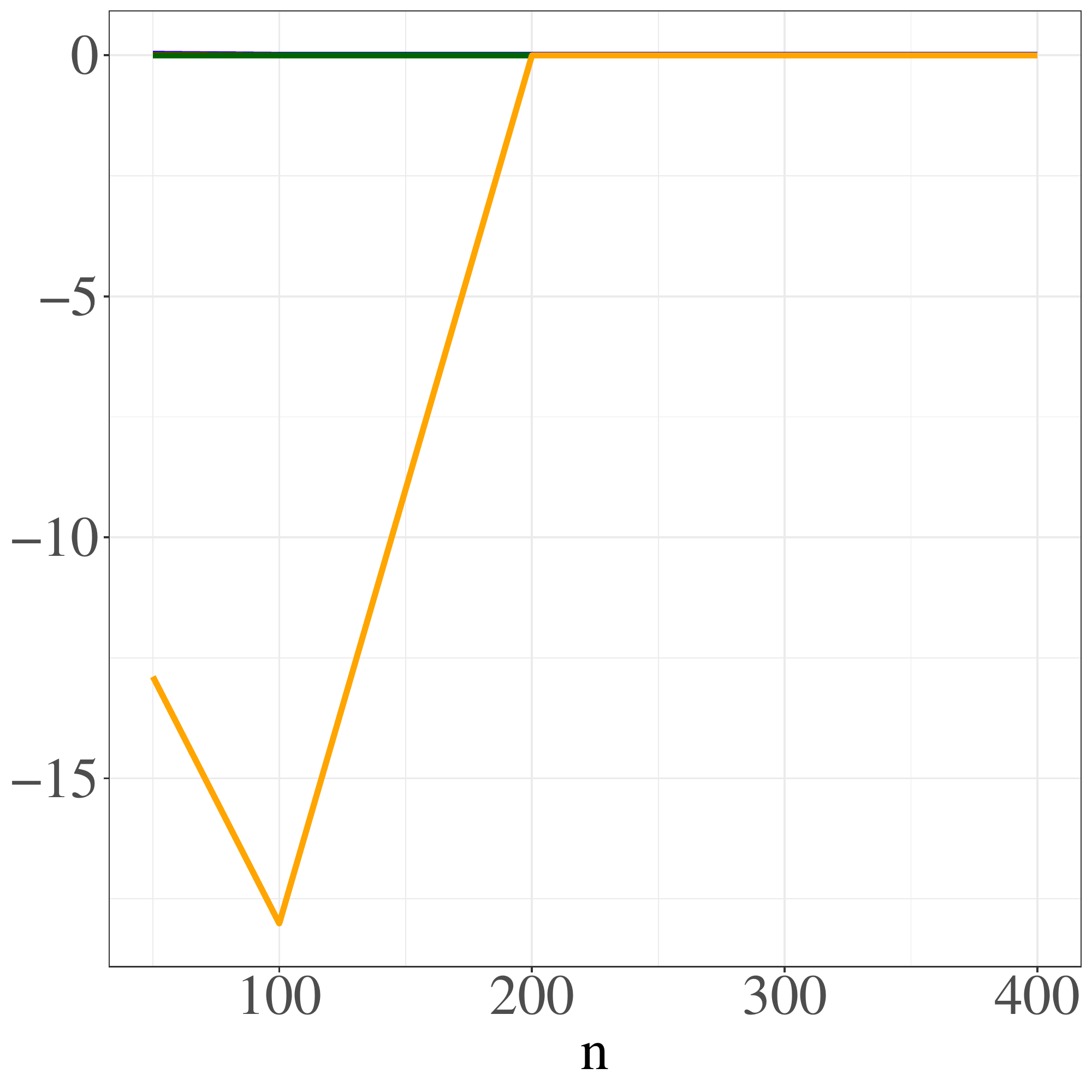}
\end{subfigure}

\vskip\baselineskip
\begin{subfigure}[b]{0.3\textwidth}
     \caption[]
        {{\footnotesize $\beta_{0i} = 1.5, \mu_{0} = - \sqrt{\log n}$}}
 	\centering
	\includegraphics[width=1.1\textwidth, height = \textwidth]{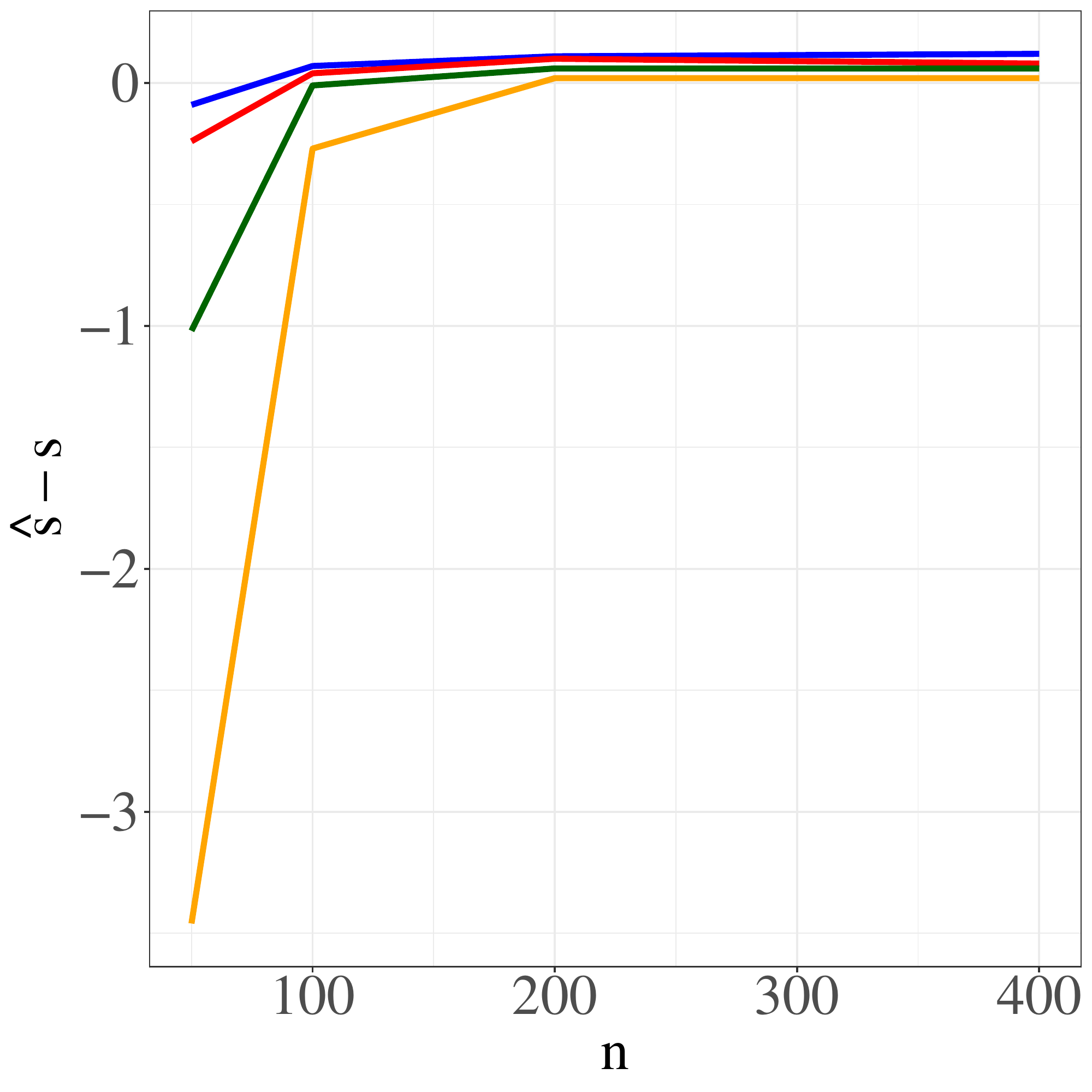}
\end{subfigure}
 \hfill \hfill \hfill \hspace{-5mm}
\begin{subfigure}[b]{0.3\textwidth}
   \caption[]
        {{\footnotesize $\beta_{0i} = \sqrt{\log n}, \mu_{0} = - \sqrt{\log n}$}}
 	\centering
	\includegraphics[width=\textwidth]{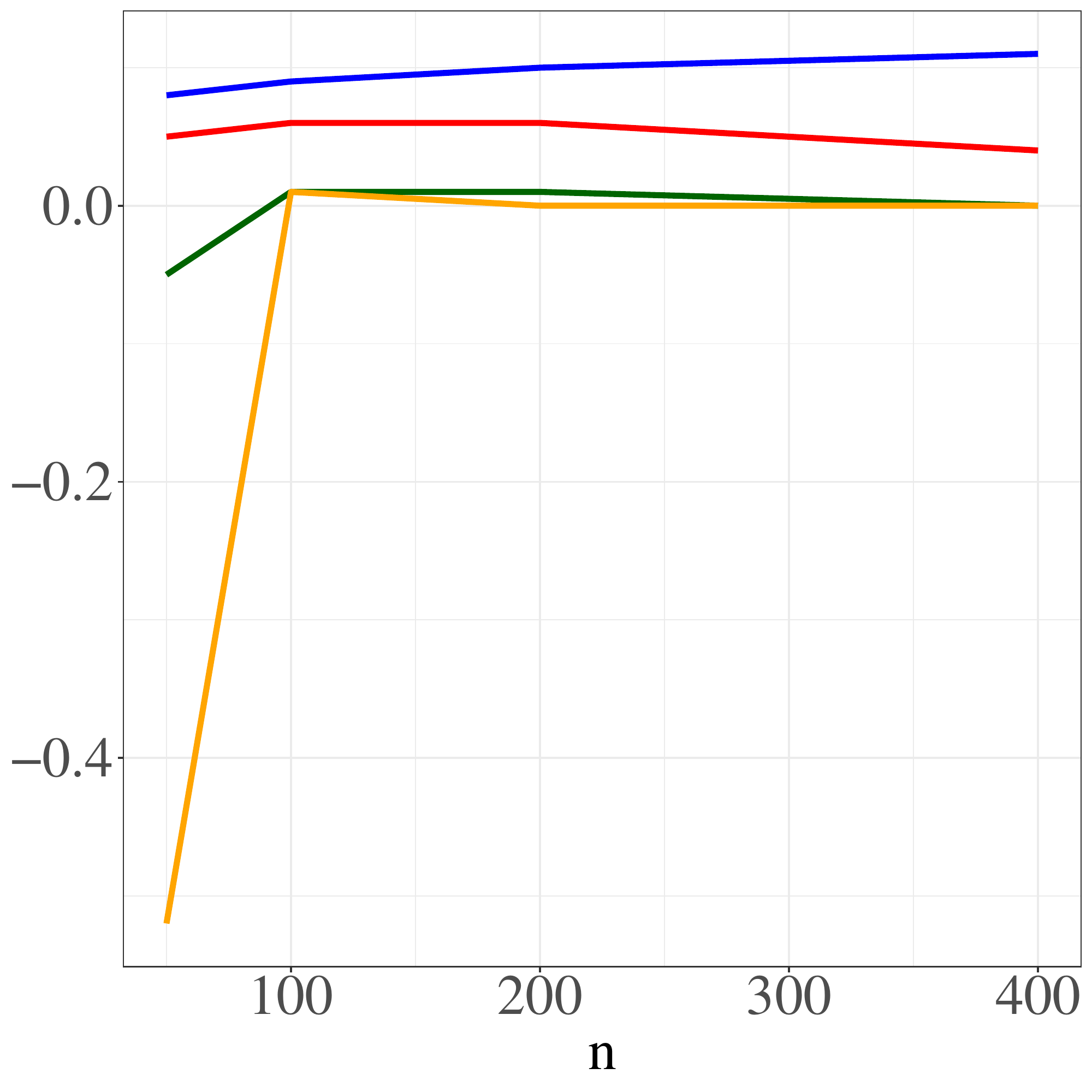}
\end{subfigure}
\hfill 
\begin{subfigure}[b]{0.3\textwidth}
    \caption[]
        {{\footnotesize $\beta_{0i} = \log n, \mu_{0} = - \sqrt{\log n}$}}
 	\centering
	\includegraphics[width=\textwidth]{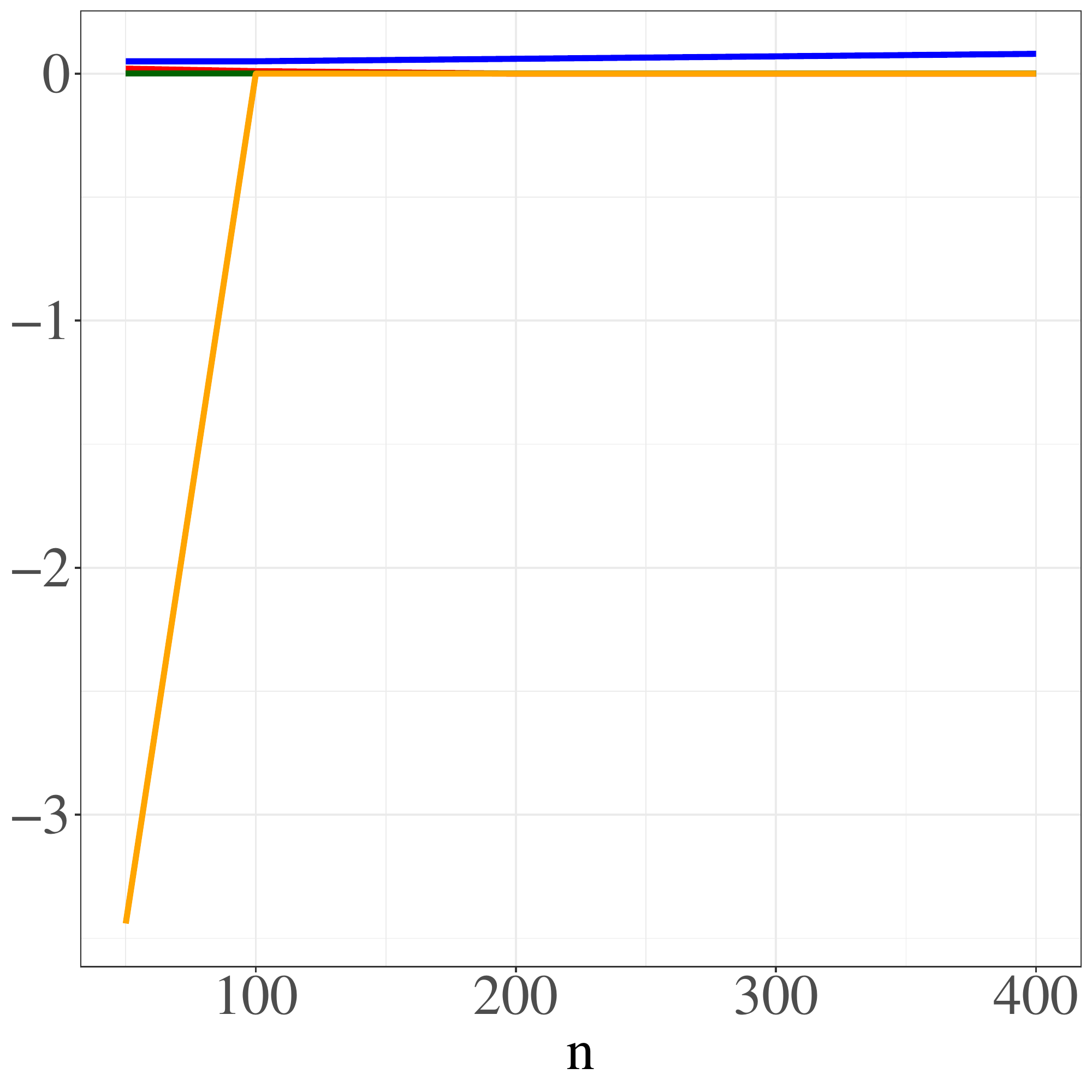}
\end{subfigure}

\vskip\baselineskip
\begin{subfigure}[b]{0.3\textwidth}
     \caption[]
        {{\footnotesize $\beta_{0i} = 1.5, \mu_{0} = - \log n$}}
 	\centering
	\includegraphics[width=1.1\textwidth, height = \textwidth]{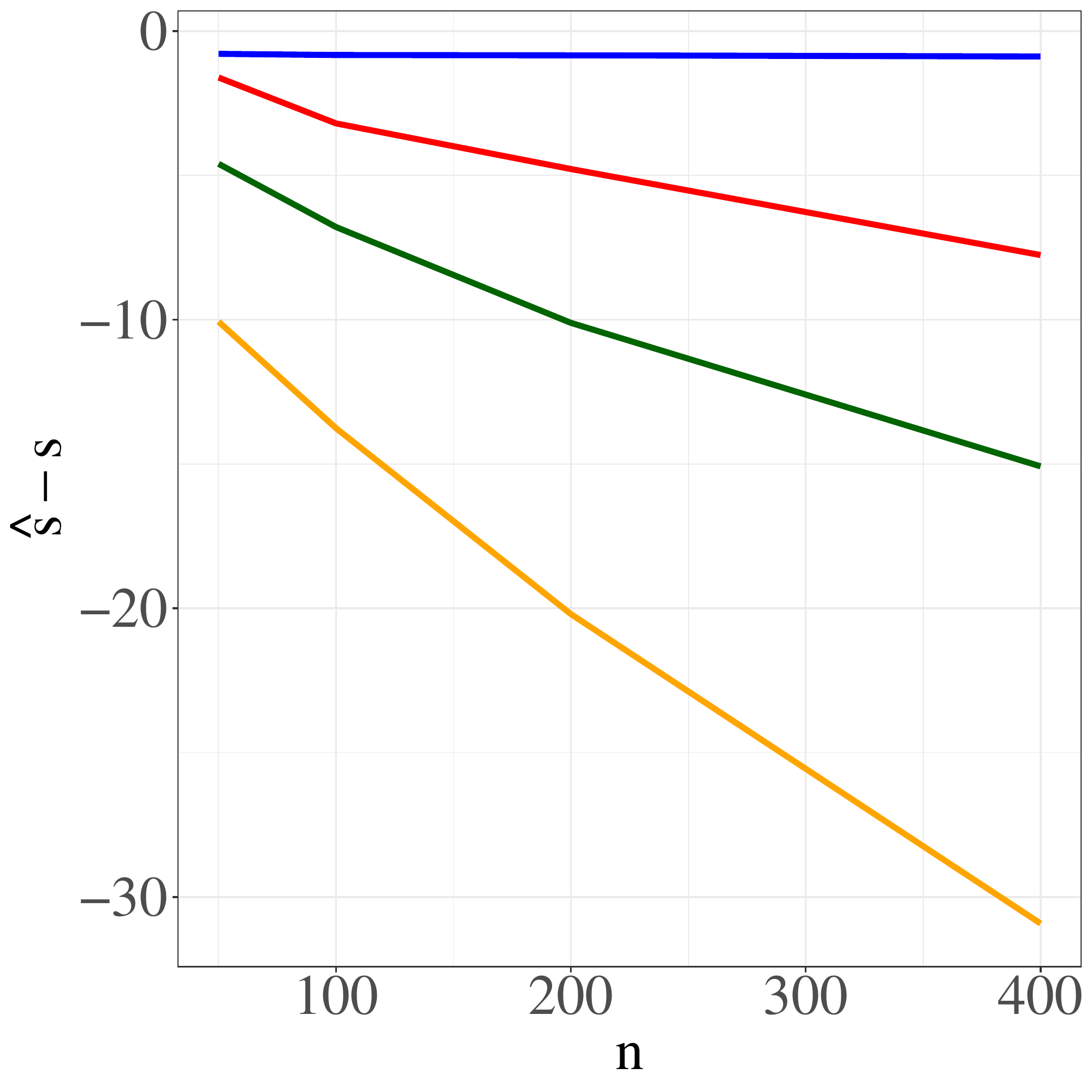}
\end{subfigure}
 \hfill \hfill \hfill \hspace{-5mm}
\begin{subfigure}[b]{0.3\textwidth}
   \caption[]
        {{\footnotesize $\beta_{0i} = \sqrt{\log n}, \mu_{0} = - \log n$}}
 	\centering
	\includegraphics[width=\textwidth]{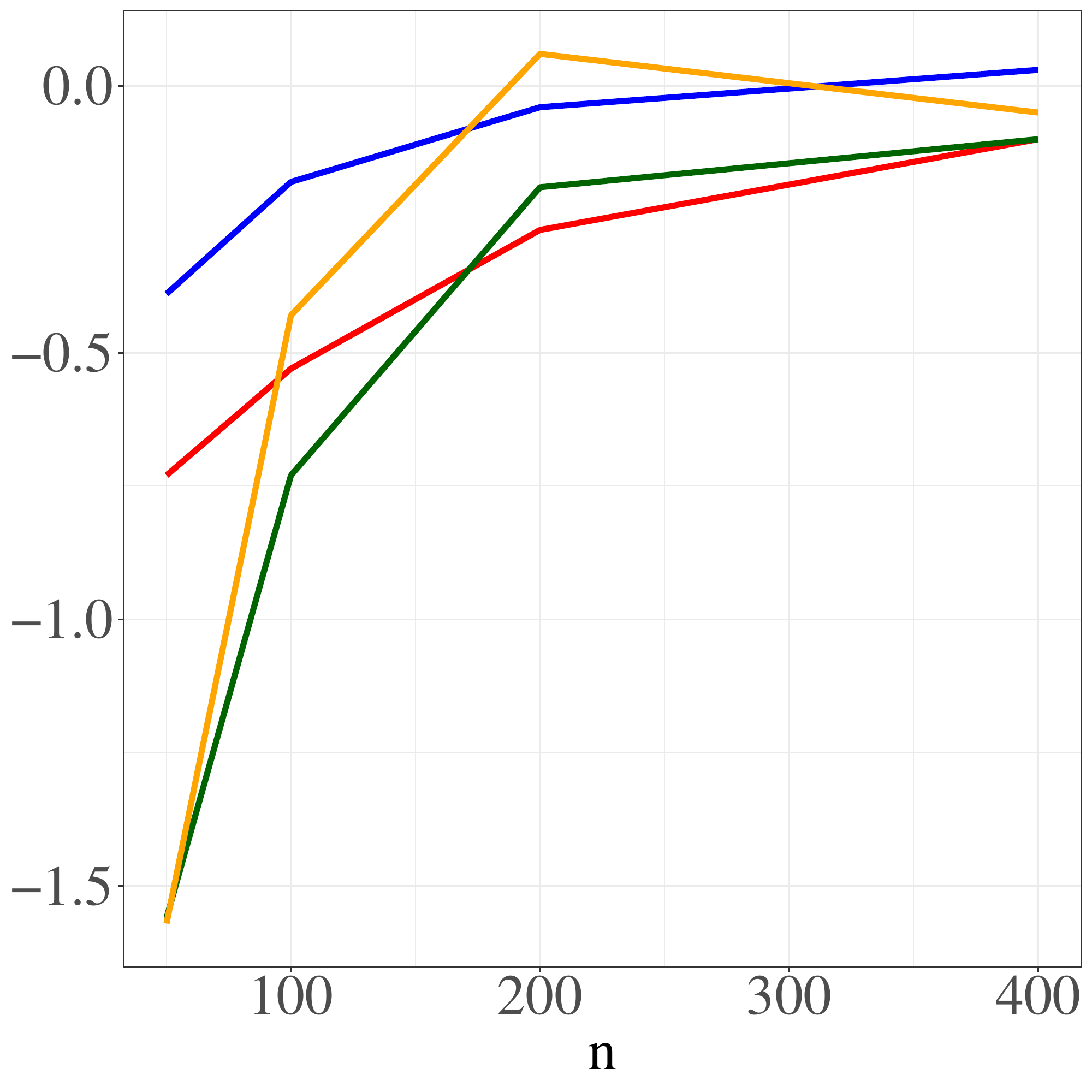}
\end{subfigure}
\hfill 
\begin{subfigure}[b]{0.3\textwidth}
    \caption[]
        {{\footnotesize $\beta_{0i} = \log n, \mu_{0} = - \log n$}}
 	\centering
	\includegraphics[width=\textwidth]{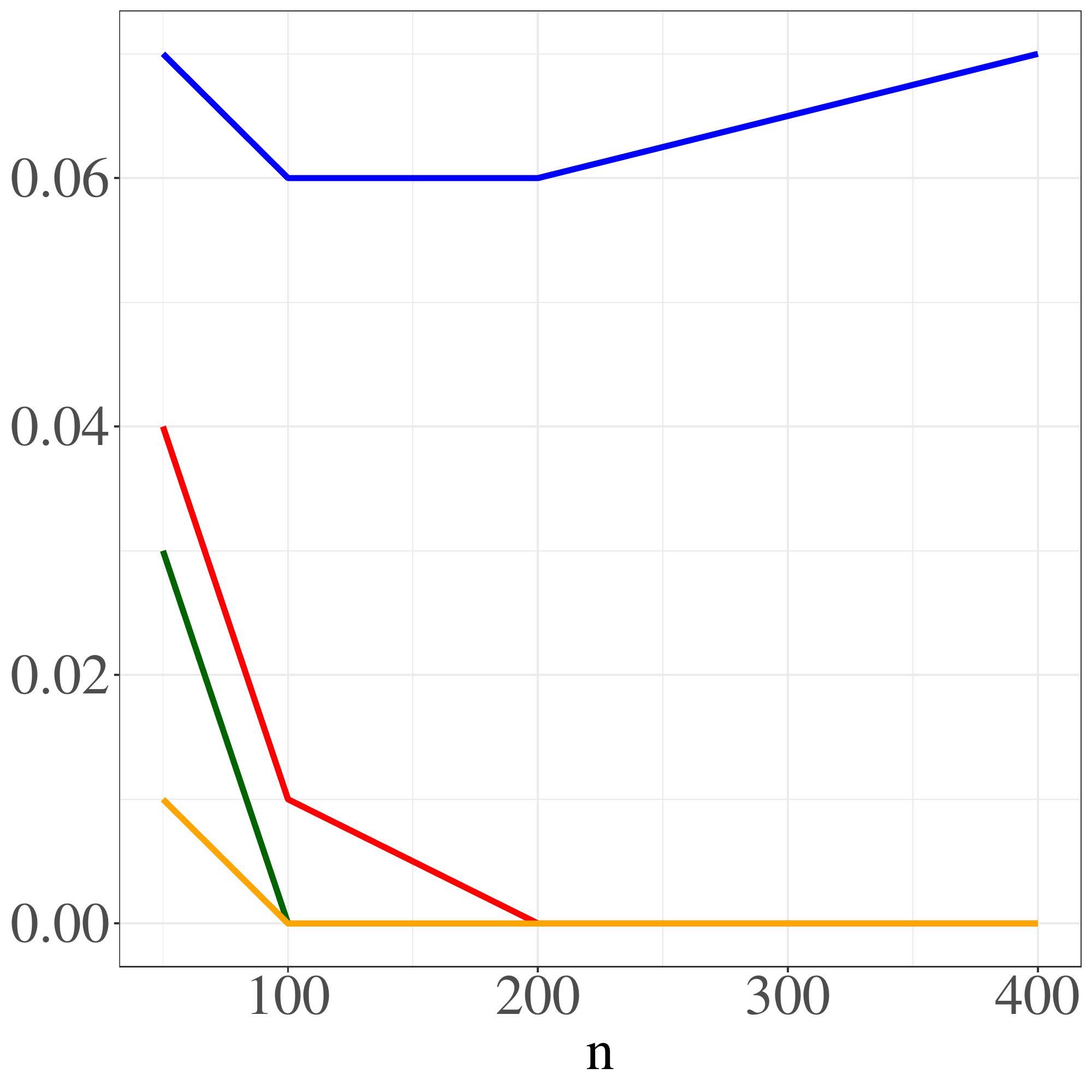}
\end{subfigure}

 \caption[ ]
 {\small Simulation results on $\hat{s} - s_{0}$ with $\hat{s}$ selected by BIC. \  \legendsquare{blue}~$s_0 = 2$,
  \legendsquare{red}~$s_0 = \lfloor \sqrt{n/2} \rfloor$,
  \legendsquare{darkgreen}~$s_0 = \left \lfloor \sqrt{n}\right \rfloor$, \legendsquare{orange}~$s_0 = \left \lfloor 2\sqrt{n}\right \rfloor$.  }
        \label{Figure-Number-of-Regressors}
\end{minipage}
\end{figure}

\begin{figure}[p]
\begin{minipage}{\textwidth}
\centering

\begin{subfigure}[b]{0.3\textwidth}
     \caption[ ]
        {{\footnotesize $\beta_{0i} = 1.5, \mu_{0} = - 1.5$}}
 	\centering
	\includegraphics[width=1.1\textwidth, height = \textwidth]{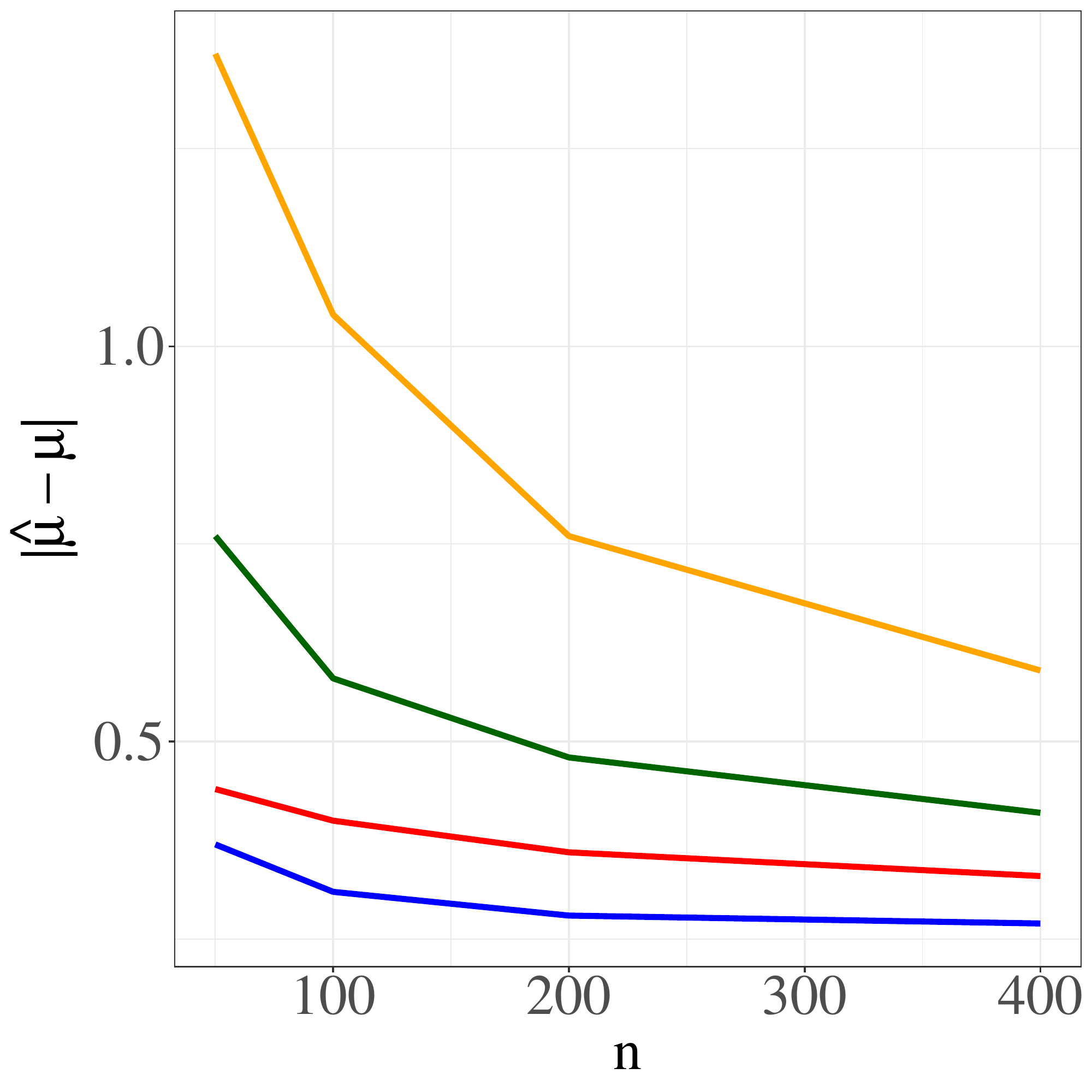}
\end{subfigure}
 \hfill \hfill \hfill \hspace{-5mm}
\begin{subfigure}[b]{0.3\textwidth}
   \caption[]
        {{\footnotesize $\beta_{0i} = \sqrt{\log n}, \mu_{0} = - 1.5$}}
 	\centering
	\includegraphics[width=\textwidth]{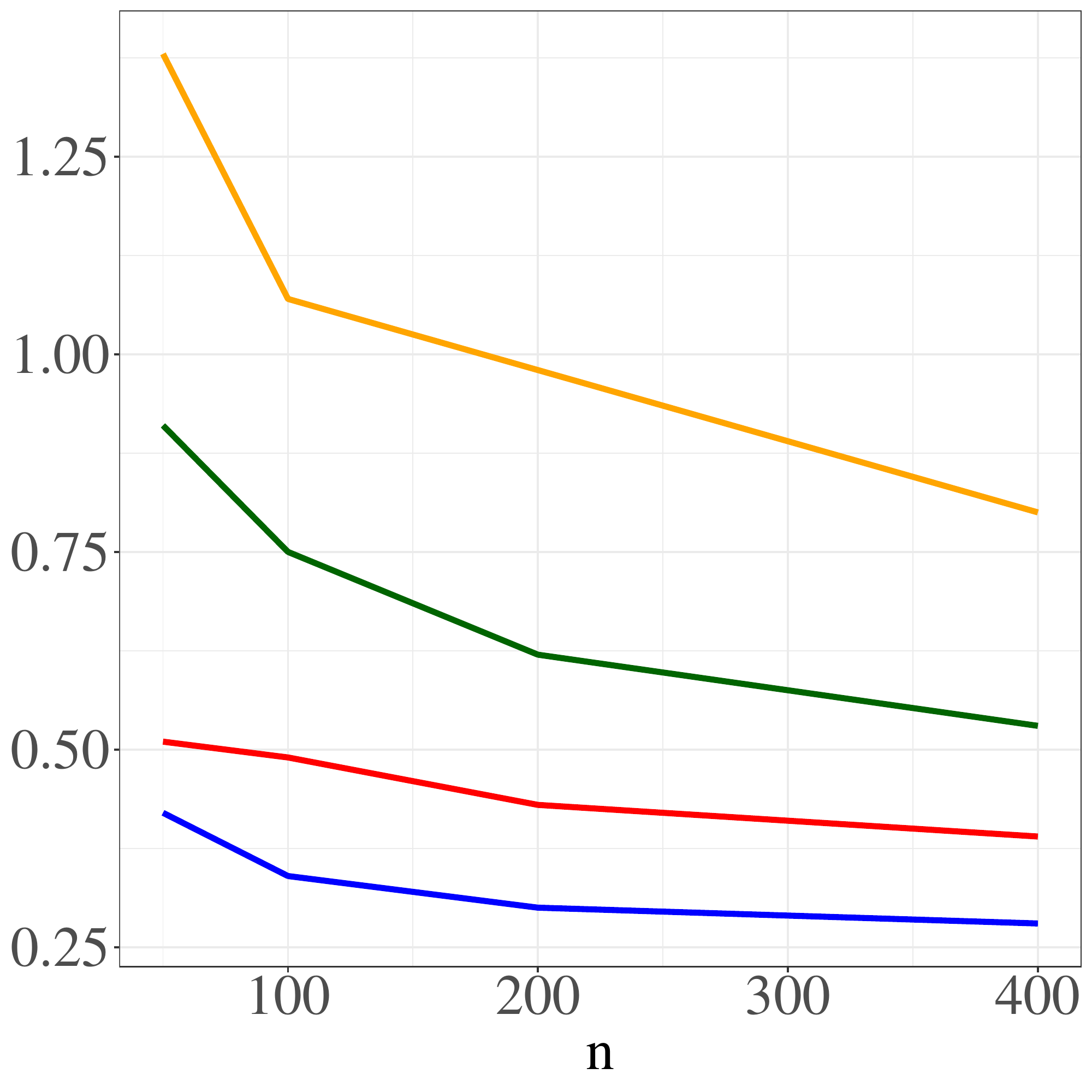}
\end{subfigure}
\hfill   
\begin{subfigure}[b]{0.3\textwidth}
   \caption[]
        {{\footnotesize $\beta_{0i} = \log n, \mu_{0} = - 1.5$}}
 	\centering
	\includegraphics[width=\textwidth]{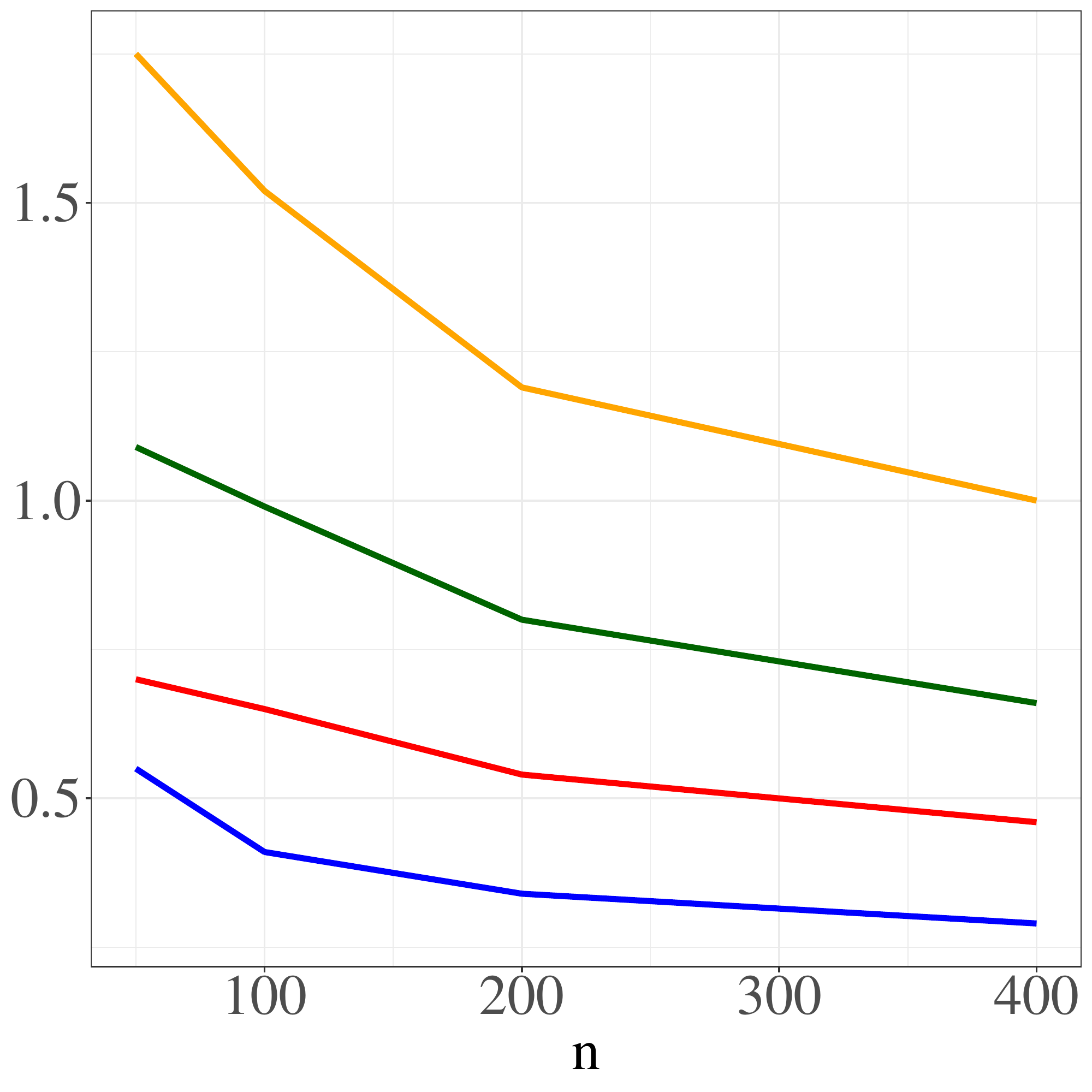}
\end{subfigure}

\vskip\baselineskip
\begin{subfigure}[b]{0.3\textwidth}
     \caption[]
        {{\footnotesize $\beta_{0i} = 1.5, \mu_{0} = - \sqrt{\log n}$}}
 	\centering
	\includegraphics[width=1.1\textwidth, height = \textwidth]{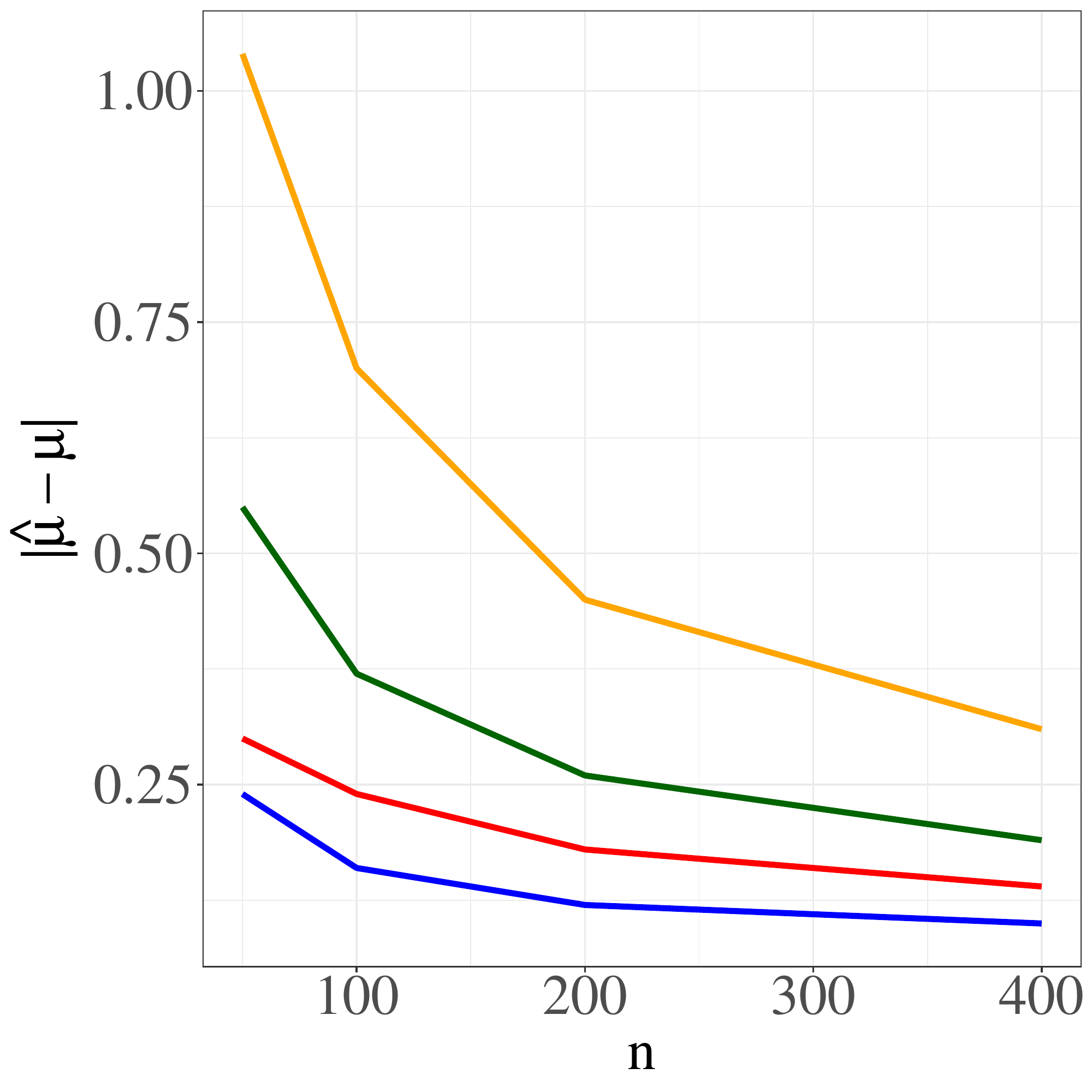}
\end{subfigure}
 \hfill \hfill \hfill  \hspace{-5mm}
\begin{subfigure}[b]{0.3\textwidth}
   \caption[]
        {{\footnotesize $\beta_{0i} = \sqrt{\log n}, \mu_{0} = - \sqrt{\log n}$}}
 	\centering
	\includegraphics[width=\textwidth]{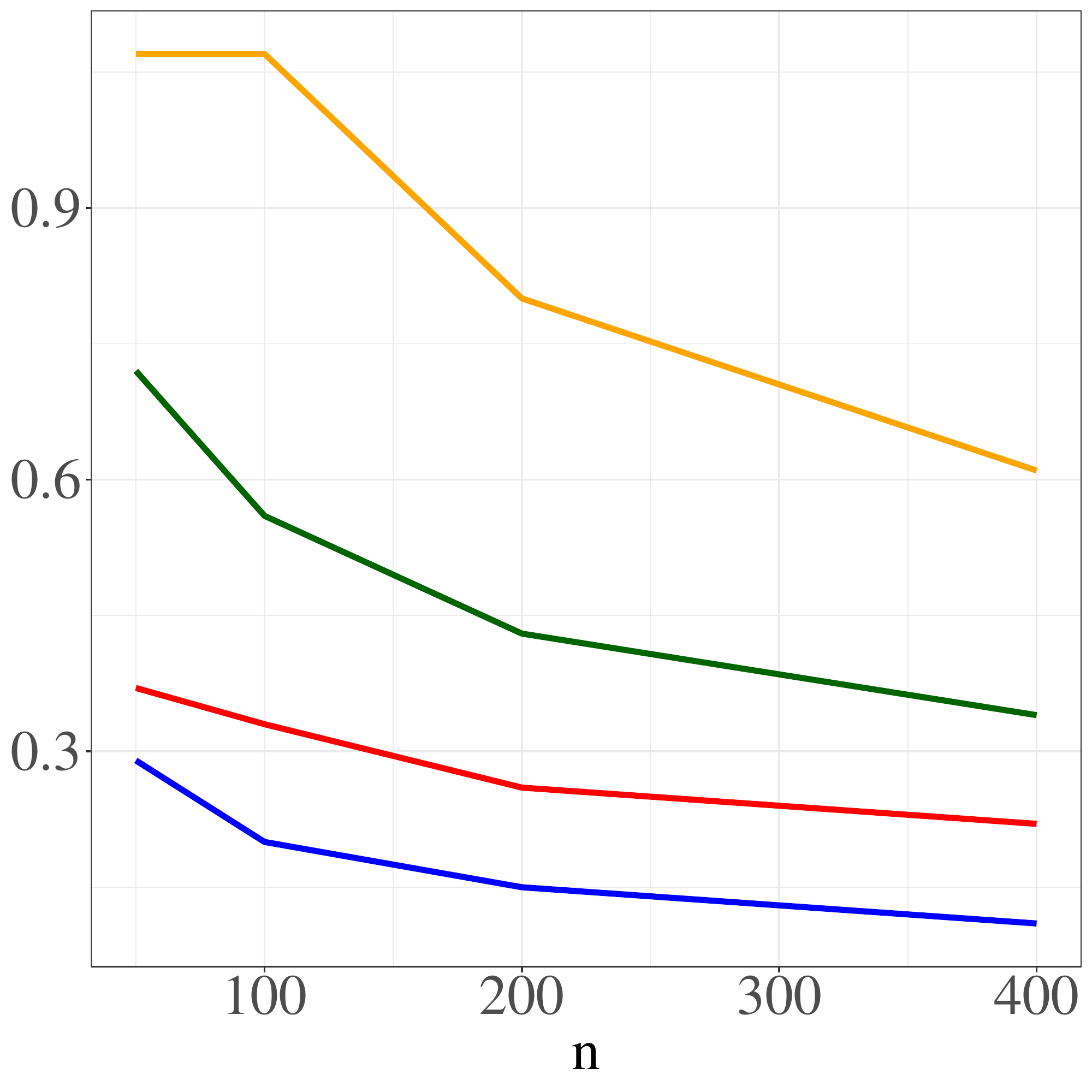}
\end{subfigure}
\hfill  
\begin{subfigure}[b]{0.3\textwidth}
    \caption[]
        {{\footnotesize $\beta_{0i} = \log n, \mu_{0} = - \sqrt{\log n}$}}
 	\centering
	\includegraphics[width=\textwidth]{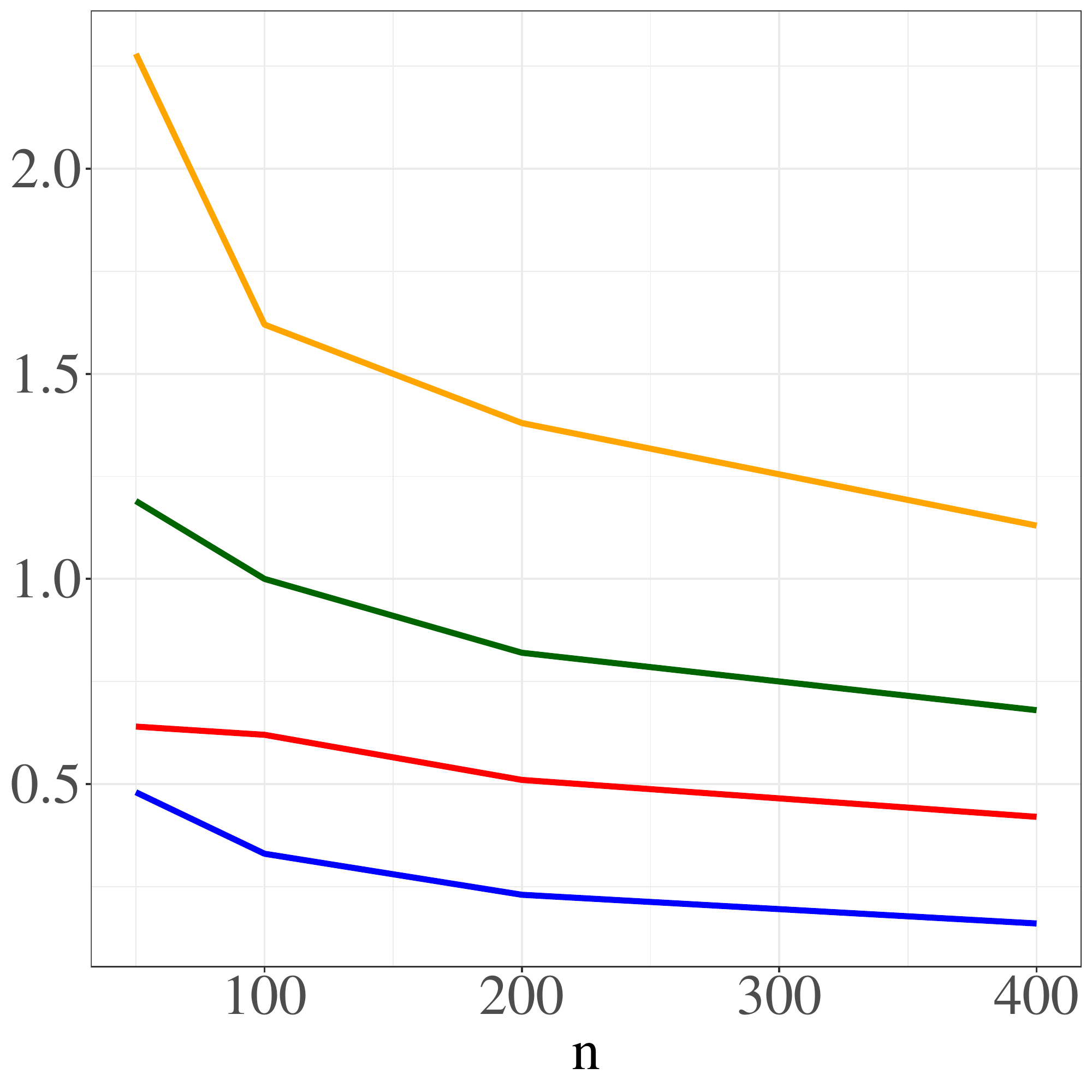}
\end{subfigure}

\vskip\baselineskip
\begin{subfigure}[b]{0.3\textwidth}
     \caption[]
        {{\footnotesize $\beta_{0i} = 1.5, \mu_{0} = - \log n$}}
 	\centering
	\includegraphics[width=1.1\textwidth, height = \textwidth]{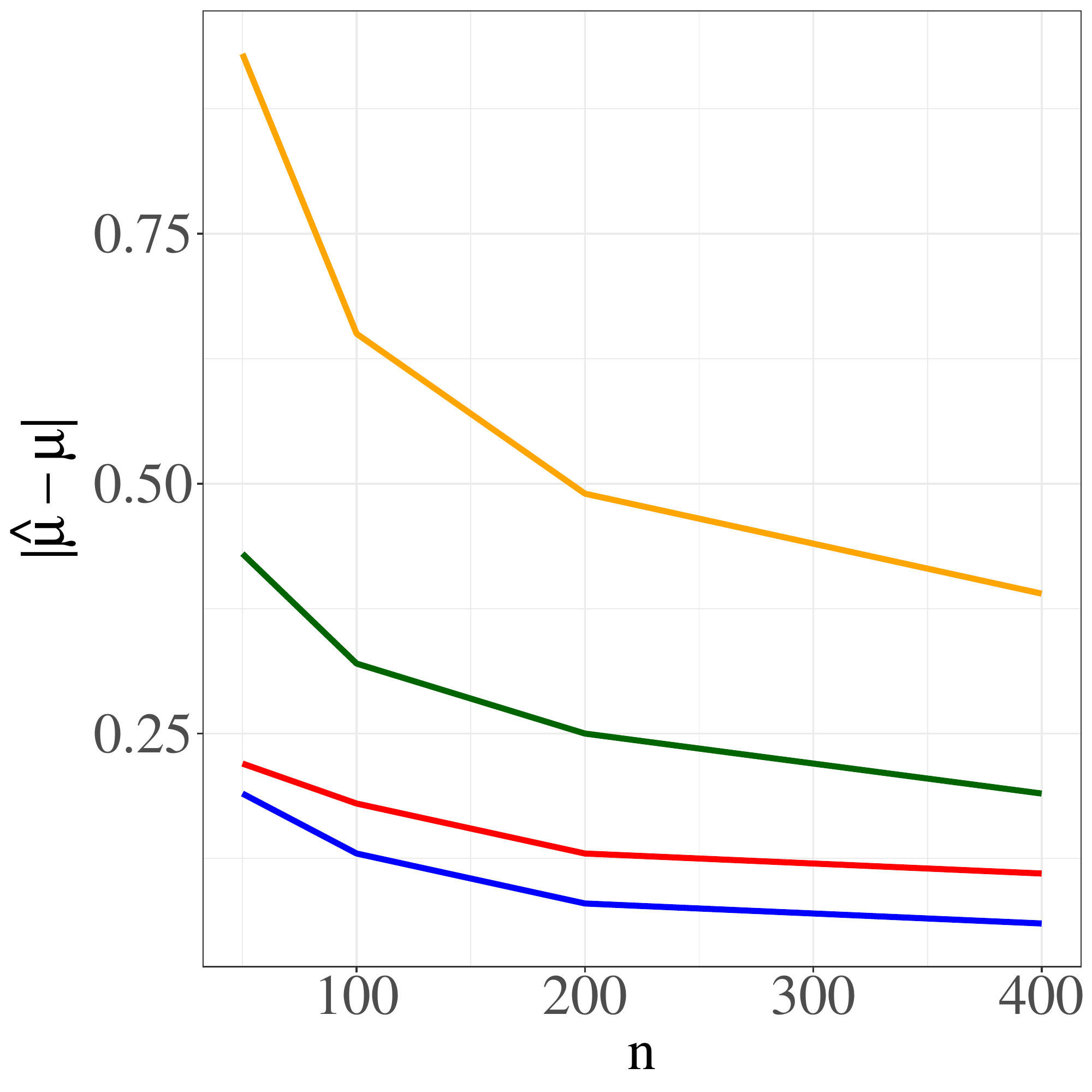}
\end{subfigure}
 \hfill \hfill \hfill  \hspace{-5mm}
\begin{subfigure}[b]{0.3\textwidth}
   \caption[]
        {{\footnotesize $\beta_{0i} = \sqrt{\log n}, \mu_{0} = - \log n$}}
 	\centering
	\includegraphics[width=\textwidth]{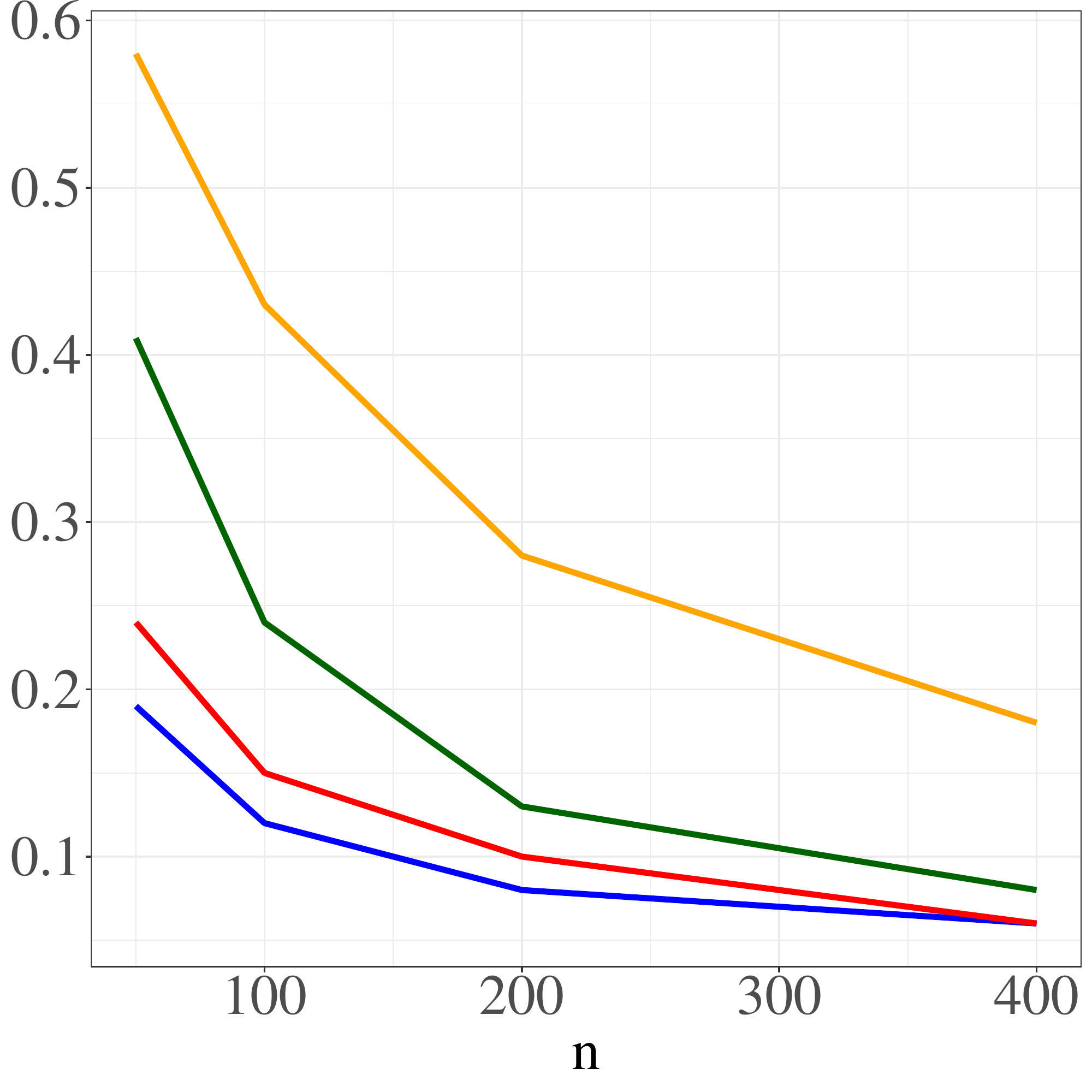}
\end{subfigure}
\hfill 
\begin{subfigure}[b]{0.3\textwidth}
    \caption[]
        {{\footnotesize $\beta_{0i} = \log n, \mu_{0} = - \log n$}}
 	\centering
	\includegraphics[width=\textwidth]{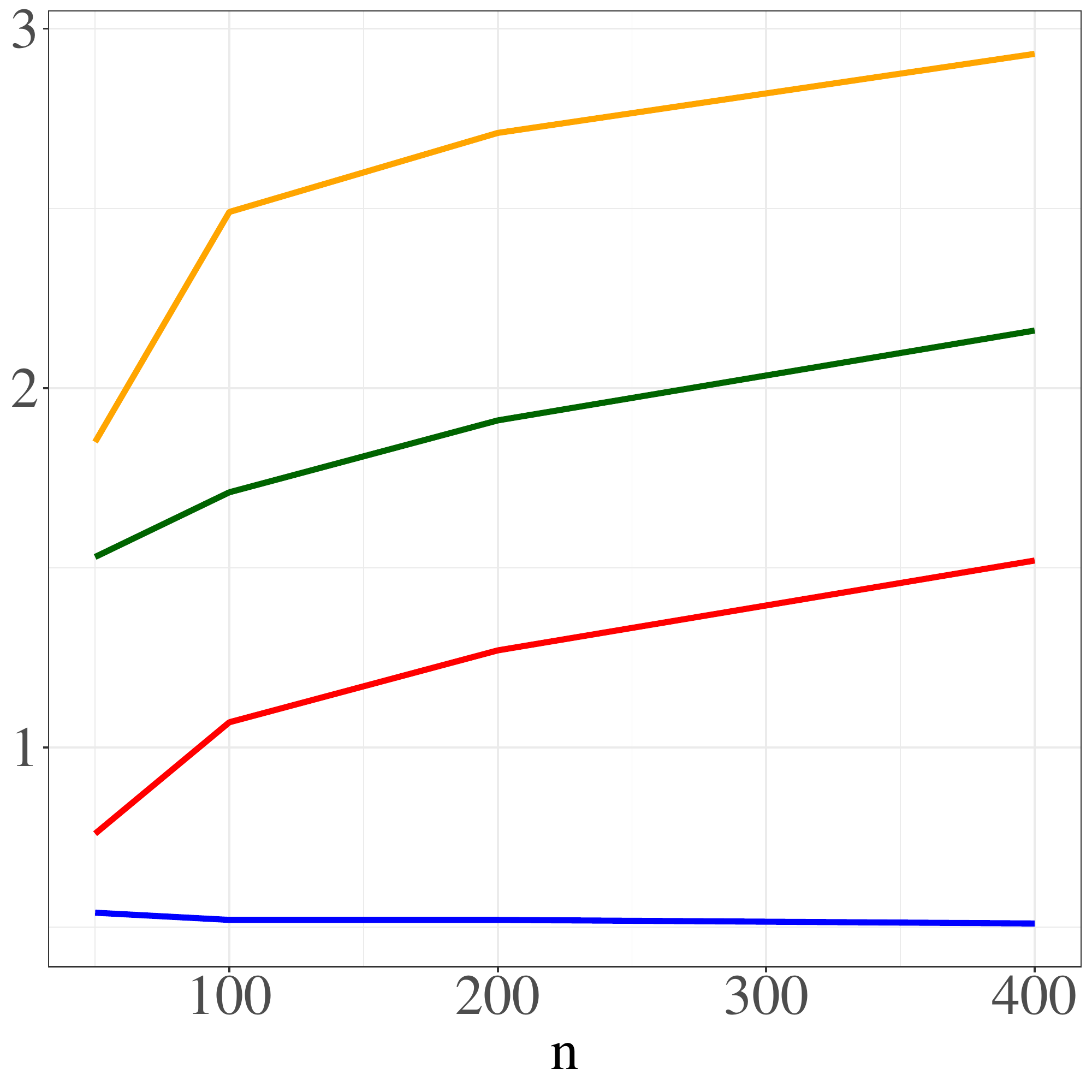}
\end{subfigure}

 \caption[ ]
      {\small Simulation results on $|\hat{\mu}(\hat{s}) - \mu_{0}|$ with $\hat{s}$ selected by BIC. \ \legendsquare{blue}~$s_0 = 2$,
  \legendsquare{red}~$s_0 = \lfloor \sqrt{n/2} \rfloor$,
  \legendsquare{darkgreen}~$s_0 = \left \lfloor \sqrt{n}\right \rfloor$, \legendsquare{orange}~$s_0 = \left \lfloor 2\sqrt{n}\right \rfloor$.  }
         \label{Figure-Norm-of-Mu}
\end{minipage}
\end{figure}

 \newpage
\section{Additional Data Analysis Results}\label{sec:addSimu}

 \begin{table}[hbt!] \centering 
 \begin{threeparttable}
  \caption{Effect of Different Network Statistics on Take-Up, Probit Link} 
  \label{Table: Degree or Beta Probit} 
\begin{tabular}{@{\extracolsep{\fill}}lcccccccc} 
\\[-1.8ex]\hline 
\hline \\[-1.8ex] 
 & \multicolumn{8}{c}{\textit{Dependent variable: take-up}} \\ 
\cline{2-9} 
 
\\[-1.8ex] & (1) & (2) & (3) & (4) & (5) & (6) & (7) & (8)\\ 
\hline \\[-1.8ex] 

 Degree & 0.006$^{***}$ &  &  &  & $-$0.0004 & $-$0.002 &  &  \\ 
  & (0.002) &  &  &  & (0.003) & (0.003) &  &  \\ 
  & & & & & & & & \\ 
 Eigenvector &  & 0.333$^{***}$ &  &  &  &  & 0.255$^{**}$ & 0.143 \\ 
  &  & (0.075) &  &  &  &  & (0.110) & (0.102) \\ 
  & & & & & & & & \\ 
 Beta &  &  & 0.114$^{***}$ &  & 0.119$^{**}$ &  & 0.042 &  \\ 
  &  &  & (0.030) &  & (0.047) &  & (0.043) &  \\ 
  & & & & & & & & \\ 
 Leader &  &  &  & 0.182$^{***}$ &  & 0.208$^{***}$ &  & 0.136$^{***}$ \\ 
  &  &  &  & (0.036) &  & (0.050) &  & (0.049) \\ 
  & & & & & & & & \\ 
  
\hline 
\hline \\[-1.8ex] 
 
\end{tabular} 
\begin{tablenotes}
      \small
       \item Note:  $\ensuremath{^{*}}\ensuremath{p<0.1};\ensuremath{^{**}}\ensuremath{p<0.05};\ensuremath{^{***}}\ensuremath{p<0.01}$

    \end{tablenotes}
  \end{threeparttable}
\end{table}

 \begin{table}[hbt!] \centering 
 \begin{threeparttable}
  \caption{Effect of Different Network Statistics on Take-Up, Identity Link (Linear Regression)} 
  \label{Table: Degree or Beta LM} 
\begin{tabular}{@{\extracolsep{\fill}}lcccccccc} 
\\[-1.8ex]\hline 
\hline \\[-1.8ex] 
 & \multicolumn{8}{c}{\textit{Dependent variable: take-up}} \\ 
\cline{2-9} 
 
\\[-1.8ex] & (1) & (2) & (3) & (4) & (5) & (6) & (7) & (8)\\ 
\hline \\[-1.8ex] 

 Degree & 0.001$^{***}$ &  &  &  & $-$0.0001 & $-$0.001 &  &  \\ 
  & (0.001) &  &  &  & (0.001) & (0.001) &  &  \\ 
  & & & & & & & & \\ 
 Eigenvector &  & 0.090$^{***}$ &  &  &  &  & 0.069$^{**}$ & 0.039 \\ 
  &  & (0.020) &  &  &  &  & (0.029) & (0.027) \\ 
  & & & & & & & & \\ 
 Beta &  &  & 0.031$^{***}$ &  & 0.032$^{**}$ &  & 0.011 &  \\ 
  &  &  & (0.008) &  & (0.013) &  & (0.012) &  \\ 
  & & & & & & & & \\ 
 Leader &  &  &  & 0.049$^{***}$ &  & 0.056$^{***}$ &  & 0.037$^{***}$ \\ 
  &  &  &  & (0.010) &  & (0.013) &  & (0.013) \\ 
  & & & & & & & & \\ 
  
\hline 
\hline \\[-1.8ex] 
 
\end{tabular} 
\begin{tablenotes}
      \small
       \item Note: $\ensuremath{^{*}}\ensuremath{p<0.1};\ensuremath{^{**}}\ensuremath{p<0.05};\ensuremath{^{***}}\ensuremath{p<0.01}$

    \end{tablenotes}
  \end{threeparttable}
\end{table}

\end{document}